\documentclass[11pt]{article}

\newlength\ys
 
\usepackage{amsmath,amssymb,dsfont,paralist,amsthm,eurosym}
\usepackage[format=default,labelfont=bf]{caption}
\usepackage{placeins}
\usepackage{microtype}
\usepackage[utf8]{inputenc}
\usepackage{scalerel}
\usepackage{stmaryrd}
\usepackage{bbold}
\usepackage{fancyhdr}

\usepackage[pagewise]{lineno}

\usepackage{algorithm}
\usepackage{algpseudocode}

\algnewcommand\algorithmicforeach{\textbf{for each}}
\algnewcommand\algorithmicvariables{\textbf{Variables:}}
\algnewcommand\Variables{\item[\algorithmicvariables]}
\algdef{S}[FOR]{ForEach}[1]{\algorithmicforeach\ #1\ \algorithmicdo}
\algrenewcommand\algorithmicrequire{\textbf{Input:}}
\algrenewcommand\algorithmicensure{\textbf{Output:}}
\algnewcommand\Fixedcomment[1]{\hfill\makebox[0.4\textwidth][l]{$\triangleright$ #1}}

\usepackage{graphicx}
\usepackage{tikz}
\usepackage{subcaption}
\usepackage{tikz-qtree}
\usepackage{rotating}

\usetikzlibrary{arrows,matrix}
\usetikzlibrary{matrix,calc}
\usetikzlibrary{shapes.geometric}
\usetikzlibrary{positioning}
\usetikzlibrary{calligraphy}
\usetikzlibrary{fit}

\pgfdeclarelayer{background}
\pgfsetlayers{background,main}

\usetikzlibrary{decorations.pathreplacing}
\tikzset{curlybrace/.style={decoration=brace,decorate}}
\usetikzlibrary{automata}
\usetikzlibrary{calc,positioning}
\usetikzlibrary{decorations.pathmorphing}
\usetikzlibrary{shapes}
\tikzset{trinode/.style={draw,triangle,minimum width=2.0cm}}
\tikzset{snake/.style={decorate,decoration=snake}}
\tikzset{curlybrace/.style={decoration=brace,decorate}}
\tikzset{triangle/.style={regular polygon,regular polygon sides=3}}
\tikzset{edge from parent path={(\tikzparentnode) -- (\tikzchildnode.north)}}
\usetikzlibrary{arrows.meta}

\usepackage{wrapfig}
\usepackage[rightcaption]{sidecap}
\usepackage[colorinlistoftodos]{todonotes}
\usepackage{mathtools}
\usepackage{enumitem}
\usepackage{multicol} 
\usepackage{todonotes}
\usepackage{imakeidx}         

\usepackage{hyperref}

\newtheorem{theorem}{Theorem}[section]
\newtheorem{lemma}[theorem]{Lemma}
\newtheorem{proposition}[theorem]{Proposition}
\newtheorem{corollary}[theorem]{Corollary}
\newtheorem{observation}[theorem]{Observation}
\newtheorem{example}[theorem]{Example}
\newtheorem{definition}[theorem]{Definition}

\newcommand{\hp}{\hat{+}}
\newcommand{\hti}{\hat{\times}}

\newcommand{\cR}{\mathcal R}

\newcommand{\rmM}{\mathrm{M}}
\newcommand{\rmB}{\mathrm{B}}
\newcommand{\C}{\mathrm{C}}

\newcommand{\T}{\mathrm{T}}

\newcommand{\sfM}{\mathsf{M}}
\newcommand{\sfA}{\mathsf{A}}
\newcommand{\sfB}{\mathsf{B}}

\newcommand{\A}{\mathsf{A}}
\newcommand{\B}{\mathsf{B}}

\newcommand{\M}{\mathsf{M}}

\newcommand{\sfT}{\mathsf{T}}

\newcommand{\0}{\mathbb{0}}
\newcommand{\1}{\mathbb{1}}

\newcommand{\pos}{\mathrm{pos}}

\newcommand{\size}{\mathrm{size}}

\newcommand{\rk}{\mathrm{ar}}

\newcommand{\wcl}{\mathrm{wcl}}

\newcommand{\rmPT}{{\mathrm{PT}}}

\newcommand{\rmPTid}{{\mathrm{PT}_{\mathrm{id}}}}

\newcommand{\rmST}{{\mathrm{ST}}}
\newcommand{\rmSTid}{{\mathrm{ST}_{\mathrm{id}}}}
\newcommand{\sfST}{\mathsf{ST}}

\newcommand{\rmAC}{\mathrm{AC}}

\newcommand{\XAC}{X/{\rmAC}}

\newcommand{\mrd}{\mathrm{rd}}
\newcommand{\mrid}{\mathrm{id}}
\newcommand{\midrd}{\mathrm{id,rd}}

\newcommand{\sfFS}{\mathsf{FS}}
\newcommand{\rmFS}{\mathrm{FS}}

\newcommand{\sfFB}{\mathsf{FB}}
\newcommand{\rmFB}{\mathrm{FB}}

\newcommand{\sfNsb}{\mathsf{N}_{\mathrm{sb}}}
\newcommand{\bbNsb}{\mathbb{N}_{\mathrm{sb}}}

\newcommand{\sfNrd}{\mathsf{N}_{\mathrm{rd}}}
\newcommand{\bbNrd}{\mathbb{N}_{\mathrm{rd}}}

\newcommand{\sfBidrd}{\mathsf{B}_{\mathrm{id,rd}}}
\newcommand{\bbBidrd}{\mathbb{B}_{\mathrm{id,rd}}}

\usepackage[most]{tcolorbox}
\newtheorem{theorem-rect}[theorem]{Theorem}
\newtheorem{corollary-rect}[theorem]{Corollary}
\newtheorem{lemma-rect}[theorem]{Lemma}
\newtheorem{construction-rect}[theorem]{Construction}
\newtheorem{definition-rect}[theorem]{Definition}

\tcolorboxenvironment{theorem-rect}{
		arc=3pt,				
		colback=white,			
		enlarge top by=2mm,		
		enlarge bottom by=-1mm,	
		grow to left by=3mm,	
		grow to right by=3mm,	
		boxrule=0.1mm,			
		left=1.7mm,				
		right=1.7mm				
		}						

\tcolorboxenvironment{corollary-rect}{
		arc=3pt,
		colback=white,
		enlarge top by=2mm,
		enlarge bottom by=-1mm,
		grow to left by=3mm,
		grow to right by=3mm,
		boxrule=0.1mm,
		left=1.7mm,
		right=1.7mm
              }

\tcolorboxenvironment{lemma-rect}{
		arc=3pt,
		colback=white,
		enlarge top by=2mm,
		enlarge bottom by=-1mm,
		grow to left by=3mm,
		grow to right by=3mm,
		boxrule=0.1mm,
		left=1.7mm,
		right=1.7mm
              }

\tcolorboxenvironment{construction-rect}{
		arc=3pt,
		colback=white,
		enlarge top by=2mm,
		enlarge bottom by=-1mm,
		grow to left by=3mm,
		grow to right by=3mm,
		boxrule=0.1mm,
		left=1.7mm,
		right=1.7mm
              }

\tcolorboxenvironment{definition-rect}{
		arc=3pt,
		colback=white,
		enlarge top by=2mm,
		enlarge bottom by=-1mm,
		grow to left by=3mm,
		grow to right by=3mm,
		boxrule=0.1mm,
		left=1.7mm,
		right=1.7mm
              }

\pgfdeclaredecoration{dashsoliddouble}{initial}{
  \state{initial}[width=\pgfdecoratedinputsegmentlength]{
    \pgfmathsetlengthmacro\lw{.7pt+.5\pgflinewidth}
    \begin{pgfscope}
      \pgfpathmoveto{\pgfpoint{0pt}{\lw}}%
      \pgfpathlineto{\pgfpoint{\pgfdecoratedinputsegmentlength}{\lw}}%
      \pgfmathtruncatemacro\dashnum{%
        round((\pgfdecoratedinputsegmentlength-3pt)/6pt)
      }
      \pgfmathsetmacro\dashscale{%
        \pgfdecoratedinputsegmentlength/(\dashnum*6pt + 3pt)
      }
      \pgfmathsetlengthmacro\dashunit{3pt*\dashscale}
      \pgfsetdash{{\dashunit}{\dashunit}}{0pt}
      \pgfusepath{stroke}
      \pgfsetdash{}{0pt}
      \pgfpathmoveto{\pgfpoint{0pt}{-\lw}}%
      \pgfpathlineto{\pgfpoint{\pgfdecoratedinputsegmentlength}{-\lw}}%
      \pgfusepath{stroke}
    \end{pgfscope}
  }
}

\tikzset{small circle/.style={circle, draw=black, inner sep=0pt,outer sep=0pt, minimum size=2.5pt}}

\title{\bf \LARGE Free polynomial strong bimonoids}

\date{\today\endgraf \normalsize
  \vspace{40mm}
  \hspace*{70mm}
  }

 \author{Manfred Droste\\Leipzig University, Germany \\ \texttt{\footnotesize droste@informatik.uni-leipzig.de}\and
    Zolt\'an F\"ul\"op\footnote{Project no TKP2021-NVA-09 has been implemented with the support provided by the Ministry of Culture and Innovation of Hungary from the National Research, Development and Innovation Fund, financed under the TKP2021-NVA funding scheme.}\\ University of Szeged, Hungary\\\texttt{\footnotesize fulop@inf.u-szeged.hu}}

\begin{document}
\maketitle



\vspace{-35mm}

\begin{abstract} Recently, in weighted automata theory the weight structure of strong bimonoids has found much interest; they form a generalization of semirings and are closely related to near-semirings studied in algebra.
Here, we define polynomials over a set $X$ of indeterminates as well as an addition and a multiplication. 
We show that with these operations, they form a right-distributive strong bimonoid, that this polynomial strong bimonoid is free over $X$ in the class of all right-distributive strong bimonoids and that it is both left- and right-cancellative. We show by purely algebraic reasoning that two arbitrary terms are equivalent modulo the laws of right-distributive strong bimonoids if and only if their representing polynomials are equivalent by 
the laws of only associativity and commutativity of addition and associativity of multiplication.
We give effective procedures for constructing the representing polynomials. As a consequence, we obtain that the  equivalence of arbitrary terms modulo the laws of right-distributive strong bimonoids can be decided in exponential time. Using term-rewriting methods, we show that each term can be reduced to a unique polynomial as normal form. We also derive corresponding results for the free idempotent right-distributive polynomial strong bimonoid over $X$.
We construct an idempotent strong bimonoid which is weakly locally finite but not locally finite and show an 
application of it in weighted automata theory.
\end{abstract}

\indent AC-reduction, idempotency, weighted automata

{\bf AMSC:} 16Y60, 16Y30,  08A40, 08B20, 68Q45, 03D15

\section{Introduction}\label{sect:introduction}

Polynomials are fundamental in many areas of Mathematics. The coefficients of polynomials are often taken from algebraic structures like fields, rings or semirings. With suitable definitions of addition and multiplication, the polynomials then also form rings or semirings, which is crucial for their importance. Already in the last century, in algebra near-fields, which can be viewed as fields satisfying only a one-sided distributivity law, were investigated \cite{dic05, zas36}. Subsequently, a theory of near-rings \cite{pil77} and near-semirings \cite{hooroo67} was developed. Basic examples for such structures are provided by additive monoids of functions with composition as multiplication operation; also, the usual multiplication of ordinal numbers is left- but not right-distributive. Near-rings also occur in operator theory and in mathematical physics \cite{sch97}. In computer science, one-sided distributivity has been investigated intensively in Structural Operational Semantics and for process algebras, cf. \cite{acecimingmou12} for a survey and unifying approach.
Unification problems in equational theories with only one-sided distributivity have been investigated already in \cite{tidarn87}; such asymmetric unification arises naturally in symbolic analysis of cryptographic protocols, see \cite{erbesckapliu13,marmeanar15}.
Recently, in theoretical computer science weighted automata with weights from strong bimonoids were investigated. 
Strong bimonoids form an extension of semirings obtained by not requiring the distributivity laws.
For weighted automata, right-distributivity of the weight structures is important for the coincidence of different kinds of behaviors of the automata.
For a survey on this area in tree automata theory, see \cite{fulvog24}.

Polynomials with coefficients from the semiring of natural numbers (resp., the ring of integers) over a set  $X$  of 
non-commuting variables, with the standard definitions of addition and multiplication, again form a semiring (resp., a ring), 
which, in fact, is isomorphic to the free semiring (resp., ring) of all terms over  $X$; moreover, each term can be represented (modulo the semiring laws) by a polynomial. 
Similar representations by 'polynomial terms' are essential for many fields of algebra, cf. \cite{gra68}.

It is a natural question, both intrinsically and motivated by the above, to consider the effect of missing left- or right-distributivity (or both).

In this paper, we aim to develop strong bimonoids of polynomials, in particular,
with the assumption of right-distributivity and possibly idempotence of addition.
For this, we will define polynomials as congruence classes of particular terms, with suitable definitions of addition and multiplication. 

Our main results are the following; here we concentrate on the right-distributive case. 
We provide two definitions of the multiplication, an inductive and a direct one, and we show that they yield the same result and that we obtain, in particular, a right-distributive strong bimonoid of polynomials. 

We then show that this right-distributive strong bimonoid of polynomials is both left- and right-cancellative.
Here, we only have to assume that the two polynomials occurring as right factors have the same size. 
This provides a crucial difference to the standard semiring of polynomials with natural numbers (or integers) as coefficients, where the analogous result fails. It shows that in our setting we have a basically different multiplication and equality, due to the absence of left-distributivity. Our proof of this general cancellation result proceeds by an analysis of the structure of graphs associated with the polynomials.  

Next, we extend the important result that the semiring of polynomials with coefficients in the natural numbers over a set  $X$  of non-commuting variables is isomorphic to the free semiring of terms over $X$ into our setting. We will show that our strong bimonoid of polynomials over  $X$  is isomorphic to the free right-distributive strong bimonoid of terms over  $X$.  Consequently, the polynomials can be seen as concrete representations of the arbitrary terms in the free right-distributive strong bimonoid.

Moreover, we show that two arbitrary terms are equivalent modulo the laws of right-distributive strong bimonoids if and only if their representing polynomials are equivalent by the laws of only associativity and commutativity of addition resp. associativity of multiplication. In the term rewriting literature, this is regarded as an AC-reduction result.
This is usually achieved by an involved analysis of critical pairs; here we obtain it by algebraic means. 
We give effective procedures for constructing the representing polynomials. As a consequence, we obtain that the mentioned equivalence of arbitrary terms can be decided in exponential time. Similar results are obtained for a notion of simple terms and the free semiring of all terms (here, with linear time complexity for the decidability part), and for suitably defined idempotency-reduced polynomials and the free idempotent right-distributive strong bimonoid.

Furthermore, using term-rewriting methods, we show that each term can be reduced to a unique polynomial term as normal form. 
As a consequence, for the free idempotent right-distributive strong bimonoid, we obtain such a reduction with uniqueness up to additive associativity and commutativity.

Our results have applications in the theory of weighted automata. 
Recall that a classical automaton either accepts or rejects a given input, so, for a given input, it produces a yes/no- or $1$/$0$-answer. Weighted automata assign to each possible input an element from a more general weight structure; this value may reflect, e.g., the number of possible executions or the minimal or maximal amount of resources needed for the execution of the given input. 
Here, weight structures of the literature include semirings, lattices and, more generally, strong bimonoids.
A central question investigated from the beginnings of weighted automata theory is 
whether weighted automata produce finitely or infinitely many such values.
It is easy to see that if the strong bimonoid  $\sfB$  is locally finite 
(i.e., each finitely generated strong subbimonoid is finite), 
then each weighted word automaton and each weighted tree automaton over  $\sfB$ produces only finitely many values 
(cf., e.g., \cite{fulvog24}).
Recently, it was shown that if the strong bimonoid $\sfB$  satisfies a weaker local finiteness condition,
then each weighted word automaton still produces only finitely many values  (cf. \cite{drostuvog10}).
However, assuming $\sfB$  is not locally finite, even if  $\sfB$ is right-distributive, there are weighted tree automata which produce infinitely many values, cf. \cite{drofultepvog24}. 
Hence the question arises whether there exist such right-distributive weakly locally finite strong bimonoids which are not locally finite. If yes, then we have an essential difference in the computation power of weighted tree automata vs. weighted word automata.
This was answered positively in \cite{drofultepvog24}.
As indicated above, calculation of maximal or minimal values are important in various quantitative settings; therefore we wish to incorporate idempotency into our strong bimonoid. 
Using our previous results, in Section \ref{sect:wlc-strong-bimonoids}, we sharpen the mentioned result of  \cite{drofultepvog24} by constructing an idempotent and right-distributive strong bimonoid which is weakly locally finite but not locally finite.


\section{Strong bimonoids and universal algebra background}
\label{sec:preliminaries}

\paragraph{Strong bimonoids.} In this paper, we will consider algebraic structures which comprise the class of semirings and which are defined as follows.

\sloppy A \emph{strong bimonoid} \cite{drostuvog10,cirdroignvog10,drovog10,drovog12} is an algebra $\B=(B,\oplus,\otimes,\0,\1)$ such that $(B,\oplus,\0)$ is a commutative monoid, $(B,\otimes,\1)$ is a monoid, and $\0$ is annihilating with respect to $\otimes$, i.e.,  $b\otimes \0 = \0 \otimes b = \0$  for all  $b \in B$. The operations $\oplus$ and $\otimes$ are called addition and multiplication, respectively. 

A strong bimonoid $\B=(B,\oplus,\otimes,\0,\1)$  is
\begin{compactitem}
  \item \emph{commutative} if $\otimes$ is commutative,
\item \emph{left-distributive} if  
$a \otimes (b \oplus c) = (a \otimes b) \oplus (a \otimes c)$  for all  $a,b,c \in B$,
\item \emph{right-distributive} if  
$(a \oplus b) \otimes c = (a \otimes c) \oplus (b \otimes c)$  for all  $a,b,c \in B$,
\item \emph{idempotent} if  $b \oplus b = b$  for all  $b \in B$.
  \end{compactitem}

A \emph{semiring} \cite{hebwei93,gol99} is a distributive strong bimonoid, i.e., a strong bimonoid which is left-distributive and right-distributive. Clearly, a commutative right-distributive strong bimonoid is also left-distributive and hence a semiring.

We give a few examples of semirings, of strong bimonoids without or with only right- (or left-) distributivity, respectively, and of research involving one-sided distributivity.

\begin{example} \label{ex:strong bimonoids} \rm
(a) Semirings include the Boolean semiring  $(\mathbb{B},\vee,\wedge,\0,\1)$ of truth values, the semiring  $(\mathbb{N},+,\times,0,1)$ of natural numbers, all rings and fields, and all distributive lattices  $(L,\vee,\wedge,0,1)$  with smallest element $0$ and greatest element $1$.

(b)  The plus-min strong bimonoid of extended natural numbers  $(\mathbb{N}_\infty,+,\min,0,\infty)$  where  $\mathbb{N}_\infty = \mathbb{N} \cup \{\infty\}$  and $+$  and  $\min$  are the usual addition and minimum operations, respectively, on natural numbers including  $\infty$.

(c)  The plus-plus strong bimonoid of natural numbers (cf. \rm \cite[Ex.~2.3]{drofultepvog24})
$(\mathbb{N_\0},\oplus,+,\0,0)$  with a new zero element  $\0 \notin \mathbb{N}$. The binary operations  $\oplus$  and  $+$, if restricted to  $\mathbb{N}$, are both the usual addition of natural numbers, and  $\0 \oplus x = x \oplus \0 = x$  and  $\0 + x = x + \0 = \0$  for each  $x \in \mathbb{N}_\0$. 

(d) All lattices  $(L,\vee,\wedge,0,1)$  with smallest element $0$ and greatest element $1$  are strong bimonoids.

(e) Near-fields, near-rings and near-semirings provide typical examples of right-distributive strong bimonoids, cf. \rm \cite{pil77, hooroo67, kri05}. Here, to obtain a right-distributive strong bimonoid, we have to add the natural requirements that the addition operation is commutative, 
the additively neutral element  $\0$ also acts like a multiplicative zero, and there is a unit element for multiplication.

(f) $\rm{Diff}^0(\mathbb{R}) = \{f: \mathbb{R} \rightarrow \mathbb{R} \mid f \; \text{is differentiable}, f(0) = 0\}$ with the usual addition and composition of functions forms a right-distributive strong bimonoid.

(g)  Let  $(\mathbb{N}[x]_0, \oplus, \circ, \0, \1)$  be the set of all polynomials over a variable  $x$  with coefficients (e.g.) in  $\mathbb{N}$  and with constant term equal to  $0$, where  $\oplus$  is the usual addition of polynomials, multiplication is given by composition, $\0$  is the zero polynomial and   $\1 = x$.
This strong bimonoid is right-distributive but not left-distributive.

(h) Multiplication of ordinal numbers is left-distributive but not right-distributive over addition
(since, e.g. $(1 + 1) \times \omega = \omega \neq 1 \times \omega + 1 \times \omega$). Hence, e.g., the set of all ordinal numbers strictly below $\omega^\omega$ with the usual addition and multiplication of ordinal numbers forms a left-distributive strong bimonoid.

(i) The strong bimonoid of words  $(\Sigma^* \cup \{\infty\}, \wedge, \cdot, \infty, \varepsilon)$  with a new element  $\infty \notin \Sigma^*$. For  $u, v \in \Sigma^*$, 
$u \wedge v$ is the greatest common prefix of  $u$  and  $v$, and
$u \cdot v = uv$  is the usual concatenation of  $u$  and  $v$.
Moreover,  $u \wedge \infty = \infty \wedge u = u$  and  $u \cdot \infty = \infty \cdot u = \infty$
for all  $u \in \Sigma^* \cup \{\infty\}$. This strong bimonoid has been investigated in \cite[section 3.6]{moh97}
for string-to-weight transducers in natural language processing. It is left-distributive, but, if  $\Sigma$  has at least two elements, it is not right-distributive and hence not a semiring.

(j)  A survey of and unifying approach for much research in Structural Operational Semantics and process algebras stressing right-distributivity (called 'left-distributivity' in the paper) is given in \cite{acecimingmou12}.

(k)  Unification problems in equational theories with only one-sided distributivity have been investigated already in \cite{tidarn87}.
Such asymmetric unification arises naturally in symbolic analysis of cryptographic protocols, see \cite{erbesckapliu13, marmeanar15}.

($\ell$)  Right-residuated lattices with right-distributivity (which give rise to examples of idempotent right-distributive strong bimonoids investigated here) may also be used to give the semantics of non-commutative substructural logics such as the non-commutative Lambek calculus, cf.
\cite[section 4]{galraf04}, \cite{velwan24}.

(m)  Nonlinear operators in functional analysis and mathematical physics, with addition and composition, also give rise to right-distributive strong bimonoids, cf. \cite[p.3842]{sch97}. \hfill$\Box$
\end{example}

For many further examples of strong bimonoids, we refer to \cite{drostuvog10,cirdroignvog10}  and in particular to \cite[Ex.~2.7.10, 1-11]{fulvog24}), highlighting the theory of weighted tree automata with weights in strong bimonoids.

Our goal is to define notions of polynomials (in non-commuting variables) with values in the natural numbers  which extend the standard notions of polynomials in non-commuting variables, and to obtain strong bimonoids of such polynomials over non-commuting variables. For this, we recall basic notions of universal algebra which we will employ subsequently.

\paragraph{General.} We denote by $\mathbb{N}$  the set of nonnegative integers and $\mathbb{N}_+=\mathbb{N}\setminus\{0\}$. For every $k,n \in \mathbb{N}$, we denote by $[k,n]$ the set $\{i\in \mathbb{N} \mid k \le i \le n\}$ and we abbreviate $[1,n]$ by $[n]$. Hence $[0] = \emptyset$.

Let $A$ be a set.  Then $A^*$ denotes the set of all finite strings over $A$, and $\varepsilon$ denotes the empty string. Any subset $\rho \subseteq A\times A$ is a \emph{(binary) relation on $A$}.
Let $\rho$ be a binary relation on $A$. As usual $\rho^n$ denotes the $n$ the power of $\rho$ for each $n\in \mathbb{N}$. Moreover, $\rho^+$ and $\rho^*$  denote the transitive closure, and the reflexive and transitive closure of $\rho$, respectively\footnote{It will always be clear from the context whether $\rho^*$ denotes the reflexive and  transitive closure of the relation $\rho$ or the set of all finite sequences over the set $\rho$.}

\paragraph{Signature and terms.}    A {\em signature} is a pair $(\Sigma,\rk)$, where $\Sigma$ is a non-empty set and $\rk: \Sigma \rightarrow \mathbb{N}$ is a mapping, called \emph{arity  mapping}. For each $k\in \mathbb{N}$, we put $\Sigma^{(k)}=\{\sigma\in \Sigma \mid \rk(\sigma)=k\}$. Usually, we abbreviate $(\Sigma,\rk)$ by $\Sigma$. If $\Sigma$ is finite, then we also call it \emph{ranked alphabet} (cf. e.g. \cite{fulvog24}).


Let $\Sigma$ be a signature and $X$ be a set disjoint with $\Sigma$.
The set of \emph{$\Sigma$-terms over $X$, } denoted by $\T_\Sigma(X)$, is the smallest set $T$ such that (i) $\Sigma^{(0)} \cup X \subseteq T$ and (ii) for every $k \in \mathbb{N}_+$, $\sigma \in \Sigma^{(k)}$, and $t_1,\ldots, t_k \in T$, we have $\sigma(t_1,\ldots,t_k) \in T$. We abbreviate $\T_\Sigma(\emptyset)$ by $\T_\Sigma$.

The  {\em size} of a term $t\in \T_\Sigma(X)$, denoted by $\size(t)$, is the total number of occurrences of elements from  $\Sigma \cup X$  in  $t$.

\paragraph{Universal algebra.}  We assume that the reader is familiar with basic concepts of universal algebra,
like subalgebra, subalgebra generated by a subset, homomorphism, congruence relation (for short: congruence), quotient algebra,  and free algebra with a generating set \cite{bursan81,wec92}, and \cite{baanip98}. 

Let $\Sigma$ be a signature. We regard the elements of  $\Sigma$ as operation symbols. 
An \emph{algebra $\sf A$ of type $\Sigma$} ($\Sigma$-algebra) is a pair $\sfA=(A,\theta)$ where $A$ is a nonempty set and $\theta$ is a mapping from $\Sigma$ to the family of finitary operations on $A$ such that for every $k\in \mathbb{N}$ and  $\sigma \in \Sigma^{(k)}$, 
the operation $\theta(\sigma)$ has arity $k$. 
As usual, nullary operations are interpreted as constants.

Let $X$ be a set. The \emph{$\Sigma$-term algebra over $X$}, denoted by $\sfT_\Sigma(X)$,  is the $\Sigma$-algebra
\(\sfT_\Sigma(X) = (\T_\Sigma(X),\theta_\Sigma)\) where, for every $k \in \mathbb{N}$, $\sigma \in \Sigma^{(k)}$, and $t_1,\ldots, t_k \in \T_\Sigma(X)$, we let $\theta_\Sigma(\sigma)(t_1,\ldots,t_k) = \sigma(t_1,\ldots,t_k)$. The \emph{$\Sigma$-term algebra}, denoted by $\sfT_\Sigma$, is the $\Sigma$-term algebra over $\emptyset$, i.e., $\sfT_\Sigma = \sfT_\Sigma(\emptyset)$.

  \label{page:initial-algebra}
Let $\mathcal{K}$ be an arbitrary class of $\Sigma$-algebras. Moreover, let $\mathsf{F}=(F,\eta)$ be a $\Sigma$-algebra in $\mathcal{K}$ and $X \subseteq F$  such that $\mathsf{F}$ is generated by $X$. 
The algebra~$\mathsf{F}$ is called \emph{freely generated in~$\mathcal{K}$ by $X$} if, for every $\Sigma$-algebra $\A=(A,\theta)$ in $\mathcal{K}$ and mapping $f\colon X \rightarrow A$, there exists an extension of $f$ to a $\Sigma$-algebra homomorphism $h\colon F \rightarrow A$ from $\mathsf{F}$ to $\A$. 
If the extension $h$ exists, then it is unique since $X$ generates $\mathsf{F}$, cf. e.g. \cite[Lm.~3.3.1]{baanip98}. A $\Sigma$-algebra  $\mathsf{F}$  is  \emph{free in $\mathcal{K}$}, if it is freely generated in $\mathcal{K}$ by some subset  $X \subseteq F$.

\begin{lemma} \label{lm:bijective-generating-sets} \rm\cite[Thm.~3.3.3]{baanip98} Let  $\A$ and $\B$ be free $\Sigma$-algebras in~$\mathcal{K}$, freely generated by the sets $X$ and  $Y$,  respectively.
If $|X| = |Y|$, then  $\A$  and  $\B$  are isomorphic.\hfill$\Box$
\end{lemma}

We will use the following well-known result without any reference.

\begin{theorem}\label{thm-free-algebra-theorem}  (cf. \cite[Thm.~3.4.2]{baanip98}, \cite[Thm.~II.10.8]{bursan81}, \cite[p.~18, Thm.~4]{wec92}) \ $\sfT_\Sigma(X)$ is a free $\Sigma$-algebra in the class of all $\Sigma$-algebras, freely generated by the set $X$.\hfill$\Box$
\end{theorem}

In the rest of this section, $\sfA$ denotes an arbitrary $\Sigma$-algebra $(A,\theta)$.

Let  $\rho$ be a congruence on $\sfA$. 
For each $a\in A$, we let $[a]_\rho = \{b\in A \mid a\rho b\}$,
the \emph{congruence class of $a$ (modulo $\rho$)}. 
For each $B\subseteq A$, we put $B/\!\rho =\{[a]_\rho \mid a\in B\}$.
The \emph{quotient algebra of $\sfA$ by $\rho$} is the $\Sigma$-algebra $\sfA/\!\rho = (A/\!\rho, \theta/\!\rho)$ where, for every $k \in \mathbb{N}$,  $\sigma \in \Sigma^{(k)}$, and $a_1,\ldots,a_k \in A$, we have $(\theta/\!\rho)(\sigma)([a_1]_{\rho},\ldots,[a_k]_{\rho}) =  [\theta(\sigma)(a_1,\ldots, a_k)]_{\rho}$.

Next we wish to consider $\Sigma$-identities and the congruence on $\sfA$ induced by a set of such identities. For this, we introduce the necessary concepts. 

Let $Z=\{z_1,z_2,\ldots\}$ be a set, we call the elements of  $Z$ \emph{variables}. For each $n\in \mathbb{N}$, we put $Z_n=\{z_1,\ldots,z_n\}$.

An \emph{assignment} is a mapping $\varphi: Z \to A$. In the sequel, its unique extension to a $\Sigma$-algebra homomorphism from  $\sfT_\Sigma(Z)$ to $\sfA$ will be denoted also by $\varphi$.
For an arbitrary $t \in \T_\Sigma(Z)$, we call $\varphi(t)$ the {\em evaluation of $t$ in $ \sfA$ by $\varphi$}.

A \emph{$\Sigma$-identity over $Z$} (or: identity) is a pair $(\ell, r)$ where $\ell,r \in \T_\Sigma(Z)$.  The $\Sigma$-algebra $\sfA$ \emph{satisfies the identity $(\ell, r)$} if, for every assignment $\varphi: Z \to A$, we have
$\varphi(\ell) = \varphi(r)$.

\begin{lemma}\rm \label{identitiestofactoralgebras}\cite[Th.~II.6.10 and Lm.~II.11.3]{bursan81}
If $\sfA$ satisfies an identity $(\ell, r)$ and $\rho$ is a congruence on $\sfA$, then $\sfA/\!\rho$ also satisfies the identity $(\ell, r)$. \hfill$\Box$
\end{lemma}

Let $E$ be a set of identities. The \emph{congruence (relation)  on $\sfA$ induced by $E$}, denoted by $=_E$, is the smallest congruence on $\sfA$ which contains the set
\begin{equation}\label{eq:identity-in-term-algebra}
E(\sfA)=\{(\varphi(\ell), \varphi(r))\mid (\ell, r)\in E, \varphi: Z \to A\}.
\end{equation}

The following lemma is well-known and can be proven similarly to \cite[p.~176, Lm.~24]{wec92}.

\begin{lemma}\label{lm:free-algebra-quotient} \rm Let  $E$ be a set of $\Sigma$-identities. Then $\sfA/\!{=_E}$ satisfies all identities in $E$. \hfill$\Box$
\end{lemma}

Next we extend the well-known syntactic characterization of the congruence on $\sfT_\Sigma(Z)$ induced by a set $E \subseteq \T_\Sigma(Z) \times \T_\Sigma(Z)$ of $\Sigma$-identities, 
cf. \cite[Thm.~3.1.12]{baanip98} and \cite[Thms.~II.14.17, II.14.19]{bursan81},
to a characterization of the congruence on $\sfA$ induced by  $E$. In fact, this is closely related to a general description of a congruence generated by a binary relation on  $\A$,
cf. \cite[Sect.~2.1.2]{wec92}. 

Each element $c\in \T_\Sigma(Z)$ in which the variable $z_1$ occurs exactly once  is called a {\em $\Sigma Z$-context}. We let $\C_{\Sigma,Z}$  be the set of all $\Sigma Z$-contexts.
Given a $\Sigma Z$-context  $c$  and a term  $t \in \T_\Sigma(Z)$, we let
$c[t]\in \T_\Sigma(Z)$  be the term obtained from  $c$  by replacing the variable  $z_1$  by  $t$.
That is, if  $\varphi: Z \to \T_\Sigma(Z)$  is given by  $\varphi(z_1) = t$
and  $\varphi(z_i) = z_i$  for each  $i \geq 2$, then  $c[t] = \varphi(c)$.

Let $E$ be a set of $\Sigma$-identities. The \emph{reduction relation induced by $E$ on $A$}, denoted  by  $\Rightarrow_E$, is the binary relation on $A$ defined as follows:
for every $a,b \in A$, we let $a \Rightarrow_E b$ if there exist a $\Sigma Z$-context $c\in \C_{\Sigma,Z}$, an identity $(\ell, r)$ in $E$,  and an assignment $\varphi: Z \to A$ such that
$a = \varphi(c[\ell])$ and $b = \varphi(c[r])$.
In this  case  we say that $b$ is obtained from $a$  in a reduction step (using the identity $(\ell, r)$).

For an identity $e=(\ell,r)$ we define $e^{-1}=(r,\ell)$ and we let $E^{-1}=\{ e^{-1} \mid e\in E\}$.
Moreover, we abbreviate $\Rightarrow_{E\cup E^{-1}}$ by $\Leftrightarrow_E$.

The subsequent characterization says that, for any two elements $a,b\in A$,
we have $a =_E b$ if and only if, there is a finite sequence of elements $a=a_0,a_1,\ldots,a_n=b$ of $A$ for some $n\in \mathbb{N}$ such that for each $i\in [n]$, the element $a_i$ can be obtained from $a_{i-1}$ in a reduction step using an identity in $E$  or the inverse of an identity. For a proof we can follow the proof of \cite[Lm.~2.3]{drofultepvog24}, stated for the case that $\Sigma$ is finite. 
However, in that proof we do not use the finiteness of $\Sigma$.

\begin{lemma}\rm \label{lm:approx-characterization}\  \cite[p.~98, Thm.~6]{wec92} Let  $E$ be a set of $\Sigma$-identities and $=_E$ the congruence on $\sfA$ induced by $E$. Then $=_E \, =\, \Leftrightarrow^*_E$. \hfill$\Box$
\end{lemma}

\section{Free strong bimonoids}\label{sect:free-strong-bimonoids}

In this section, we will shortly consider the free strong bimonoid, the free right-distributive strong bimonoid, and the free idempotent right-distributive strong bimonoid, 
each freely generated by a nonempty set $X$ of variables.

For this we consider the particular signature 
\begin{align*}
\Sigma_{\mathrm{sb}} = \{\hat{+}, \hat{\times}, \hat{0}, \hat{1}\} \text{ with } \rk(\hat{+})=\rk(\hat{\times})=2 \text{ and } \rk(\hat{0})=\rk(\hat{1})=0.
\end{align*}
(We write $\;\hat{}\;$ over the symbols of $\Sigma$ because later the symbols $+, \times, 0$, and $1$ will denote addition, multiplication, zero and one of polynomials in certain algebras, respectively.) We call $\Sigma_{\mathrm{sb}}$ the \emph{strong bimonoid signature}. 

Let $X$ be a nonempty set, which we keep fixed throughout the paper. 

We consider the 
$\Sigma_{\mathrm{sb}}$-term algebra  
$\sfT_{\Sigma_{\mathrm{sb}}}(X)$ over $X$.
Since  we use the signature $\Sigma_{\mathrm{sb}}$ in almost all of the rest of the paper, we drop it from the notation, i.e
\begin{quote}
we abbreviate  \ $\sfT_{\Sigma_{\mathrm{sb}}}(X)=(\T_{\Sigma_{\mathrm{sb}}}(X),\theta_{\Sigma_{\mathrm{sb}}})$ \ by \
$\sfT(X)=(\T(X),\theta).$
\end{quote}

We write elements of $\T(X)$ in infix form, e.g. we write $(\hat{1}\hat{+}\hat{1})\hat{\times} x$ for $\hat{\times}(\hat{+}(\hat{1},\hat{1}),x)$ where $x\in X$.
Moreover, we denote the binary operations $\theta(\hat{+})$ and  $\theta(\hat{\times})$ on  $\T(X)$  by  $+'$  and  $\times'$, respectively, and we write them in infix form. That is, for every  $t_1, t_2 \in \T(X)$,  we have
\[t_1 +' t_2 = \theta(\hat{+})(t_1, t_2) = t_1 \hat{+} t_2 \ \text{ and } \ 
t_1 \times' t_2 = \theta(\hat{\times})(t_1, t_2) = t_1\hat{\times} t_2 \enspace.\]
Also, we have the constants  $0' = \theta(\hat{0}) = \hat{0}$ and  $1' = \theta(\hat{1}) = \hat{1}$. Hence, we may write $\sfT(X)$ in the form
\[\sfT(X) = (\T(X), +', \times', 0',1').\]

Subsequently, we will consider the following identities:

\begin{center}
\begin{tabular}{ll}
$e_1: \big(z_1 \hat{+} (z_2 \hat{+} z_3) \ , \ (z_1 \hat{+} z_2) \hat{+} z_3\big)$ & \\[2mm]
$e_2: \big(z_1 \hat{+} z_2 \ , \ z_2 \hat{+} z_1\big)$ &  $e_3: \big(z \hat{+} \hat{0} \ , \  z\big)$ \\[2mm]
$e_4: \big(z_1 \hat{\times} (z_2 \hat{\times} z_3)  \ , \  (z_1 \hat{\times} z_2) \hat{\times} z_3\big)$   & \\[2mm]
$e_5: \big(\hat{1}\hat{\times} z \ , \ z\big)$ & $e_6:  \big(z \hat{\times} \hat{1} \ , \ z\big)$ \\[2mm]
$e_7: \big(z \hat{\times} \hat{0}  \ , \  \hat{0}\big) $ & $e_8: \big(\hat{0}\hat{\times} z \ , \ \hat{0}\big)$  \\[2mm]
$e_9: \big((z_1 \hat{+} z_2) \hat{\times} z_3 \ , \ (z_1 \hat{\times} z_3) \hat{+} (z_2 \hat{\times} z_3)\big)$\\[2mm]
 $e_{10} : \big( z_1 \hat{\times} (z_2 \hat{+} z_3) \ , \ (z_1 \hat{\times} z_2) \hat{+} (z_1 \hat{\times} z_3)\big)$\\[2mm]
 $e_{11} : \big((z \hat{+} z) \ , \ z\big)$
\end{tabular}
\end{center}

First, for motivation of the subsequent development in this section, we recall well-known fundamental facts on the free semiring and on polynomials in the present notation.

Let  $E_S = \{e_1,...,e_{10}\}$ (the set of ``semiring axioms'') , and let $=_{E_S}$ denote the congruence relation on $\sfT(X)$ generated by $E_S$. Thus, the quotient algebra of $\sfT(X)$ by $=_{E_S}$ is the algebra
\[\sfT(X)/\!{=_{E_S}} = (\T(X)/\!{=_{E_S}}, +'/\!{=_{E_S}}, \times'/\!{=_{E_S}}, [0']_{=_{E_S}},[1']_{=_{E_S}})\enspace.\]
We abbreviate $\sfT(X)/\!{=_{E_S}}$ by $\mathsf{FS}(X)$ and abbreviate also the components of $\sfT(X)/\!{=_{E_S}}$ (i.e. of $\mathsf{FS}(X)$) and write the quotient algebra in the form
\[\mathsf{FS}(X)=(\mathrm{FS}(X), \oplus, \otimes, \0, \1). \]
Moreover, for $t \in \T(X)$, we abbreviate the notation  $[t]_{=_{E_S}}$ by~$[t]_{E_S}$, and we abbreviate $X/\!{=_{E_S}}$ by $X/\!{E_S}$.

Then, e.g., for every  $t_1, t_2 \in \T(X)$,  we have
  \begin{align*}
    &[t_1]_{E_S} \oplus [t_2]_{E_S} = [t_1 +' t_2]_{E_S}  \ \text{ and } \  [t_1]_{E_S} \otimes [t_2]_{E_S} = [t_1 \times' t_2]_{E_S}\\
    &[t_1]_{E_S} \oplus \0 = [t_1]_{E_S}   \ \text{ and } [t_1]_{E_S} \otimes \0 =  \0\\
    &[t_1]_{E_S} \otimes \1 = \1 \otimes [t_1]_{E_S} = [t_1]_{E_S}
    \enspace.
  \end{align*}

By Lemma \ref{lm:free-algebra-quotient} it follows that the algebra
$\sfFS(X)$ satisfies all identities in $E_S$.  In particular, 
      \begin{compactitem}
      \item identities $e_1 - e_{3}$ assure that $(\rmFS(X),\oplus,\0)$ is a commutative monoid,
      \item identities $e_4 - e_{6}$ assure that $(\rmFS(X),\otimes,\1)$ is a monoid,
      \item identities $e_7 - e_{8}$ assure that $\0$ is annihilating with respect to $\otimes$, and
      \item identities $e_9 - e_{10}$ assure that $\otimes$ distributes  over $\oplus$.
      \end{compactitem}
      Hence $\sfFS(X)$ is a  semiring.  In fact, it is well-known that  $\mathsf{FS}(X)$  is free with generating set $X/\!{E_S}$ in the class of all semirings. Therefore it is called \emph{the free semiring over} $X$.

Expressions of the form $\sum_{w \in F} n_w w$  where $F$  is a finite subset of  $X^*$  and  $n_w \in \mathbb{N}$  for each  $w \in F$ are called \emph{polynomials} over $X$. We consider polynomials as terms in $\T(X)$ in the natural way. For instance, the polynomial $2xy +2$ corresponds to the term $\big(((\hat{1} \hat{+} \hat{1}) \hat{\times} x) \hat{\times}y\big)\hat{+}(\hat{1}\hat{+} \hat{1})$. Let $\mathbb{N}[X]$ be the set of all these polynomials. Then $\mathsf{N}[X]=(\mathbb{N}[X], +, \cdot, 0, 1)$, with the usual addition and multiplication of polynomials, is a semiring. Moreover, it is free with generating set $X$ in the class of all semirings, cf., e.g., \cite[Sec.~3.2.2, Ex.~3]{wec92}. Then the following result is folklore.

\begin{theorem}\label{thm:free-semiring}  The free semiring $\mathsf{FS}(X)$ over $X$ is isomorphic to the semiring $\mathsf{N}[X]$ of polynomials over $X$. Moreover, the polynomials form a representation of the semiring terms over $X$ in the sense that for each term $t \in \T(X)$ there is a polynomial $p \in \mathbb{N}[X]$ which is semiring-equivalent to $t$, i.e.,  $t =_{E_S} p$.\hfill$\Box$
\end{theorem}

The goal of this paper is to derive analogous results for strong bimonoids, right-distributive strong bimonoids, and idempotent right-distributive strong bimonoids. 
In the subsequent notation, let rd and id,  indicate  right-distributivity and idempotence, respectively.

\begin{definition}\rm\label{def:rd-quotient-id-quotient} (a) Let $E = \{e_1,...,e_8\}$ (``strong bimonoid axioms'') and let $\sfFB(X) = (\rmFB(X),\oplus,\otimes,\0,\1)$ be the quotient algebra of $\sfT(X)$ by $=_E$.

(b) Let $E_{\mrd} = E\cup \{e_9\}$ (axiom of right-distributivity added), 
and let $\sfFB_\mrd(X)=(\rmFB_\mrd(X),\oplus_\mrd,\otimes_\mrd,\0,\1)$
be the quotient algebra of $\sfT(X)$ by $=_{E_{\mrd}}$.

(c) Let $E_{\midrd} = E_\mrd\cup \{e_{11}\}$ (axiom of idempotency added),
and let $\sfFB_\midrd(X)=(\rmFB_\midrd(X),\oplus_\midrd,\otimes_\midrd,\0,\1)$
be the quotient algebra of $\sfT(X)$ by $=_{E_{\midrd}}$.\hfill $\Box$
\end{definition}

For easier subsequent use, the definition is summarized by the following table. 

\

 \begin{tabular}{l|l|l}
set of identities & quotient algebra & short notation \\\hline
& & \\
 $E = \{e_1,...,e_8\}$ & $\sfT(X)/\!{=_E}$ & $\sfFB(X) = (\rmFB(X),\oplus,\otimes,\0,\1)$ \\
 \hline
 & & \\
 $E_{\mrd} = E\cup \{e_9\}$ &  $\sfT(X)/\!{=_{E_\mrd}}$ &  $\sfFB_\mrd(X)=(\rmFB_\mrd(X),\oplus_\mrd,\otimes_\mrd,\0,\1)$ \\
 \hline 
 & & \\
 $E_{\midrd} = E_\mrd\cup \{e_{11}\}$ & $\sfT(X)/\!{=_{E_\midrd}}$ & 
 $\sfFB_\midrd(X)=(\rmFB_\midrd(X),\oplus_\midrd,\otimes_\midrd,\0,\1)$
 \end{tabular}

\


By Lemma  \ref{lm:free-algebra-quotient}, $\sfFB(X)$, $\sfFB_{\mrd}(X)$, and $\sfFB_\midrd(X)$  satisfy all identities in $E$, $E_\mrd$, and $E_{\midrd}$, respectively. Hence  $\sfFB(X)$ is a strong bimonoid, $\sfFB_\mrd(X)$  is a right-distributive strong bimonoid and
$\sfFB_{\midrd}(X)$  is an idempotent right-distributive strong bimonoid.

Given  $t \in \T(X)$, we abbreviate the congruence class  $[t]_{=_E}$ in  
$\rmFB(X)$ by $[t]_E$. Similarly, we denote  $[t]_{=_{E_{\mrd}}}$ and $[t]_{=_{E_\midrd}}$ by $[t]_{\mrd}$ and  $[t]_{\midrd}$, respectively.

Moreover, let  $X/\!{E} = \{ [x]_{E} \mid x \in X \}$, a generating set for $\sfFB(X)$,
$X/{\mrd} = \{ [x]_{\mrd} \mid x \in X \}$, a generating set for $\sfFB_{\mrd}(X)$,
and  $X/{\midrd} = \{ [x]_{\midrd} \mid x \in X \}$, a generating set for $\sfFB_{\mrd}(X)$.

The following result is immediate by \cite[Sec.~3.2.4, Cor.~2]{wec92} or by \cite[Cor. 3.5.8, Thm. 3.5.14]{baanip98} combined with Lemma \ref{lm:approx-characterization}.

\begin{proposition}\label{prop:three-free-sb}\rm$\,$
\begin{compactitem}
\item[(a)]  $\sfFB(X)$  is a free strong bimonoid, freely generated by the set  $X/\!{E}$.
\item[(b)]    $\sfFB_{\mrd}(X)$  is a free right-distributive strong bimonoid, freely generated by the set  $X/{\mrd}$.
\item[(c)]    $\sfFB_{\midrd}(X)$  is a free idempotent right-distributive strong bimonoids, freely generated
by the set  $X/{\midrd}$. \hfill$\Box$
\end{compactitem}
\end{proposition}


\section{Strong bimonoids of polynomials}\label{sect:strong-bimonoids-polynomials}

The goal of this section is to develop our new concept of polynomials together with their operations of addition and multiplication. In particular, the aim is that for our polynomials,
the multiplication is right-distributive over addition, but otherwise, ``as free as possible" (in particular, in general not commutative or left-distributive).
For this, we define congruence classes of particular terms, a class of polynomials, and a class of idempotency-reduced polynomials. For each of the three classes,  we define an addition and a multiplication, and we show that they form strong bimonoids, right-distributive strong bimonoids, and idempotent right-distributive strong bimonoids, respectively. We also show that the multiplication in the right-distributive strong bimonoids is cancellative both from the right and the left. Then in the next section we will show that our strong bimonoids constitute the free objects in the classes of strong bimonoids, right-distributive strong bimonoids, respectively idempotent right-distributive strong bimonoids.

We start with the definition of particular terms, called simple terms, and the $\Sigma_{\mathrm{sb}}$-algebra $\sfST(X)$ (cf. \cite[Sect.~3]{drofultepvog24}). 

We call a term  $s  \in \T(X)$  \emph{simple}, if  
\begin{compactitem}
\item $s = \hat{0}$  or
\item $s\ne \hat{0}$  and it contains neither   $\hat{0}$, nor a subterm of the form  $\hat{1} \hat{\times} s$,  nor a subterm of the form $s \hat{\times} \hat{1}$.
\end{compactitem}

Let $\rmST(X)$ denote the set of all simple terms in $\T(X)$. Note that $\hat{1}$ is simple, hence
e.g., $\hat{1}\hat{+}s\in\rmST(X) $ for each $s\in \rmST(X)$.

Next we define the $\Sigma_{\mathrm{sb}}$-algebra $\sfST(X)=(\rmST(X),+_\mathsf{ST},\times_\mathsf{ST},\hat{0},\hat{1})$ as follows:
\begin{compactitem}
\item for each $s\in \rmST(X)$, let $s+_\mathsf{ST} \hat{0}= \hat{0} +_\mathsf{ST} s=s$ and   $s \times_\mathsf{ST} \hat{0}= \hat{0} \times_\mathsf{ST} s=\hat{0}$,
\item for each $s\in \rmST(X)$, let $s\times_\mathsf{ST} \hat{1}= \hat{1} \times_\mathsf{ST} s=s$,
\item for every $s,t\in \rmST(X)\setminus\{\hat{0},\hat{1}\}$, let $s +_\mathsf{ST} t = s\hat{+}t$ and $s \times_\mathsf{ST} t = s\hat{\times} t$.
\end{compactitem}

Thus, on  $\rmST(X)\setminus\{\hat{0},\hat{1}\}$, the operations  $+_\mathsf{ST}$  and  $\times_\mathsf{ST}$ are the restrictions of  $\hat{+}$  respectively $\hat{\times}$,
but the rules for  $\hat{0}$  and  $\hat{1}$ are simplified to satisfy the usual algebraic laws.

Then we consider the set $\rmAC= \{e_1,e_2,e_4\}$ of identities and the quotient algebra 
\[\sfST(X)/\!{=_\rmAC}=\big(\rmST(X)/\!{=_\rmAC},+_\mathsf{ST}/\!{=_\rmAC},\times_\mathsf{ST}/\!{=_\rmAC},[\hat{0}]_{=_\rmAC},[\hat{1}]_{=_\rmAC}\big).\]
Lastly, we abbreviate the latter notation by $\sfNsb[X]=(\bbNsb[X],+,\times,0,1)$.

The following is immediate by Lemmas \ref{identitiestofactoralgebras} and \ref{lm:free-algebra-quotient}.

\begin{proposition}\rm \label{prop:SB-strong-bimonoid} The algebra $\sfNsb[X]=(\bbNsb[X],+,\times,0,1)$ is a strong bimonoid. \hfill$\Box$
\end{proposition}

To investigate the structure of the elements of the strong bimonoid $\sfNsb[X]$, we need to describe structural properties of their representing simple terms. For this, next we consider sum terms and product terms, defined as follows. 

Let $s\in \rmST(X)$. We say that $s$ is a \emph{sum term}  if there are  $s_1, s_2 \in  \rmST(X) \setminus \{ \hat{0} \} $ such that
$s=s_1 \hat{+} s_2$ and, we say that $s$ is a \emph{product term}  if there are
$s_1, s_2 \in  \rmST(X) \setminus \{ \hat{0}, \hat{1} \} $ such that $s=s_1 \hat{\times} s_2$.

Subsequently, to describe the structure of sum terms and product terms, the following simplification of notation will be useful.

Let $n\ge 2$ and $s, s_1, \ldots, s_n\in \rmST(X)$. 
We will regard the expressions $s_1 \hat{+} \ldots \hat{+} s_n$ and $s_1 \hat{\times} \ldots \hat{\times} s_n$  
as terms where,  for notational ease in our denotation, as usual, we left out some parentheses. 
Moreover, we will write  $s = s_1 \hat{+} \ldots \hat{+} s_n$  (respectively, $s=s_1 \hat{\times} \ldots \hat{\times} s_n$), if  $s_1 \hat{+} \ldots \hat{+} s_n$  
(respectively, $s_1 \hat{\times} \ldots \hat{\times} s_n$) is obtained from  $s$ in that way.
Formally, $s = s_1 \hat{+} \ldots \hat{+} s_n$  means that 
there exists a term  $c\in \T_{\{\hat{+}\}}(Z_n)$  
in which each of the variables  in $Z_n$  occurs exactly once and the order of these variables is $z_1,\ldots,z_n$,  and we obtain $s$ by replacing  $z_i$ in $c$   by  $s_i$, for each $i\in[n]$.
If this is the case, then we write 
$s = c[s_1,\ldots,s_n]$. We define  $s = s_1 \hat{\times} \ldots \hat{\times} s_n$ analogously with $c\in \T_{\{\hat{\times}\}}(Z_n)$. In fact, 
$s_1 \hat{+} \ldots \hat{+} s_n$   and $s_1 \hat{\times} \ldots \hat{\times} s_n$
are similar to the usual flattening of the term $s$ (cf. \cite[Sect.~5]{bacpla85}).

If $s = s_1 \hat{+} \ldots \hat{+} s_n$ and, for each  $i \in [n]$, the term $s_i$ is a product term or  $s_i\in X\cup \{\hat{1}\}$, then  we call  $s_1 \hat{+} \ldots \hat{+} s_n$ a \emph{sum-product decomposition} of  $s$.
Analogously, if $s = s_1 \hat{\times} \ldots \hat{\times} s_n$ and, for each  $i \in [n]$, the term $s_i$ is a sum term or  $s_i\in X$, then we call $s_1 \hat{\times} \ldots \hat{\times} s_n$ a \emph{product-sum decomposition} of  $s$ .

For instance, let $s=((x\hat{\times} y)\hat{+}\hat{1})\hat{+}(z\hat{\times}x)$. Then $(x\hat{\times} y)\hat{+}\hat{1}\hat{+}(z\hat{\times}x)$ is a sum-product decomposition of $s$ 
with $c=(z_1\hat{+}z_2)\hat{+}z_3$, $s_1=x\hat{\times} y$, $s_2=\hat{1}$, and $s_3=z\hat{\times}x$. Hence we also write $s=(x\hat{\times} y)\hat{+}\hat{1}\hat{+}(z\hat{\times}x)$.
Analogously, let $s=(x\hat{\times}(\hat{1}\hat{+}y))\hat{\times}(x\hat{+}z)$. Then $x\hat{\times}(\hat{1}\hat{+}y)\hat{\times}(x\hat{+}z)$ is a product-sum decomposition of $s$ with $c=(z_1\hat{\times}z_2)\hat{\times}z_3$,
$s_1=x$, $s_2=\hat{1}\hat{+}y$, and $s_3=x\hat{+}z$, i.e.,
we write $s=x\hat{\times}(\hat{1}\hat{+}y)\hat{\times}(x\hat{+}z)$.

By considering the structure of sum  terms and product terms, the following is immediate.

\begin{observation}\label{obs:unique-decomposition}\rm $\,$
\begin{compactitem}

\item[(a)] Each sum term $s \in \rmST(X)$  has a unique sum-product 
 decomposition $s = s_1 \hat{+} \ldots \hat{+} s_n$.  \hfill
 
\item[(b)] Each product term $s \in \rmST(X)$  has a unique product-sum decomposition
$s = s_1 \hat{\times} \ldots \hat{\times} s_n$. 

 \hfill$\Box$
\end{compactitem}    
\end{observation}

\begin{lemma}\label{lm:AC-equiv-characterization} \rm For every $s,t \in  \rmST(X)\setminus \{ \hat{0}, \hat{1} \}$, we have $s =_{\rmAC} t$ if and only if one of the following two conditions hold:
\begin{compactenum}
\item[(a)] both $s$ and $t$ are sum terms, and if $s = s_1 \hat{+}\ldots \hat{+} s_n$  and $t = t_1 \hat{+}\ldots \hat{+} t_k$ are their sum-product decompositions, then we have  $k = n$ and there is a permutation $\varphi :[n] \to [n]$ 
such that $s_i =_{\rmAC} t_{\varphi(i)}$ for each $i\in [n]$.
\item[(b)] both $s$ and $t$ are product terms, and if $s =s_1 \hat{\times} \ldots \hat{\times} s_n$ and 
$t = t_1 \hat{\times} \ldots \hat{\times} t_k$ are their product-sum decompositions, then we have  $k = n$ and $s_i =_{\rmAC} t_i$ for each $i\in [n]$.
\end{compactenum}
\end{lemma}
\begin{proof}  By Lemma  \ref{lm:approx-characterization}, $s =_{\rmAC} t$  is equivalent to $s \Leftrightarrow^*_{\rmAC} t$. We will use this fact in the rest of the proof.

It is clear that any of the conditions (a) and (b) implies $s \Leftrightarrow^*_{\rmAC} t$  and hence  $s =_{\rmAC} t$. 

Now assume that $s \Leftrightarrow^*_{\rmAC} t$. First we claim that either both $s$ and $t$ are sum terms or they are both product terms. 
For this, let $s_1,s_2,t_1,t_2 \in  \rmST(X)$ such that $s=s_1\hat{+} s_2$ and $t=t_1\hat{\times} t_2$. Then $s_1\hat{+} s_2  \Leftrightarrow^*_{\rmAC} t_1\hat{\times} t_2$
is impossible because by using identities of $\rmAC$, we cannot change the root symbol $\hat{+}$ of $s_1\hat{+} s_2$ to $\hat{\times}$. Our claim follows.

We continue the proof with case distinction. First we consider the case that both $s$ and $t$ are sum terms. Consider their sum-product decompositions $s = s_1 \hat{+}\ldots \hat{+} s_n$  and 
$t = t_1 \hat{+} \ldots \hat{+} t_k$. By our assumption, we have
\begin{equation}\label{eq:AC-transformation-new}
s=  s_1 \hat{+} \ldots \hat{+} s_n \Leftrightarrow^*_{\rmAC} t_1 \hat{+} \ldots \hat{+} t_k = t\enspace.
\end{equation}

Each $s_i$ is a product term or is in $X\cup \{\hat{1}\}$. 
Hence, if we apply the identities of $\rmAC$ inside some term $s_i$, then this cannot change the root of $s_i$, hence the result is again a product term (or the same element of  $X \cup \{\hat{1}\}$). 
The only other way to apply an identity of $\rmAC$ to the term $s_1 \hat{+}\ldots \hat{+} s_n$ 
is to use identity $e_1$ (associativity of sum) to change the parenthesizing of this sum or to use identity $e_2$  (commutativity of sum) to interchange some summands  $s_i, s_j$.
In each case, we obtain again a sum-product decomposition with the same number of summands. Thus $k=n$, and the identities  $e_1, e_2$ can change the order of the summands.
Therefore, \eqref{eq:AC-transformation-new} implies also that there is a permutation $\varphi :[n] \to [n]$ 
such that $s_i \Leftrightarrow^*_{\rmAC} t_{\varphi(i)}$ for each $i\in [n]$. Hence condition  (a) holds. 

Now we consider the case that both $s$ and $t$ are product terms. Similarly as in the previous case, consider their product-sum decompositions $s = s_1 \hat{\times} \ldots \hat{\times} s_n$  and 
$t = t_1 \hat{\times} \ldots \hat{\times} t_k$. 

If we apply the identities of $\rmAC$ inside some term $s_i$, then, by what we noted above, the result is again a sum term (or the same element of  $X$). The only other way to apply an identity of $\rmAC$ to the term 
$s_1 \hat{\times}\ldots \hat{\times} s_n$ is to use identity $e_4$ (associativity of product) to change the parenthesizing of the product.
In each case, we obtain again a product-sum decomposition with the same number of factors. Thus $k=n$, the identity  $e_4$ can change the parenthesizing of the product-sum decomposition, and $s_i \Leftrightarrow^*_{\rmAC} t_i$ 
for each $i\in [n]$. Hence condition  (b) holds. 
\end{proof}

Since in all of the following,  $\rmAC$-equivalence of simple terms will be very important for us,
we include a further graph-theoretic characterization, also for later use for decision algorithms (see Lemma \ref{lm:AC-equiv-decidable} and Corollaries \ref{cor:polynom-term-equiv-is-decidable} and \ref{cor:equiv-is-decidable}).
We will represent simple terms as labeled trees 
(i.e., particular labeled and directed graphs, cf. \cite{ahohopull74}) as follows. 

Let  $\boxplus$ and $\boxtimes$ be two symbols with 
$\boxplus,\boxtimes \notin X \cup \{\hat{0},\hat{1}\}$. For each simple term $s \in  \rmST(X)$, 
we define the labeled tree \emph{$\overline{s}$ 
 corresponding to $s$} \label{page:overline-s} by case distinction and induction on $\size(s)$ as follows.
\begin{compactitem}
\item[(a)] If $s\in X\cup \{\hat{0},\hat{1}\}$, then $\overline{s}$ is the tree with one node labeled by $s$.
\item[(b)] Let $s$ be a sum term with sum-product decomposition $s=s_1\hat{+}\ldots \hat{+}s_n$ and let $\overline{s_1},\ldots, \overline{s_n}$  the trees which correspond to $s_1,\ldots,s_n$, respectively. Then $\overline{s}$
is the tree whose root is labeled by $\boxplus$
and for each $i\in[n]$ there is an edge from this root to the root of $\overline{s_i}$. 
\item[(c)] Let $s$ be a product term with product-sum decomposition $s=s_1\hat{\times}\ldots \hat{\times}s_n$ and let $\overline{s_1},\ldots, \overline{s_n}$  the trees which correspond to $s_1,\ldots,s_n$, respectively. Moreover, for each $i\in[n]$, let 
$s'_i$ be the tree obtained from $\overline{s_i}$ by replacing its root label $y\in X\cup \{\hat{+}\}$ by the pair $(i,y)$. Then $\overline{s}$ is the tree consisting of its root and the trees  $s'_1, \ldots, s'_n$, where the root is labeled by $\boxtimes$
and for each $i\in[n]$ there is an edge from this root to the root of $s'_i$. 
\end{compactitem}

Clearly, for each $s \in  \rmST(X)$, we can construct $\overline{s}$ in $O(n)$ time, where $n=\size(s)$.

An \emph{isomorphism} of two labeled trees is, as usual, a bijection preserving the labeling and the edge relation \cite{ahohopull74}.
Above, in (c), the particular relabeling of the roots will ensure that isomorphisms of trees have to preserve the order of the factors in product-sum decompositions.

The following result is well-known in the area of term rewriting (cf., e.g., \cite{benkapnar87,akujantaktam17}). We include a short proof
for completeness and ease of the reader.

\begin{proposition}\label{prop:AC-equiv-graphs} \rm
Let $s,t \in  \rmST(X)$  be simple terms. Then $s =_{\rmAC} t$ if and only if  
$\overline{s}$  is isomorphic to  $\overline{t}$.
\end{proposition}

\begin{proof}  We proceed by induction on $\size(s)$ and case distinction. If  $s \in X \cup \{ \hat{0}, \hat{1} \})$, the result is immediate.

Therefore assume now that  $s$  is a sum term with sum-product decomposition  $s=s_1\hat{+}\ldots \hat{+}s_n$. First assume that $s =_{\rmAC} t$. 
Then, by Lemma \ref{lm:AC-equiv-characterization}(a), $t$  is a sum term with a sum-product decomposition  
$t = t_1 \hat{+}\ldots \hat{+} t_n$ and there is a bijection $\varphi :[n] \to [n]$ 
such that $s_i =_{\rmAC} t_{\varphi(i)}$ for each $i\in [n]$. By induction hypothesis, for each  $i \in [n]$, there is an isomorphism  
$\psi_i$  from $\overline{s_i}$  to  $\overline{t_i}$. Let  $\psi$  be the mapping which maps the root of $\overline{s}$  to the root of  $\overline{t}$ and is the joint extension of the mappings  $\psi_i$ $(i \in [n])$. Then  $\psi$  is an isomorphism from  $\overline{s}$  to  $\overline{t}$. 

Second, assume there is an isomorphism $\psi$  from  $\overline{s}$  to  $\overline{t}$. 
Then  $\psi$  maps the root of $\overline{s}$  to  $\overline{t}$. Hence these two roots are both labeled with $\boxplus$,
and  $t$  is a sum term with a sum-product decomposition $t = t_1 \hat{+}\ldots \hat{+} t_k$. 
Since $\psi$  maps  $\overline{s}$  isomorphically to  $\overline{t}$, we obtain $k = n$,
and  $\psi$  induces a bijection  $\varphi :[n] \to [n]$  
such that for each  $i \in [n]$, $\psi$  maps  $\overline{s_i}$  to  $\overline{t_{\varphi(i)}}$.
By induction hypothesis, we obtain  $s_i =_{\rmAC} t_{\varphi(i)}$  for each  $i \in [n]$. 
Then Lemma \ref{lm:AC-equiv-characterization}(a) implies  $s =_{\rmAC} t$.

Next, let  $s$  be a product term with product-sum decomposition $s =s_1 \hat{\times} \ldots \hat{\times} s_n$. 
If  $s =_{\rmAC} t$, we proceed similarly as above, using Lemma \ref{lm:AC-equiv-characterization}(b),
and we obtain that $\overline{s}$ is isomorphic to  $\overline{t}$.

Conversely, assume there is an isomorphism $\psi$  from  $\overline{s}$  to  $\overline{t}$.
Now the roots of  $\overline{s}$  and  $\overline{t}$  are both labeled with $\boxtimes$ and correspond to each other by $\psi$. 
Hence  $t$  is a product term with a product-sum decomposition 
$t = t_1 \hat{\times}\ldots \hat{\times} t_k$. 

Consider the trees  $\overline{s_i}$, $s'_i$ $(i \in [n])$
and the trees  $\overline{t_j}$, $t'_j$ $(j \in [k])$  occurring in the constructions of  $\overline{s}$  and  $\overline{t}$. 
Since $\psi$  maps  $\overline{s}$  isomorphically to  $\overline{t}$,
we obtain  $k = n$ and that  
there is a bijection  $\varphi :[n] \to [n]$  
such that for each  $i \in [n]$, $\psi$  maps  $s'_i$  isomorphically to  $t'_{\varphi(i)}$.
Since  $\psi$  preserves the labelings, it follows that  $i = \varphi(i)$, for each  $i \in [n]$.
Moreover, $\overline{s_i}$  is isomorphic to  $\overline{t_i}$, 
and by induction hypothesis we obtain that $s_i =_{\rmAC} t_i$, for each  $i \in [n]$. 
Then Lemma \ref{lm:AC-equiv-characterization}(b) implies  $s =_{\rmAC} t$.
\end{proof}

Next we turn to our investigation of the strong bimonoid  $\sfNsb[X]$.
Since in  $\sfNsb[X]$  the addition is associative and commutative and the multiplication is associative,
in the following as usual we will write sums and (ordered) products of several elements of  $\bbNsb[X]$  often without parentheses,
 where multiplication binds stronger than addition.

In the following, we abbreviate $X/\!{=_{\rmAC}}$ by $\XAC$ and $[s]_{=_{\rmAC}}$ by $[s]_{\rmAC}$ for each $s\in \rmST(X)$.  Hence $\XAC=\{ [x]_{\rmAC} \mid x \in X \}$, where $[x]_{\rmAC}=\{x\}$ for each $x \in X$.

The elements of $\sfNsb[X]$ are AC-congruence classes of simple terms. Therefore we name them
\emph{simple term classes}.
For the following, it will be important to have uniqueness results
for the representation of these simple term classes. 
We will employ sum terms and product terms and their congruence classes, called, respectively, sum classes  and product classes: A simple term class $p \in  \bbNsb[X]$ is a \emph{sum class} (\emph{product class}) if there is a
sum term (product term) $u \in \rmST(X)$ such that $p=[u]_{\rmAC}$.
Clearly, for each $p \in  \bbNsb[X]$, the following holds: if $p$ is a sum class (product class), then  there are $q, r \in  \bbNsb[X] \setminus \{ 0 \} $
($q, r \in  \bbNsb[X] \setminus \{ 0, 1\} $) such that $p=q + r$ ($p=q \times r$).

Let  $n \geq 2$  and  $p, p_1,\ldots, p_n \in \bbNsb[X]$.  
If  $p = p_1 + \ldots + p_n$, and, for each  $i\in [n]$, $p_i$  is a product class or  $p_i\in \XAC \cup \{1\}$,
then we call $p_1 + \ldots + p_n$ a \emph{sum-product decomposition} of  $p$.\\
If $p = p_1 \times \ldots \times p_n$, and, for each  $i\in [n]$, $p_i$  is a sum class or  $p_i\in \XAC$,
then we call $p_1 \times \ldots \times p_n$ a \emph{product-sum decomposition} of  $p$.
Now we can show the following.

\begin{lemma}\label{lm:uniqueness-lemma} \rm (a)  Each simple term class  $p \in \bbNsb[X] \setminus (\XAC\cup \{ 0, 1 \})$  is either
a sum class or a product class (but not both). 

(b)  Each sum class $p\in \bbNsb[X]$  has a sum-product decomposition
$p = p_1 + \ldots + p_n$.
Moreover, for every sum-product decomposition  $p = q_1 + \ldots + q_k$  of  $p$  we have  $k = n$
and the sequence $q_1,\ldots, q_n$  constitutes a permutation of  $p_1,\ldots,p_n$.

(c)  Each  product class $p \in \bbNsb[X]$  has a unique product-sum decomposition
$p = p_1 \times \ldots \times p_n$.
\end{lemma}
\begin{proof} Let  $p  \in \bbNsb[X] \setminus (\XAC \cup \{ 0, 1 \})$.

(a) Clearly, $p$ is sum class or a product class. Assume that it is both a sum class and a product class.
Then there is a sum term $s\hat{+}t$ and a product term $s'\hat{\times} t'$ in $\rmST(X)$ such that $[s\hat{+}t]_{\rmAC}=p=[s'\hat{\times} t']_{\rmAC}$. 
By Lemma \ref{lm:AC-equiv-characterization}, this is a contradiction. 

(b) Let $p$  be a sum class, i.e., let $p=[s]_{\rmAC}$ for some sum term $s\in \rmST(X)$. Let  
$s = s_1 \hat{+} \ldots \hat{+} s_n$
be the sum-product decomposition of  $s$  and, for each $i \in [n]$, let $p_i = [s_i]_{\rmAC}$. 
Then we have $p = p_1 + \ldots + p_n$, and this sum constitutes a sum-product decomposition of  $p$.

We show that $n$ is unique and the sequence $p_1, \ldots,  p_n$ is also unique up to a permutation.

For this, assume there $p = q_1 + \ldots + q_k$  is a further sum-product decomposition of  $p$, i.e., $q_i$   is a product class or  $q_i\in \XAC \cup \{1\}$ for each $i\in[k]$. Choose terms  $t_1,\ldots, t_k \in \rmST(X)$ such that,
for each  $i \in [k]$, $q_i = [t_i]_{\rmAC}$ and $t_i$  is a product term or $t_i \in X \cup \{\hat{1}\}$.
Let $t \in \rmST(X)$ be a term with sum-product decomposition 
$t=t_1 \hat{+} \ldots \hat{+} t_k$. Then  $p = [t]_{\rmAC}$.
Hence $s =_{\rmAC} t$. Since $s$ and $t$ sum terms, by  Lemma  \ref{lm:AC-equiv-characterization}(a) we obtain that $k=n$ and  $t_1,\ldots, t_n$  is, up to ${\rmAC}$-equivalence of the summands, a permutation of  $s_1,\ldots,s_n$. Consequently, $q_1,\ldots, q_n$  is a permutation of  $p_1,\ldots,p_n$. This proves  (b). 

(c) Now assume that  $p$  is a product class, i.e., there exists a product term $s\in \rmST(X)$ with $p=[s]_{\rmAC}$. 
Let $s = s_1 \hat{\times} \ldots \hat{\times} s_n$
be the product-sum decomposition of  $s$ and, for each $i \in [n]$, let $p_i = [s_i]_{\rmAC}$. Then we have $p=p_1 \times \ldots \times  p_n$, and this constitutes a product-sum decomposition of  $p$.

We show that $n$ is unique and also the sequence $p_1, \ldots,  p_n$ is unique.

For this, let $p = q_1 \times \ldots \times q_k$  be a further product-sum decomposition of  $p$. 
Choose terms  $t_1,\ldots, t_k \in \rmST(X)$  such that, for each $i \in [k]$,
$q_i = [t_i]_{\rmAC}$ and the term $t_i$ is a sum term or  $t_i\in X$. 
Let $t \in \rmST(X)$ be a term with product-sum decomposition 
$t=t_1 \hat{\times} \ldots \hat{\times} t_k$. Then  $p = [t]_{\rmAC}$.
Hence $s=_{\rmAC} t$. So, since  $s$ and $t$ are product terms, by Lemma \ref{lm:AC-equiv-characterization}(b), we obtain that $k=n$ and $s_i =_{\rmAC} t_i$ for each $i\in[n]$. 
Hence, also  $p_i =q_i$ for each $i\in[n]$.
This proves  (c). 
\end{proof}

Clearly, if  $s, t$  are simple terms and  $s =_{\rmAC} t$, then by Lemma \ref{lm:approx-characterization}, we have
$\size(s) = \size(t)$. Hence for each $p \in \bbNsb[X]$  we define \emph{the size of $p$}  by $\size(p) = \size(s)$, where
$s$ is a term in $\rmST(X)$  with  $p = [s]_{\rmAC}$. Subsequently, in our proofs for a simple class $p \in \bbNsb[X]$ 
we will often proceed by an induction over  $\size(p)$.

Next, we will define the subclass of polynomial terms in  $\rmST(X)$
and then the corresponding subclass of polynomials in  $\bbNsb[X]$. This is motivated by the usual definition of polynomials in the free ring or semiring over  $X$.
More precisely, we want to capture those terms which we obtain from
arbitrary simple terms by applying repeatedly right-distributivity, but not left-distributivity, of multiplication over addition
(combined with associativity of multiplication).
 
First we define monomial terms and monomials. For this, we recall that the elements of  $\T_{\{\hat{\times}\}}(X)$  can be viewed as products of elements only from  $X$,
with any kind of parenthesizing. 

\begin{definition}\label{def:monomial-terms}\rm The set of \emph{monomial terms} is the set 
$\T_{\{\hat{\times}\}}(X)$.

For each monomial term $t \in \T_{\{\hat{\times}\}}(X)$, we call its
congruence class  $[t]_{\rmAC}$  a \emph{monomial}. The set of all monomials is the set 
\[\{[x_1]_{\rmAC} \times\ldots\times [x_n]_{\rmAC} \mid n \in \mathbb{N}_+ \text{ and } x_1,\ldots, x_n\in X \}.\]
\hfill$\Box$
\end{definition}

\begin{definition}\label{def:polynomial-terms}\rm 
The set of \emph{polynomial terms}, denoted by $\rmPT(X)$, is the smallest subset $U$ of $\rmST(X)$ satisfying the following conditions:
\begin{compactitem}
\item[(a)]  $\hat{0}$, $\hat{1}$, and all monomial terms are in $U$.
\item[(b)]  If $s,t\in U$ with $s\ne \hat{0} \ne t$, then $s \hat{+} t$ is in $U$.  
\item[(c)]  If  $s$  is a monomial term and  $t\in U$ with $t\not\in \{\hat{0},\hat{1}\}$,
then  $s \hat{\times} t$ is in $U$. \hfill$\Box$
\end{compactitem}
\end{definition}

For every $s,t\in \rmPT(X)$ with $s\ne \hat{0} \ne t$, the term $s \hat{+} t$ is called a \emph{sum polynomial term}, and 
for every monomial term $s$ and $t\in \rmPT(X)$ with $t\not\in \{\hat{0},\hat{1}\}$, the term $s \hat{\times} t$ is called a \emph{product polynomial term}.

For each polynomial term (respectively, sum polynomial term, product polynomial term)
$t \in \rmPT(X)$, we call its  congruence class  $[t]_{\rmAC}$  
a \emph{polynomial} (respectively, a \emph{sum polynomial}, a \emph{product polynomial}).  

Now, we let  \[\bbNrd[X] = \rmPT(X)/\!{=_{\rmAC}}\,=\{ [t]_{\rmAC} \mid t \in \rmPT(X) \},\]
and call $\bbNrd[X]$ the set of all polynomials in  $\bbNsb[X]$. In particular, $0, 1\in  \bbNrd[X]$. 
Subsequently, our goal is to obtain a right-distributive strong bimonoid structure on $\bbNrd[X]$; this explains the subscript $\mrd$.

\begin{observation}\label{obs:polynomials} \rm The set $\bbNrd[X]$ is the smallest subset $U$ of $\bbNsb[X]$ satisfying the following conditions:
\begin{compactitem}
\item[(a)]  $0$, $1$, and all monomials are in $U$.
\item[(b)]  If $p,q\in U$ with $p\ne 0 \ne q$, then $p + q$ is in $U$.  
\item[(c)]  If  $m$  is a monomial  and  $q\in U$ with $q\not\in \{0,1\}$,
then  $m \times q$ is in $U$.  \hfill$\Box$
\end{compactitem}
\end{observation}

Clearly, for each $p \in  \bbNrd[X]$, the following holds: if $p$ is a sum polynomial (product polynomial), then  there are $q, r \in  \bbNrd[X] \setminus \{ 0 \} $
(a monomial $m$ and $q \in  \bbNrd[X] \setminus \{ 0, 1 \} $) such that $p=q + r$ ($p=m \times q$, respectively).

The next lemma is analogous to Lemma \ref{lm:uniqueness-lemma} and states existence and uniqueness results
for the representation of polynomials of $\bbNrd[X]$.

\begin{lemma}\label{lm:uniqueness-lemma-polynomials} \rm (a)  Each polynomial  $p \in \bbNrd[X] \setminus (\XAC\cup \{ 0, 1 \})$  is either a sum polynomial or a product polynomial (but not both). 

(b)  Each sum polynomial $p \in \bbNrd[X]$  has a sum-product decomposition,
$p = p_1 + \ldots + p_n$, such that, for each  $i\in [n]$, $p_i$  is a product polynomial or  $p_i\in \XAC \cup \{1\}$.
Moreover, for every sum-product decomposition
$p = q_1 + \ldots + q_k$, we have  $k = n$ and 
$q_1,\ldots, q_n$  constitutes a permutation of  $p_1,\ldots,p_n$.

(c)  For each  product polynomial $p \in \bbNrd[X]$ either of the following two conditions holds:
$p$ is a monomial or there are a unique monomial $m$ and a unique sum polynomial $q$ such that $p=m\times q$.
\end{lemma}

\begin{proof} (a), (b) These parts are straightforward by induction on  $\size(p)$, using Lemma \ref{lm:uniqueness-lemma}.

(c)  By Observation \ref{obs:polynomials}, the set $\bbNrd$ of polynomials can be obtained 
by closing the set containing  $0, 1$ and all monomials
under sums and under products from the left with monomials. 
Let  $p \in \bbNrd[X]$  be a product polynomial. 
By the above and Lemma \ref{lm:uniqueness-lemma}, we have  $p = m \times q$  for some monomial  $m$  and a polynomial  $q$.
We proceed by case distinction and induction on  $\size(p)$. 
If  $q$ is a monomial, we are done.  If $q$  is a sum polynomial, then the uniqueness of $m$ and $q$ follows from Lemma \ref{lm:uniqueness-lemma}. 
Now assume that  $q$  is a product polynomial.
By induction hypothesis, we have $q = m' \times q'$
for some monomial $m'$ and a sum polynomial $q'$. Hence  $p = (m \times m') \times q'$  as required.
The uniqueness part follows again by Lemma \ref{lm:uniqueness-lemma}.    
\end{proof}

Next we show that the set $\rmPT(X)$ of polynomial terms is closed under $\rmAC$-equivalence.

\begin{lemma}\label{lm:pol-classes-subset-simple-classes}\rm
For every  $s \in \rmPT(X)$ and $t \in \rmST(X)$, if  $s =_\rmAC t$, then also  $t \in \rmPT(X)$. 
\end{lemma}
\begin{proof}
We proceed by case distinction and induction on  $\size(s)$. For this, assume that  $t \in \rmST(X)$  with  $s =_{\rmAC} t$.

The statement obviously holds for $s=\hat{0}$ ($s=\hat{1}$) because in this case $t=\hat{0}$ ($t=\hat{1}$).

If  $s$  is a monomial term, then by Lemma  \ref{lm:AC-equiv-characterization}, $t$  is also a monomial term. 

Now assume that  $s$ is a sum polynomial term. Then $s$  has a sum-product decomposition $s = s_1 \hat{+} \ldots \hat{+} s_n$
such that each $s_i$  $(i \in [n])$ is an element of $X\cup\{\hat{1}\}$ or a product polynomial term. By Lemma \ref{lm:AC-equiv-characterization}, there is a sum-product decomposition $t = t_1 \hat{+} \ldots \hat{+} t_n$
and, moreover, for each $t_i$ there is an $s_j$ with $s_j =_{\rmAC} t_i$.
By the induction hypothesis, each $t_i$ is a polynomial term, hence $t$ is a sum polynomial term.

Lastly, assume that  $s$ is a product polynomial term which is not a monomial term. By Lemma \ref{lm:uniqueness-lemma-polynomials}(c), there are a unique monomial $m$  and a sum polynomial $q$ with
$[s]_{\rmAC}=m\times q$. Hence there are a monomial term $s'$ and a sum polynomial term $s''$ such that $m=[s']_{\rmAC}$ and $q=[s'']_{\rmAC}$. Then
$t =_{\rmAC} s =_{\rmAC} s' \hat{\times} s''$. 
By Lemma \ref{lm:AC-equiv-characterization}(b), $t$ is a product term and it  has a product-sum decomposition
$t = t' \hat{\times} t''$  with  $s' =_{\rmAC } t'$  and $s'' =_{\rmAC } t''$.
Hence, $t'$  is a monomial term and $t''$  is a sum term and by induction hypothesis a polynomial term. Thus  $t \in \rmPT(X)$. 
\end{proof}

In an equivalent formulation, Lemma \ref{lm:pol-classes-subset-simple-classes} says that if  $s \in \rmPT(X)$, 
then  $[s]_{\rmAC} \subseteq \rmPT(X)$.

We note that $\bbNrd[X]$ is closed under the addition $+$. However,  it  is not closed under the multiplication $\times$,
since, e.g., for each  $x \in X$, $p = 1 + [x]_\rmAC$  and  $q = [x]_\rmAC$  are polynomials,
but  $p \times q  =  (1 + [x]_\rmAC) \times [x]_\rmAC$  is not a polynomial. Therefore, to equip $\bbNrd[X]$ with a strong bimonoid structure, we 
 can take the restriction of  $+$  to   $\bbNrd[X]$, but we define a new multiplication  $\times_\mrd$ on $\bbNrd[X]$ so that, together with the $+$, we will obtain a right-distributive strong bimonoid.
We define the operation  $\times_\mrd$  as follows.

\begin{definition}\rm\label{def:operations} Let  $q, r \in \bbNrd[X]$. We define  $q \times_\mrd r$ by case distinction and induction on $\size(q)$ as follows. We put  $q \times_\mrd 0 = 0 = 0 \times_\mrd q$  and
$q \times_\mrd 1 = q = 1 \times_\mrd q$. Now let  $q, r$  be different from  $0$  and  $1$.

\begin{compactitem}
\item[(a)]  If  $q$  is a monomial, we let  $q \times_\mrd r = q \times r$, a product polynomial.

\item[(b)] Let $q$ be a sum polynomial. By Lemma \ref{lm:uniqueness-lemma-polynomials}(b), we can write
 $q = q_1 + \ldots + q_n$ where $n \geq 2$ and
each $q_i$ is $1$, a monomial, or a product polynomial. By induction hypothesis, $q_i \times_\mrd r $ is already defined for each $i\in[n]$.
Then we define $q \times_\mrd r = q_1 \times_\mrd r + \ldots + q_n  \times_\mrd r$. 
Note that the description of  $q$  as a sum of  $n$  number of $1$'s, monomials, or product polynomials is unique in the sense of Lemma \ref{lm:uniqueness-lemma-polynomials}(b).
Therefore    $q \times_\mrd r$  is well-defined, and  it is a sum polynomial.

\item[(c)]  Let $q$ be  product polynomial which is not a monomial. By Lemma \ref{lm:uniqueness-lemma-polynomials}(c), there are a unique monomial $m$ and a sum polynomial $q'$
with $q = m \times q'$. By induction hypothesis, $q' \times_\mrd r$  is already defined. We define $q \times_\mrd r = m \times (q' \times_\mrd r)$.
Since $m$ and $q'$ are unique, the product $q \times_\mrd r$  is well-defined and it is a product polynomial.
 \hfill$\Box$
\end{compactitem}
\end{definition}

Given  $q, r \in \bbNrd[X]$  and polynomial terms  $s, t \in \rmPT(X)$  with  $q = [s]_\rmAC$ and $ r = [t]_\rmAC$, we have  $q \times_\mrd r\in \bbNrd[X]$. Hence, by the definition of $\bbNrd[X]$, there is a polynomial term  $u \in \rmPT(X)$  such that  $q \times_\mrd r = [u]_\rmAC$. 
We could find such a polynomial term  $u$  by following the inductive procedure 
described above in Definition \ref{def:operations}.

Next we wish to give a \emph{direct} one-step construction for finding this polynomial term  $u$ from $s$ and $t$. We call an element of  $X \cup \{\hat{1}\}$ which occurs as a subterm of  $s$  also a \emph{leaf of} $s$.
Intuitively, we will perform the multiplication with the term  $t$  at all occurrences of the leaf  $\hat{1}$  in $s$ and at particular occurrences of leaves from  $X$  in  $s$.

Formally, for any polynomial terms  $s, t \in \rmPT(X)\setminus\{\hat{0}\}$, we  define the term  $s\langle\hti t\rangle$. 

For this, first we define two auxiliary concepts. A subterm $s'$ of $s$ is called a \emph{sum subterm of $s$} if $s'$ is a sum term. Next, let  $v$  be a monomial term. We say that \emph{$v$  occurs as a summand in $s$}, if  $s = v$  or $v$ is a summand in a sum subterm of $s$, i.e.,  $v$  is a child of a  $\hp$-symbol of  $s$. We note that if $\hat{1}$ occurs in $s$, then it occurs as a summand in $s$ in the above sense.

Now we define $s\langle\hti t\rangle$ as follows. We put  $\hat{1}\langle\hti t\rangle = t$  and  $s\langle\hti \hat{1}\rangle = s$. 
 If  $s \in X$ and $t\ne \hat{1}$, then we put  $s\langle\hti t\rangle = s \hti t$. 

If  $\size(s) \geq 2$, then we let  $s\langle\hti t\rangle$  
be the term obtained from  $s$  as follows:

\begin{compactitem}
\item whenever $\hat{1}$ occurs in $s$, 
then we replace this leaf $\hat{1}$  by $t$,
\item whenever  $x\in X$ is the rightmost factor of a monomial term which occurs as a summand in $s$, 
then we replace this leaf $x$  by  $x \hti t$.
\end{compactitem}

We give some examples:
\begin{compactitem}
\item[-] $(x\hp y)\langle\hti t\rangle=(x\hti t)\hp (y\hti t)$  and $(x\hti y)\langle\hti t\rangle=x\hti (y\hti t)$,
\item[-] $\Big((x\hti y) \hti\big(\hat{1}\hp (y\hti z)\big)\Big)\langle\hti t\rangle=(x\hti y) \hti\big(t\hp (y\hti (z\hti t))\big)$,
\item[-] $\Big(\big(x\hti (y\hp z)\big)\hp \big(y\hp (x\hti y)\big)\Big)\langle\hti t\rangle=\big(x\hti ((y\hti t)\hp (z\hti t))\big)\hp \big((y\hti t)\hp (x\hti (y\hti t))\big)$,
\item[-] $\Big(x\hti \Big((y\hp\hat{1})\hp \big( \hat{1}\hp (x\hti y)\big)\Big)\Big)\langle\hti t\rangle=x\hti \Big(((y\hti t)\hp t)\hp \big( t\hp (x\hti (y\hti t))\big)\Big)$.
\end{compactitem}    

Note that in the construction of  $s\langle\hti t\rangle$, the structure of  $s$  determines which occurrences of leaves of  $s$  get replaced; in particular, this does not depend on  $t$. Clearly,  $s\langle\times t\rangle$  is a simple term.

Observe that in case  $s$  is a sum term with sum-product decomposition  $s = s_1 \hp \ldots \hp s_n$, 
then  $s\langle\hti t\rangle = s_1\langle\hti t\rangle \hp \ldots \hp s_n\langle\hti t\rangle$. 

In case  $s$  is a product term such that  $s = s_1 \hti s_2$  with a monomial term  $s_1$  and a polynomial sum term $s_2$, the product-sum decomposition of  $s = s_1 \hti s_2$  has the sum term  $s_2$  as its rightmost factor.
Consequently, all leaves of  $s$  which are rightmost factors of some monomial term occurring as a summand in $s$, 
are leaves of  $s_2$. Hence $s\langle\hti t\rangle = s_1 \hti (s_2\langle\hti t\rangle)$.

Now we can show that in the setting above, we obtain our goal with the polynomial term  $u = s\langle\hti t\rangle$:

\begin{lemma}\label{lm:rd-pol-terms-for-products} \rm 
Let $s, t \in \rmPT(X) \setminus \{\hat{0}\}$.  Then $s\langle\hti t\rangle \in \rmPT(X)$  and $[s]_\rmAC \times_\mrd [t]_\rmAC = [s\langle\hti t\rangle]_\rmAC$.
\end{lemma} 

\begin{proof} We proceed by induction on the size of  $s$.
If  $s = \hat{1}$  or $t = \hat{1}$, the result is clear. 
Now we assume that  $s \neq \hat{1} \neq t$.
If  $s \in X$, we have  $[s]_\rmAC \times_\mrd [t]_\rmAC = [s]_\rmAC \times [t]_\rmAC
= [s \hti t]_\rmAC = [s\langle\hti t\rangle]_\rmAC$. Hence let  $\size(s) \geq 2$.

First, assume that  $s$  is a monomial term, so  $s = x_1 \hti \ldots \hti x_n$  for some  $n \geq 2$
and  $x_i \in X$  for each  $i \in [n]$. Then
$[s]_\rmAC \times_\mrd [t]_\rmAC = [s]_\rmAC \times [t]_\rmAC 
= [s \hti t]_\rmAC = [x_1 \hti \ldots \hti x_{n-1}\hti(x_n \hti t)]_\rmAC
= [s\langle\hti t\rangle]_\rmAC$, as claimed.

Next, assume that $s$  is a sum term with sum-product decomposition $s = s_1 \hp \ldots \hp s_n$  and polynomial terms  $s_1, \ldots, s_n$. Then  $s\langle\hti t\rangle = s_1\langle\hti t\rangle \hp \ldots \hp s_n\langle\hti t\rangle$  and by our induction assumption we obtain
$s_i\langle\hti t\rangle \in \rmPT(X)$  and $[s_i]_\rmAC \times_\mrd [t]_\rmAC = [s_i\langle\hti t\rangle]_\rmAC$  for each  $i \in [n]$.
So, $s\langle\hti t\rangle \in \rmPT(X)$  and  
\begin{align*}
[s]_\rmAC \times_\mrd [t]_\rmAC 
= & [s_1]_\rmAC \times_\mrd [t]_\rmAC + \ldots + [s_n]_\rmAC \times_\mrd [t]_\rmAC
= [s_1\langle\hti t\rangle]_\rmAC + \ldots + [s_n\langle\hti t\rangle]_\rmAC
\\ = & [s_1\langle\hti t\rangle \hp \ldots \hp s_n\langle\hti t\rangle]_\rmAC
= [s\langle\hti t\rangle]_\rmAC.
\end{align*}

Finally, let  $s = s_1 \hti s_2$ with a monomial term  $s_1$  and a polynomial sum term  $s_2$. By applying our induction hypothesis to  $s_2$, we obtain
$[s]_\rmAC \times_\mrd [t]_\rmAC 
= [s_1]_\rmAC \times ([s_2]_\rmAC \times_\mrd [t]_\rmAC)
= [s_1]_\rmAC \times [s_2\langle\hti t\rangle]_\rmAC$.
On the other hand, as noted above we have
$s\langle\hti t\rangle = (s_1 \hti s_2)\langle\hti t\rangle = s_1 \hti (s_2\langle\hti t\rangle)$.
Thus  $[s\langle\hti t\rangle]_\rmAC = [s_1 \hti (s_2\langle\hti t\rangle)]_\rmAC 
= [s_1]_\rmAC \times [s_2\langle\hti t\rangle]_\rmAC$, and the result follows.    
\end{proof}

Next we show that the  direct one-step  construction $..\langle\hti..\rangle$, as a binary operation, is associative.

\begin{lemma}\label{lm:associativity}\rm
For every $s, t, u \in \rmPT(X)\setminus\{\hat{0}\}$, we have 
$s\langle\hti (t\langle\hti u\rangle)\rangle = (s\langle\hti t\rangle)\langle\hti u\rangle$. 
\end{lemma}
\begin{proof} The statement is obvious if some of $s,t$ and $u$ are equal to $\hat{1}$. 

Therefore assume that $s, t, u \in \rmPT(X)\setminus\{\hat{0}, \hat{1}\}$. We observe that the leaves of  $s$  which get replaced in the construction of
$s\langle\hti t\rangle$  are the same as those leaves of $s$  which get replaced in the construction of
$s\langle\hti (t\langle\hti u\rangle)\rangle$.  
Moreover, the leaves of  $t$ which get replaced in the construction of
$t\langle\hti u\rangle$ correspond to those leaves of copies of  $t$  in  $s\langle\hti t\rangle$
which get replaced in the construction of  $(s\langle\hti t\rangle)\langle\hti u\rangle$.
This implies the statement.
\end{proof}

Now we can show one of our main results.

\begin{theorem}  \label{thm:strong-bimonoid-Pol} $\sfNrd[X]=(\bbNrd[X], +, \times_\mrd, 0, 1)$  is a right-distributive strong bimonoid.
\end{theorem}

\begin{proof}   Since  $\bbNrd[X]$  is closed under addition by  $+$, $(\bbNrd[X], +, 0)$  is a commutative monoid.

Now let  $p, q, r \in \bbNrd[X]$. We claim that  $p \times_\mrd (q \times_\mrd r) = (p \times_\mrd q) \times_\mrd r$.
We may assume that  $p, q, r \neq 0$ because otherwise the statement is trivial.
Choose polynomial terms  $s, t, u \in \rmPT(X)$  with  $p = [s]_\rmAC, q = [t]_\rmAC$  and  $r = [u]_\rmAC$. 
By Lemmas \ref{lm:rd-pol-terms-for-products} and \ref{lm:associativity}, we obtain
\begin{align*}
p \times_\mrd (q \times_\mrd r) = & [s]_\rmAC \times_\mrd [t\langle\hti u\rangle]_\rmAC
=  [s\langle\hti(t\langle\hti u\rangle)\rangle]_\rmAC = [(s\langle\hti t\rangle)\langle\hti u\rangle]_\rmAC 
= [s\langle\hti t\rangle]_\rmAC \times_\mrd r \\
= & (p \times_\mrd q) \times_\mrd r,
\end{align*}
as needed.
Thus $(\bbNrd[X], \times_\mrd, 1)$  is a monoid.

For right-distributivity, we have to show that  
$(p + q) \times_\mrd r = p \times_\mrd r + q \times_\mrd r$.
We may assume that  $p, q, r \neq 0$.
Choose polynomial terms  $s, t, u \in \rmPT(X)$  with  $p = [s]_\rmAC, q = [t]_\rmAC$  and  $r = [u]_\rmAC$. By Lemma \ref{lm:rd-pol-terms-for-products}, we obtain
\begin{align*}
(p + q) \times_\mrd r = & ([s \hp t]_\rmAC) \times_\mrd [u]_\rmAC
= [(s \hp t)\langle\hti u\rangle]_\rmAC = [s\langle\hti u\rangle \hp t\langle\hti u\rangle]_\rmAC 
= [s\langle\hti u\rangle]_\rmAC + [t\langle\hti u\rangle]_\rmAC \\
= & (p \times_\mrd r) + (q \times_\mrd r), 
\end{align*}
as needed. Hence  $\sfNrd[X]$  is a right-distributive strong bimonoid.
\end{proof}

In the following, we will prove a general cancellativity result for the 
multiplication $\times_\mrd$ of the right-distributive strong bimonoid $\sfNrd[X]$. 
As preparation, we begin with a natural particular case.
For a simple term class  $p \in \bbNsb[X]$ and  $n \in \mathbb{N}$, we let
$n \cdot p = p + \ldots + p$ ($n$  summands). First we show:

\begin{lemma}\rm\label{lm:cancellation-for-simple-classes}  
Let  $p, q \in \bbNsb[X]$ and  $n \in \mathbb{N}_+$.
Then  $n \cdot p = n \cdot q$  implies  $p = q$.
\end{lemma}

\begin{proof} Assume  $n \cdot p = n \cdot q$.
First, consider that  $p \in X$  or  $p$ is a product class.
Then  $n \cdot p = p + \ldots + p$ is the sum-product decomposition of  $n \cdot p$.
If  $q$  was a sum class, the sum-product decomposition of  $n \cdot q$  would contain
at least $2n$  summands. This contradicts Lemma \ref{lm:uniqueness-lemma}(b).
Hence also  $q \in X$  or  $q$ is a product class.
Then  $n \cdot q = q + \ldots + q$ is the sum-product decomposition of  $n \cdot q$.
Then Lemma \ref{lm:uniqueness-lemma}(b) implies  $p = q$.

Therefore we may assume that  $p$  and  $q$  are sum classes. 
Let  $p = p_1 + \ldots + p_k$ and  $q = q_1 + \ldots + q_{k'}$  
be their sum-product decompositions.
Then  $n \cdot p$  and  $n \cdot q$  have sum-product decompositions
with  $n \cdot k$  resp.  $n \cdot k'$  summands.
Then by Lemma \ref{lm:uniqueness-lemma}(b), we have  $k = k'$  and
there is a bijection $\pi$  from the summands of the sum-product decomposition
$n\cdot p = n \cdot p_1 + \ldots + n\cdot p_k$  to the 
summands of the sum-product decomposition
$n\cdot q = n \cdot q_1 + \ldots + n\cdot q_k$  preserving and reflecting 
equality of the summands.
Let  $I \subseteq \{1,\ldots,k\}$ such that the classes  $p_i$ $(i \in I)$
are pairwise different and  $\{p_i \mid i \in I\} = \{p_1,\ldots,p_k\}$.
Then also the elements  $\pi(p_i)$ $(i \in I)$  are pairwise different and
$\{\pi(p_i) \mid i \in I \} = \{q_1,\ldots,q_k\}$.
It follows that for each  $i \in I$ we have
$|\{ j \in \{1,\ldots,k\} \mid p_j = p_i \}| = |\{ j \in \{1,\ldots,k\} \mid q_j = \pi(p_i) \}|$. Thus  $p = p_1 + \ldots + p_k = q_1 + \ldots + q_k = q$, as claimed.    
\end{proof}

Next, we prove the following general cancellation result for the strong bimonoid of polynomials. 
Note that usually in such cancellativity results one considers two products 
where either the two left factors or the two right factors are equal;
here we only require the much weaker property that the two right factors  $r, r'$  have the same size. 
In the semiring  $\mathsf{N}[X]=(\mathbb{N}[X], +, \cdot, 0, 1)$,
the result corresponding to Theorem \ref{thm:cancellation-strong-in-Pol} clearly does not hold,
since, e.g., $(x+1)(x+x) = xx + xx + x + x = (x+x)(x+1)$  and  $\size(x+x) = \size(x+1)$. Note that,
due to the left-distributivity, the multiplication and the equality in $\mathsf{N}[X]$ are different from the one in $\bbNrd[X]$; this explains why the stronger result of 
Theorem \ref{thm:cancellation-strong-in-Pol} does not hold in $\mathsf{N}[X]$.  

\begin{theorem}\label{thm:cancellation-strong-in-Pol}  Let $p,q,r,r'\in \bbNrd[X] \setminus \{0\}$  with  $r \neq 1 \neq r'$
such that  $p \times_\mrd r = q \times_\mrd r'$  and  $\size(r) = \size(r')$. Then  $p = q$  and  $r = r'$. 
\end{theorem}

\begin{proof} We proceed by induction on  $p$.
If  $p \neq 1$, by  $r \neq 1$ we obtain  $\size(p \times_\mrd r) > \size(r)$. 
Together with the assumption that  $\size(r) = \size(r')$, this implies that  $p = 1$  if and only if  $q = 1$.
Hence we may assume that  $p \neq 1 \neq q$.

First assume that  $p \in X/{\rmAC}$  or  $p$  is a product polynomial.
Then  $p \times_\mrd r$, hence also $q \times_\mrd r'$, is a product polynomial.
Thus also  $q \in X/{\rmAC}$  or  $q$  is a product polynomial.
We write  $p = [x]_\rmAC \times p'$  and  $q = [y]_\rmAC \times q'$  with  $x, y \in X$  and 
$p', q' \in \bbNrd[X] \setminus \{0\}$ (possibly, $p' = 1$  or  $q' = 1$).
Then  $p \times_\mrd r = [x]_\rmAC \times (p' \times_\mrd r)$  and  
$q \times_\mrd r' = [y]_\rmAC \times (q' \times_\mrd r')$.
Choose  $u, v \in \rmPT(X)$  with  $p' \times_\mrd r = [u]_\rmAC$  and  $q' \times_\mrd r' = [v]_\rmAC$.
We have  $p \times_\mrd r = [x \hti u]_\rmAC$  and  $q \times_\mrd r' = [y \hti v]_\rmAC$.
By Proposition \ref{prop:AC-equiv-graphs}, there is an isomorphism  $\varphi$
from $\overline{x \hti u}$  onto   $\overline{y \hti v}$.
The labeled graph $\overline{x \hti u}$  has  $\boxtimes$  as a root and children depending on the structure of  $u$. 
If  $u$  is a sum term, the children of  $\boxtimes$  in   $\overline{x \hti u}$  are labeled with  $(1,x)$  and  $(2,\boxplus)$.
If  $u \in X$, the children of  $\boxtimes$  are labeled with   $(1,x)$  and  $(2,u)$.  
If  $u$  is a product polynomial term with product-sum decomposition $u=u_1\hti \ldots \hti u_n$, then $\boxtimes$ has $n+1$ children.
The first child is labeled with  $(1,x)$. If $u$ is a monomial, then the further $n$ children are labeled with $(2,u_1),\ldots, (n+1,u_n)$, otherwise, $u_n$  is a sum term and the children  are labeled with $(2,u_1),\ldots, (n,u_{n-1}),(n+1,\boxplus)$.

We have a similar description of the labeled graph  $\overline{y \hti v}$. 
It follows that  $\varphi$  maps the child of the root of $\overline{x \hti u}$ labeled with
$(1,x)$  onto the child of the root of $\overline{y \hti v}$ labeled with
$(1,y)$  and that  
$\varphi$  induces an isomorphism from  $\overline{u}$  onto  $\overline{v}$.
Thus  $x = y$, and by Proposition \ref{prop:AC-equiv-graphs}, we obtain  $[u]_\rmAC = [v]_\rmAC$.
Hence  $p' \times_\mrd r = [u]_\rmAC = [v]_\rmAC = q' \times_\mrd r'$.
By our induction hypothesis, we obtain  $p' = q'$  and  $r = r'$, showing also  $p = q$, as claimed.

Second, assume that  $p$, and by the above hence also  $q$,  is a sum polynomial. Write
$p = p_1 + \ldots + p_k$  and  $q = q_1 + \ldots + q_n$  where $k, n \geq 2$  and
for each  $i \in [k]$  and  $j \in [n]$,  $p_i$  and  $q_j$  are product polynomials or elements of
$X/{\rmAC} \cup \{ 1 \}$. Then
$p \times_\mrd r = p_1 \times_\mrd r + \ldots + p_k \times_\mrd r$  and
$q \times_\mrd r' = q_1 \times_\mrd r' + \ldots + q_n \times_\mrd r'$.

Now consider the sum-product decomposition of  $p_1 \times_\mrd r + \ldots + p_k \times_\mrd r$
and the sum-product decomposition of  $q_1 \times_\mrd r' + \ldots + q_n \times_\mrd r'$. 
Note that if  $i \in [k]$  and  $p_i = 1$, then  $p_i \times_\mrd r = r$.
Moreover, if  $i \in [k]$  and  $p_i \neq 1$, then  $p_i \in X/{\rmAC}$  or  $p_i$  is a product polynomial, 
hence, using $r \neq 1$ in case $p_i \in X/{\rmAC}$, 
we obtain that $p_i \times_\mrd r$  is a product polynomial and 
thus appears as a summand in the sum-product decomposition of  
$p_1 \times_\mrd r + \ldots + p_k \times_\mrd r$
and that  $\size(p_i \times_\mrd r) > \size(r)$.
A similar observation holds for the sum-product decomposition of  
$q_1 \times_\mrd r' + \ldots + q_n \times_\mrd r'$. 

By Lemma \ref{lm:uniqueness-lemma-polynomials}(b), there is a bijection  $\varphi$  which maps the summands
of the first sum-product decomposition onto equal summands of the second sum-product decomposition.
By the above size considerations and  $\size(r) = \size(r')$, 
it follows that  $\varphi$  maps the set  $\{ p_i \times_\mrd r \mid p_i \neq 1, i \in [k] \}$ 
bijectively onto the set  $\{ q_j \times_\mrd r' \mid q_j \neq 1, j \in [n] \}$ (preserving the possible multiplicities of the summands contained in these sets). Moreover, if
$i \in [k]$, $j \in [n]$  with  $p_i \neq 1 \neq q_j$ and  $\varphi(p_i \times_\mrd r) = q_j \times_\mrd r'$, then $p_i \times_\mrd r = q_j \times_\mrd r'$, so  $p_i = q_j$  by our induction hypothesis, and, provided that there are such  $i \in [k]$  and  $j \in [n]$  as described,
we also obtain $r = r'$  by our induction hypothesis.
Consequently,
$\sum_{i \in [k], p_i \neq 1} p_i = \sum_{j \in [n], q_j \neq 1} q_j $.

Now let  $k_1 = |\{ i \in [k] \mid p_i = 1 \}|$  and  $n_1 = |\{ j \in [n] \mid q_j = 1 \}|$.
Then  $\varphi$  also maps the sum-product decomposition of  $k_1 \cdot r$  
onto the sum-product decomposition of  $n_1 \cdot r'$.
Thus  $k_1 \cdot r = n_1 \cdot r'$.
Hence  $k_1 \cdot \size(r) + (k_1 - 1) = \size(k_1 \cdot r) = \size(n_1 \cdot r') = n_1 \cdot \size(r') + (n_1 - 1)$,
showing  $k_1 = n_1$  since $\size(r) = \size(r')$.
Thus  $k_1 \cdot r = n_1 \cdot r' = k_1 \cdot r'$.
Then, provided that  $k_1 \neq 0$, 
Lemma \ref{lm:cancellation-for-simple-classes}  implies  $r = r'$.
Hence, in each case we have  $r = r'$.
As shown above, we have  
$|\{ i \in [k] \mid p_i \neq 1 \}| = |\{ j \in [n] \mid q_j \neq 1 \}|$.
Together with  $k_1 = n_1$, this implies  $k = n$.
Consequently,  $p = p_1 + \ldots + p_k = q_1 + \ldots + q_k =  q$, as claimed.
\end{proof}

\

Clearly, Theorem \ref{thm:cancellation-strong-in-Pol} does not hold if $r=1$ and $r'\ne 1$, or vice versa, because  $[x] \times_\mrd 1 = 1 \times_\mrd [x]$  and $\size(1)=\size([x])$, but $1\ne [x]$.

We note that Theorem \ref{thm:cancellation-strong-in-Pol} generalizes the following (almost trivial) observation on non-commutative free monoids:

If $p,q,r,r'\in X^*$  such that  $pr = qr'$  and  $\size(r) = \size(r')$, 
then  $p = q$  and  $r = r'$.

Indeed, the above observation is contained in Theorem \ref{thm:cancellation-strong-in-Pol}, because the free monoid $X^*$ is isomorphic to the monoid generated by $X/\rmAC$ in $\sfNrd[X]$  with $\times_\mrd$.

Now we can show that the multiplication $\times_\mrd$ in $\sfNrd[X]$ is cancellative both from the left and the right. We will use right cancellativity later in the proof of Lemma \ref{lm:Pid-closed-product} and Theorem \ref{thm:strong-bimonoid-Pol-id}, and then also in Lemma \ref{lm:p-q-r-large}.

\begin{corollary}\label{cor:cancellation-in-Pol} \rm Let $p,q,r,r' \in \bbNrd[X] \setminus \{0\}$.
\begin{compactitem}
\item[(a)] If  $p \times_\mrd r = p \times_\mrd r'$, then $r = r'$. 
\item[(b)] If  $p \times_\mrd r = q \times_\mrd r$, then $p = q$.
\end{compactitem}
\end{corollary} 

\begin{proof} (a) Assume that  $p \times_\mrd r = p \times_\mrd r'$.
Choose $s, u, v \in \rmPT(X)$  with $p = [s]_\rmAC$, $r = [u]_\rmAC$, and $r' = [v]_\rmAC$. 
By Lemma \ref{lm:rd-pol-terms-for-products}, we obtain $s\langle\hti u\rangle =_\rmAC s\langle\hti v\rangle$, 
so  $\size(s\langle\hti u\rangle) = \size(s\langle\hti v\rangle)$.
It follows that  $u = \hat{1}$  if and only if  $v = \hat{1}$. Therefore we may assume that  $u \neq \hat{1} \neq v$.

Let  $m_1$  be the number of the occurrences of leaves  $x\in X$ of  $s$  which, in the construction of  $s\langle\hti u\rangle$, 
we replace by  $x \hti u$, and let $m_2$ be the number of the occurrences of the leaf $\hat{1}$ in $s$. These occurrences of the leaves in $X\cup\{\hat{1}\}$ of $s$ play the same role in the construction of $s\langle\hti v\rangle$.

Hence we have  
$\size(s\langle\hti u\rangle) = \size(s) + m_1 \cdot (\size(u) + 1) + m_2 \cdot (\size(u) - 1)$  and similarly
$\size(s\langle\hti v\rangle) = \size(s) + m_1 \cdot (\size(v) + 1) + m_2 \cdot (\size(v) - 1)$. 
Hence  $\size(u) = \size(v)$, so $\size(r) = \size(r')$.
Then Theorem \ref{thm:cancellation-strong-in-Pol} implies  $r = r'$.

(b) Immediate by Theorem \ref{thm:cancellation-strong-in-Pol}.
\end{proof}

Next we wish to obtain a right-distributive strong bimonoid of polynomials which is idempotent.
The main difference as compared to $\sfNrd[X]$ is that in all sums of the form
$p_1 + \ldots + p_n$  where for every  $i \in [n]$, $p_i$  is a product polynomial or an element of $X/{\rmAC}\cup\{1\}$,
each of these polynomials  $p_i$ occurs only once, i.e., the polynomials  $p_1,\ldots,p_n$  are pairwise different.

We begin with defining, slightly more general also for later purposes, 
\emph{idempotency-reduced} (for short \emph{id-reduced}) \emph{terms}  as follows.

\begin{definition}\label{def:id-reduced-terms-1}\rm
The set of \emph{id-reduced terms}, denoted by $\rmSTid(X)$ is the smallest subset $U$  of $\rmST(X)$ satisfying the following conditions: 
\begin{compactitem}
\item[(a)] $\hat{0}$, $\hat{1}$, and all monomial terms are in $U$,
\item[(b)] if  $t \in \rmST(X)$  is a sum term and it has a sum-product decomposition $t = t_1 \hat{+} \ldots \hat{+} t_n$
such that $t_1,\ldots,t_n \in U$ and $t_i \ne_{\rmAC} t_j$
for every $1\le i < j \le n$, then $t \in U$,
\item[(c)] if $s, t\in U \setminus \{\hat{0}, \hat{1}\}$, then the product term $s \hat{\times} t$ is in $U$.
\end{compactitem}

We put  $\rmPTid(X)= \rmPT(X) \cap \rmSTid(X)$, the set of all polynomial terms which are id-reduced.

\hfill$\Box$
\end{definition}

Now we show that we obtain id-reduced polynomial terms in the following way.

\begin{lemma}\label{lm:id-reduced-polynomial-terms}\rm
The set $\rmPTid(X)$ of id-reduced polynomial terms is the smallest subset $U$  of $\rmST(X)$ satisfying the following conditions: 
\begin{compactitem}
\item[(a)] $\hat{0}$, $\hat{1}$, and all monomial terms are in $U$,
\item[(b)] if  $t \in \rmST(X)$  is a sum term and it has a sum-product decomposition $t = t_1 \hat{+} \ldots \hat{+} t_n$
such that $t_1,\ldots,t_n \in U$ and $t_i \ne_{\rmAC} t_j$
for every $1\le i < j \le n$, then $t \in U$,
\item[(c)] if $s$ is a monomial term and $t\in U \setminus \{\hat{0}, \hat{1}\}$, then the product
term $s \hat{\times} t$ is in $U$.
\end{compactitem}
\end{lemma}

\begin{proof} Let  $U$  be defined as described above. The closure condition  (c) of the 
Lemma is weaker than (c) of Definition \ref{def:id-reduced-terms-1}, 
and the closure condition (b) of the Lemma is weaker than  (b) of  Definition \ref{def:polynomial-terms}. This shows that
$U \subseteq \rmSTid(X) \cap \rmPT(X)$.

To prove the converse, let  $s \in \rmSTid(X) \cap \rmPT(X)$. We show by induction on  $\size(s)$  that
$s \in U$. This is trivial if  $s \in \{\hat{0}, \hat{1}\}$ or if  $s$  is a monomial term.

Now let  $s$  be a sum term. By  $s \in \rmSTid(X) \cap \rmPT(X)$  and Lemma \ref{lm:AC-equiv-characterization}(a), it follows that $s$  has a sum-product decomposition 
$s = t_1 \hat{+} \ldots \hat{+} t_n$  where
$t_1,\ldots,t_n$ are id-reduced and polynomial terms and  $t_i \ne_{\rmAC} t_j$
for every $1\le i < j \le n$. By the induction hypothesis, we have $t_1,\ldots,t_n \in U$.
Then, by condition (b), we obtain that $s \in U$. 

Finally, let  $s$  be a product term, but not a monomial term. By  $s \in \rmSTid(X) \cap \rmPT(X)$  and Lemma \ref{lm:AC-equiv-characterization}(b), it follows that  $s = s'\hat{\times} t$, where $s'$ is a monomial term and $t$ is id-reduced and a sum polynomial term. 
By induction hypothesis, then  $t \in U$ and hence  $s \in U$ by condition (c).
\end{proof}

Next we show an analogous result as Lemma \ref{lm:pol-classes-subset-simple-classes}:

\begin{lemma}\label{lm:id-reduced-classes-subset-simple-classes}\rm
For every  $s \in \rmPTid(X)$ and $t \in \rmST(X)$, if  $s =_\rmAC t$, then also  $t \in \rmPTid(X)$. 
\end{lemma}

\begin{proof} We can follow the proof of Lemma \ref{lm:pol-classes-subset-simple-classes} almost verbatim,
observing the additional property of  $s$ given by Definition \ref{def:id-reduced-terms-1}(b).  
\end{proof}

In an equivalent formulation, Lemma \ref{lm:id-reduced-classes-subset-simple-classes} says that if  $s \in \rmPTid(X)$, 
then  $[s]_{\rmAC} \subseteq \rmPTid(X)$.

For each id-reduced polynomial term $t \in \rmPTid(X)$ (respectively,  id-reduced sum polynomial term, id-reduced product polynomial term), we call its
congruence class  $[t]_{\rmAC}$    an
\emph{id-reduced polynomial} (respectively, an \emph{id-reduced sum polynomial}, an \emph{id-reduced product polynomial}).  

Now, we let  \[\bbBidrd[X] = \rmPTid(X)/\!{=_{\rmAC}} \, = \{ [t]_{\rmAC} \mid t \in \rmPTid(X) \},\]
so $\bbBidrd[X]$ is the set of all id-reduced polynomials in  $\bbNrd[X]$. In particular, $0, 1\in  \bbBidrd[X]$.
Our next goal is to obtain an idempotent and  right-distributive strong bimonoid structure on $\bbBidrd[X]$, this explains the index id. 
The following is straightforward.

\begin{observation}\label{obs:id-reduced-polynomials}\rm The set $\bbBidrd[X]$  of \emph{id-reduced polynomials} is the smallest subset $U$ 
of $\bbNsb[X]$ satisfying the following conditions:
\begin{compactitem}
\item[(a)]  $0, 1$,  and all monomials (defined above) are in $U$. 

\item[(b)] If  $p \in\bbNsb[X]$  is a sum class and it has a sum-product decomposition $p = p_1+ \ldots + p_n$
with pairwise different $p_1,\ldots,p_n\in U$, then  $p \in U$.

\item[(c)]  If  $m$  is a monomial and  $q\in U$  is a  polynomial, 
then the product polynomial  $m \times q$  is in $U$.

\hfill$\Box$
\end{compactitem}
\end{observation}

We show a decomposition of the elements of $\bbBidrd[X]\setminus\{0\}$  that we will use later without any reference.

\begin{observation}\label{obs:id-reduced-polynomials-decomposition}\rm For each  $p\in \bbBidrd[X]\setminus\{0\}$ we have  $p = p_1 + \ldots + p_n$ for some $n\ge 1$, where each of
$p_1, \ldots, p_n$  is $1$, a monomial, or an id-reduced product polynomial. Moreover, if $n\ge 2$, then $p_1, \ldots, p_n$ are pairwise different.
\end{observation}
\begin{proof} By Observation \ref{obs:id-reduced-polynomials}, the elements of $\bbBidrd[X]$  can be obtained by starting with the elements described in Observation \ref{obs:id-reduced-polynomials}(a) and then closing this set by the sum and product operations described in (b) and (c), respectively. We show by induction that our statement holds for all the elements  $p\in \bbBidrd[X]\setminus\{0\}$ obtained in this way.

If $p$ has the form described in Observation \ref{obs:id-reduced-polynomials}(a), then the statement is immediate with $n=1$ and $p_1=p$. If $p$ is obtained by applying Observation \ref{obs:id-reduced-polynomials}(b), then $p = p_1 + \ldots + p_n$ is a sum-product decomposition, and each $p_i$ is an id-reduced polynomial by the construction hypothesis. Hence  $p$  has the required form (with  $n \geq 2$). If $p$ is obtained by applying Observation \ref{obs:id-reduced-polynomials}(c), then it is an id-reduced product polynomial. Hence the statement holds again with $n=1$ and $p_1=p$.
\end{proof}

Next we will define an addition operation and a multiplication operation on $\bbBidrd[X]$. As preparation, we show the following.
 
\begin{lemma}\label{lm:Pid-closed-product}\rm $\bbBidrd[X]$ is closed under $\times_\mrd$.
\end{lemma} 
\begin{proof} Let $q,r\in \bbBidrd[X]$. We show that $q\times_\mrd r\in \bbBidrd[X]$ by induction on the structure of $q$. We proceed by case distinction and induction on $\size(q)$.

If $q\in \{0,1\}$, the statement is obvious.

Next let $q$ be a monomial. Then  $q \times_\mrd r = q \times r$, 
and  $q \times r$  is id-reduced by Observation \ref{obs:id-reduced-polynomials}(c).
  
Secondly, let
$q = q_1 + \ldots + q_n$ where  $n \ge 2$, and each  $q_i$  is $1$, a monomial, or an id-reduced  product polynomial,  
and $q_1, \ldots , q_n$  are pairwise different. 
Then, by Corollary  \ref{cor:cancellation-in-Pol}(b), the product polynomials
$q_ 1 \times_{\mrd} r, \ldots, q_n  \times_{\mrd} r$  are also pairwise different and, by the induction hypothesis, they are id-reduced polynomials. 
Hence, by Observation \ref{obs:id-reduced-polynomials}(b), 
$q \times_{\mrd} r = q_1 \times_{\mrd} r + \ldots + q_n \times_{\mrd} r$ 
is an id-reduced polynomial.

Finally, let $q=m \times q'$, where $m$ is a monomial and $q'$ is an id-reduced sum polynomial. Then, by Definition \ref{def:operations}(b),  $q' \times_\mrd r$ is a sum polynomial and, by induction hypothesis,
it is id-reduced. By Definition \ref{def:operations}(c), we have $q \times_\mrd r = m \times (q' \times_\mrd r)$, where by Observation  \ref{obs:id-reduced-polynomials}(c)
the product polynomial in the right-hand side is id-reduced.
\end{proof}

As a consequence, we deduce that  $\rmPTid(X)$ is closed under the product operation $..\langle\hti ..\rangle$
and hence we may compute the product of idempotent polynomial terms in one step precisely as for polynomial terms.

\begin{corollary}\rm\label{cor:product-idempotent-polynomial-terms}
Let  $s,t \in \rmPTid(X)$. Then  $s\langle\hti t\rangle \in \rmPTid(X)$ and
$[s]_\rmAC \times_\mrd [t]_\rmAC = [s\langle\hti t\rangle]_\rmAC$.    
\end{corollary}

\begin{proof} Clearly, we have  $s\langle\hti t\rangle \in \rmPT(X)$.
By Lemma \ref{lm:rd-pol-terms-for-products} and Lemma \ref{lm:Pid-closed-product}, 
we get  $[s\langle\hti t\rangle]_\rmAC = [s]_\rmAC \times_\mrd [t]_\rmAC$
and  $[s]_\rmAC \times_\mrd [t_\rmAC] \in \bbBidrd[X]$.
Hence, $[s]_\rmAC \times_\mrd [t_\rmAC] = [t']_\rmAC$  for some  $t' \in \rmPTid(X)$.
Then  $t' =_\rmAC s\langle\hti t\rangle$, so by Lemma \ref{lm:id-reduced-classes-subset-simple-classes}
we obtain  $s\langle\hti t\rangle \in \rmPTid(X)$  as needed.   
\end{proof}

Recall that when passing from  $\sfNsb[X]$  to  $\sfNrd[X]$, we could keep the addition but had to adjust the multiplication operation.
Now, when passing from  $\sfNrd[X]$  to  $\sfBidrd[X]$, by Lemma \ref{lm:Pid-closed-product} we can take as multiplication the restriction of $\times_{\mrd}$ to $\bbBidrd[X]$, but we have to adjust the addition  $+$ in order to obtain its idempotency. We define the operation addition $+_{\mrid}$ on $\bbBidrd[X]$ as follows.

\begin{definition}\rm\label{def:id-operations} Let  $q, r \in \bbBidrd[X]$.
We put  $q +_{\mrid} 0 = q = 0 +_{\mrid} q$.

Now let us assume that $q \neq  0 \neq r$.
Then $q = q_1 + \ldots + q_k$, where $k\ge 1$, $q_1, \ldots, q_k$ are pairwise different and each of them is 
$1$, a monomial, or an id-reduced product polynomial, and $r = r_1 + \ldots + r_n$ with 
$n\ge 1$ and the corresponding conditions for  $r_1,\ldots, r_n$.  Then we define
$q +_{\mrid} r = q_1 + \ldots + q_k + r_{i_1} + \ldots + r_{i_\ell}$,  where $\ell\ge 0$ and
$\{ i_1,\ldots,i_\ell \}$ is the maximal subset of  $\{ 1,\ldots,n \}$  (with respect to inclusion) such that $\{q_1,\ldots,q_k\}\cap \{r_{i_1},\ldots,r_{i_\ell}\} =\emptyset$.\hfill$\Box$
\end{definition}

\begin{figure}
  \begin{tabular}{l|l}
    strong bimonoid   & carrier set\\\hline
    \\
    $\sfNsb[X] = (\bbNsb[X],+,\times,0,1)$ & $\bbNsb[X]=\rmST(X)/\!{=_{\rmAC}}$\\
    strong bimonoid  & $\rmST(X)$ is the set of simple terms\\
    (cf. Proposition~\ref{prop:SB-strong-bimonoid}) & $\rmST(X)\subseteq \T(X)$\\
    \\\hline
    \\
    $\sfNrd[X]=(\bbNrd[X], +, \times_\mrd, 0, 1)$ & $\bbNrd[X] = \rmPT(X)/\!{=_{\rmAC}}$\\
    right-distributive strong bimonoid  & $\rmPT(X)$ is the set of polynomial terms\\
    (cf. Theorem~\ref{thm:strong-bimonoid-Pol})& $\rmPT(X)\subseteq \rmST(X)$\\
    \\\hline
    \\
    $\sfBidrd[X]=(\bbBidrd[X], +_{\mrid}, \times_{\mrd}, 0, 1)$ & $\bbBidrd[X] = \rmPTid(X)/\!{=_{\rmAC}}$\\
    idempotent right-distr. strong bimonoid  & $\rmPTid(X)$ is the set of id-reduced terms\\
    (cf. Theorem~\ref{thm:strong-bimonoid-Pol-id}) & $\rmPTid(X) \subseteq \rmPT(X)$
  \end{tabular}

  \

  \
  
  \hspace*{6mm}$\rmAC = \{e_1,e_2,e_4\}$ is the set of  identities for associativity of $+$ and $\times$,  and commutativity of $+$.\\
  \caption{\label{fig:overview} An overview of the free strong bimonoids described in Section \ref{sect:strong-bimonoids-polynomials}.}
 \end{figure}


Now we can show the following result.

\begin{theorem}\label{thm:strong-bimonoid-Pol-id} $\sfBidrd[X]=(\bbBidrd[X], +_{\mrid}, \times_{\mrd}, 0, 1)$  is an idempotent right-distributive strong bimonoid.
\end{theorem}

\begin{proof} It is immediate from the definition that the operation $+_{\mrid}$ is idempotent.
To check associativity, let  $p, q, r \in \bbBidrd[X] \setminus \{0\}$.
We claim that  $p +_{\mrid} (q +_{\mrid} r) = (p +_{\mrid} q) +_{\mrid} r$.
Each of the three polynomials is either $1$, a monomial, an id-reduced product polynomial  or it is
a sum of pairwise different monomials, id-reduced product polynomials, or $1$’s.
On the right-hand side of the equation we obtain the sum of all these monomials, id-reduced product polynomials or $1$  of  $p$,
then those ones of  $q$  added which are different from the ones of  $p$,
and then those of  $r$  added which are different from the ones of both  $p$  and  $q$.
This equals the left-hand side of the equation.
This proves our claim.

Similarly, it is easy to see that the addition operation  $+_{\mrid}$  is commutative.

By Lemma \ref{lm:Pid-closed-product} and Theorem \ref{thm:strong-bimonoid-Pol}, the operation  $\times_{\mrd}$ is associative.
To show that  $\times_{\mrd}$  is right-distributive over $+_{\mrid}$, 
we can follow the argument for the corresponding statements of Theorem \ref{thm:strong-bimonoid-Pol},
but using id-reduced polynomials and the operations  $+_{\mrid}$  and  $\times_{\mrd}$;
all corresponding sums of monomials, id-reduced product polynomials, or  $1$s  have each summand occurring only once.
Here we use Corollary \ref{cor:cancellation-in-Pol}(b) again.

This shows that  $\sfBidrd[X]$  is an idempotent right-distributive strong bimonoid. 
\end{proof}

We just note that the strong bimonoid  $\sfBidrd[X]$  has the same right- and 
left-cancellation properties as shown in Theorem \ref{thm:cancellation-strong-in-Pol} and Corollary \ref{cor:cancellation-in-Pol} for $\sfNrd[X]$;
this is immediate from the two mentioned results, since the multiplication $\times_\mrd$ in $\sfBidrd[X]$ is the restriction of the multiplication of $\sfNrd[X]$.

In Figure \ref{fig:overview} we show the three strong bimonoids $\sfNsb[X]$, $\sfNrd[X]$, and $\sfBidrd[X]$ defined in this section and the definition of their carrier sets. In the following section we will show that they are free in their respective classes of strong bimonoids. 
The reader interested only in the construction of a right-distributive and weakly locally finite strong bimonoid that is not locally finite 
 may directly proceed to Section \ref{sect:wlc-strong-bimonoids}.


\section{The polynomial strong bimonoids are free}\label{sect:free-strong-bimonoids-polynomials}

In this section we will first prove the result indicated in the title, i.e.,
we will to show that the strong bimonoids of polynomials  $\sfNsb[X], \sfNrd[X]$, and $\sfBidrd[X]$  are,
respectively, a free strong bimonoid, a free right-distributive strong bimonoid, and
a free idempotent right-distributive strong bimonoid.
This is analogous to Theorem \ref{thm:free-semiring} that the semiring $\mathsf{N}[X]$ of all polynomials in non-commuting variables with coefficients from $\mathbb{N}$ is free in the class of all semirings. 
Then, we deduce consequences of this result on representing arbitrary terms by simple terms, 
by polynomial terms and by id-reduced polynomial terms.
First we show the following result.

\begin{theorem}\label{thm:S-is free}
$\sfNsb[X]$ is a free strong bimonoid, freely generated by $X/{\rmAC}$.
\end{theorem}
\begin{proof}
By Proposition  \ref{prop:SB-strong-bimonoid}, $\sfNsb[X]$ is a strong bimonoid. Clearly,
$\sfNsb[X]$ is generated by $X/{\rmAC}$.
Let  $\sfA=(A, \oplus, \otimes, 0_\sfA, 1_\sfA)$  be any strong bimonoid, and let  $h: X/{\rmAC} \rightarrow A$
be a mapping. We have to show that  
$h$  extends to a strong bimonoid homomorphism from  $\sfNsb[X]$  to  $\sfA$. 

First, recall the  $\Sigma_{\mathrm{sb}}$-algebra  $\sfST(X)=(\rmST(X),+_\mathsf{ST},\times_\mathsf{ST},\hat{0},\hat{1})$ from Section \ref{sect:strong-bimonoids-polynomials}.
By a straightforward induction, we obtain a  $\Sigma_{\mathrm{sb}}$-algebra homomorphism  $h'$  from  $\sfST(X)$
to the strong bimonoid  $\A$  satisfying  $h'(x) = h([x]_{\rmAC})$  for each  $x \in X$.
We define a mapping  $h'': \bbNsb[X] \to A$  by putting  $h''([t]_{\rmAC}) = h'(t)$
for each  $t \in \rmST(X)$. 
We claim that  $h''$  is well-defined and a strong bimonoid  homomorphism;
then our result follows.

To show that  $h''$  is well defined, let  $t, t' \in \rmST(X)$  with  $t =_{\rmAC} t'$.
We have to show that  $h'(t) = h'(t')$. By Lemma \ref{lm:approx-characterization}, we have  $t \Leftrightarrow^*_{\rmAC} t'$.
By symmetry and transitivity of equality, it suffices to consider the case that  $t \Rightarrow_{\rmAC} t'$. So assume that $t'$  is obtained from  $t$ in a reduction step using one of the identities  $e_1, e_2$  or  $e_4$. Then $t'$  is obtained from  $t$  by replacing a subterm  $u$ of  $t$  by a term  $u'$ such that $h'(u) = h'(u')$. For instance, if we use $e_2$, then $u = u_1 \hat{+} u_2$ and $u' = u_2 \hat{+} u_1$. Clearly, $h'(u) = h'(u_1) \oplus h'(u_2) = h'(u_2) \oplus h'(u_1) = h'(u')$ because  $h'$ is a $\Sigma_{\mathrm{sb}}$-algebra homomorphism and $\oplus$ is commutative. We obtain $h'(u) = h'(u')$ in the other two cases similarly. Using again that  $h'$  is a   homomorphism, it follows that  $h'(t) = h'(t')$. Hence  $h''$  is well-defined.

It remains to show that  $h''$  is a strong bimonoid homomorphism.
Clearly,  $h''([\hat{0}]_{\rmAC}) = h'(\hat{0}) = 0_\sfA$  and  $h''([\hat{1}]_{\rmAC}) = h'(\hat{1}) = 1_\sfA$.
If  $t, t' \in \rmST(X)$, we have  
$h''([t]_{\rmAC} + [t']_{\rmAC}) = h''([t +_{\sfST} t']_{\rmAC}) = h'(t +_{\sfST} t')
= h'(t) \oplus h'(t') = h''([t]_{\rmAC}) \oplus h''([t']_{\rmAC})$, 
as needed; the case of multiplication is analogous.
The result follows.
\end{proof}

Next, we can prove the corresponding result for right-distributive resp. idempotent right-distributive strong bimonoids.

\begin{theorem}\label{thm:Pol-Polid-are-free}$\,$ 
\begin{compactitem}
\item[(a)] $\sfNrd[X]$ is a free right-distributive strong bimonoid, freely generated by $X/{\rmAC}$.
\item[(b)] $\sfBidrd[X]$ is a free idempotent right-distributive strong bimonoid, freely generated by $X/{\rmAC}$.
\end{compactitem}
\end{theorem}

\begin{proof} (a)  By Theorem \ref{thm:strong-bimonoid-Pol}, $\sfNrd[X]$  is a right-distributive strong bimonoid. 
Clearly, $\sfNrd[X]$ is generated by $X/{\rmAC}$.
Let  $\sfA=(A, \oplus, \otimes, 0_\sfA, 1_\sfA)$  be a right-distributive strong bimonoid and  $h: X/{\rmAC} \rightarrow A$  a mapping.
By Theorem \ref{thm:S-is free}, $h$  extends to a strong bimonoid homomorphism  $h'$
from  $\sfNsb[X]$  to   $\sfA$. Note that  $\bbNrd[X] \subseteq \bbNsb[X]$.
We claim that  $h'$, restricted to  $\bbNrd[X]$, also constitutes a strong bimonoid homomorphism
from  $\sfNrd[X]$  to   $\sfA$. Then the result follows. 

Let  $q, r \in \bbNrd[X] \setminus \{ 0 \}$.
Clearly, we have  $h'(q + r) = h'(q) \oplus h'(r)$ because $h'$
is a homomorphism from  $\sfNsb[X]$  to   $\sfA$.

Next, we claim that  $h'(q \times_\mrd r) = h'(q) \otimes h'(r)$.
Clearly, we may assume that  $q, r \neq 1$.

We proceed by case distinction and induction on the size of  $q$.

First, assume  $q$  is a monomial. Then  $q \times_\mrd r = q \times r$  and we obtain
$h'(q \times_\mrd r) = h'(q \times r) = h'(q) \otimes h'(r)$  as claimed.

Secondly, let  $q$  be a sum polynomial, with sum-product decomposition  $q = q_1 + \ldots + q_n$, say.  In this case we have  $q \times_\mrd r = q_1 \times_\mrd r + \ldots + q_n \times_\mrd r$, a sum polynomial.
By induction hypothesis, we can assume that  $h'(q_i \times_\mrd r) = h'(q_i) \otimes h'(r)$  for each  $i\in[n]$.
Then we obtain
\begin{align*}
 & h'(q \times_\mrd r) = h'(q_1 \times_\mrd r + \ldots + q_n  \times_\mrd r) \\
= \; & h'(q_1 \times_\mrd r) \oplus \ldots \oplus h'(q_n  \times_\mrd r) \tag{since $h'$ is a homomorphism from $\sfNsb[X]$  to   $\sfA$} \\
= \; & h'(q_1) \otimes h'(r) \oplus \ldots \oplus h'(q_n) \otimes h'(r)  \tag{by induction hypothesis} \\
= \; & (h'(q_1) \oplus \ldots \oplus h'(q_n)) \otimes h'(r)  \tag{since  $\sfA$  is right-distributive}\\
= \; & h'(q_1 + \ldots + q_n) \otimes h'(r)   \tag{since $h'$ is a homomorphism from $\sfNsb[X]$  to   $\sfA$}\\
= \; & h'(q) \otimes h'(r). 
\end{align*}

Thirdly, let  $q$  be a product polynomial of the form  $q = m \times s$
with a monomial  $m$  and a sum polynomial  $s$ (cf. Lemma \ref{lm:uniqueness-lemma-polynomials}).
In this case, we have  $q \times_\mrd r = m \times (s \times_\mrd r)$.
By induction hypothesis, we have  $h'(s \times_\mrd r) = h'(s) \otimes h'(r)$.
Then
\begin{align*}
& h'(q \times_\mrd r) = h'(m \times (s \times_\mrd r))\\
= \; & h'(m) \otimes h'(s \times_\mrd r) \tag{since $h'$ is a homomorphism from $\sfNsb[X]$  to   $\sfA$}\\
= \; & h'(m) \otimes (h'(s) \otimes h'(r))  \tag{by induction hypothesis}\\
= \; & (h'(m) \otimes h'(s)) \otimes h'(r)\\
= \; & h'(m \times s) \otimes h'(r)   \tag{since $h'$ is a homomorphism from $\sfNsb[X]$  to   $\sfA$}\\
= \; & h'(q) \otimes h'(r). 
\end{align*}

Hence  $h'$  is a strong bimonoid homomorphism as claimed, and the result follows. 

(b)  By Theorem \ref{thm:strong-bimonoid-Pol-id}, $\sfBidrd[X]$  is an idempotent right-distributive strong bimonoid. 
Clearly, $\sfBidrd[X]$ is generated by $X/{\rmAC}$.
We consider a mapping  $h: X/{\rmAC}\to A$, where
 $\sfA=(A, \oplus, \otimes, 0_\sfA, 1_\sfA)$ is  an idempotent right-distributive strong bimonoid. By (a), we can extend  $h$  to a homomorphism  $h'$  from
$\sfNrd[X]$  to $\sfA$.
We claim  $h'$, restricted to the subset $\bbBidrd[X]$  of $\bbNrd[X]$,
constitutes a strong bimonoid homomorphism from
$\sfBidrd[X]$  to $\sfA$.

Let  $q, r \in \bbBidrd[X] \setminus \{ 0 \}$.
First we show that  $h'(q +_{\mrid} r) = h'(q) \oplus h'(r)$.

We have  $q = q_1 + \ldots + q_k$  and  $r = r_1 + \ldots + r_n$
where  $k,n \geq 1$  and  $q_1,\ldots, q_k, r_1,\ldots, r_n$ are $1$'s, monomials, or id-reduced product polynomials. Moreover,
if $k \ge 2$ ($n \ge 2$), then the polynomials  $q_1,\ldots, q_k$  
($r_1,\ldots, r_n$, respectively) are pairwise different.
Then  $q +_{\mrid} r$  is the sum of all these polynomials, but with each polynomial occurring only once.
So,  $q +_{\mrid} r = q_1 + \ldots + q_k + r_{i_1} + \ldots + r_{i_\ell}$,
where $\ell\ge 0$ and $\{ i_1,\ldots, i_\ell \}$ is the maximal subset of  $\{ 1,\ldots,n \}$ such that $\{ q_1,\ldots, q_k \}\cap\{r_{i_1}, \ldots, r_{i_\ell}\}=\emptyset$. Thus
\begin{equation}\label{eq:idemponetcy}
\{ q_1,\ldots, q_k, r_{i_1}, \ldots, r_{i_\ell} \} = \{ q_1,\ldots, q_k, r_1,\ldots, r_n \}\enspace.
\end{equation}

Then  $h'(q) = h'(q_1) \oplus \ldots \oplus h'(q_\ell)$,
similarly  $h'(r) = h'(r_1) \oplus \ldots \oplus h'(r_n)$, and
\begin{align*}
& h'(q +_{\mrid} r) 
= h'(q_1) \oplus \ldots \oplus h'(q_k) \oplus h'(r_{i_1}) \oplus \ldots \oplus h'(r_{i_\ell}) \\
= \; & h'(q_1) \oplus \ldots \oplus h'(q_k) \oplus h'(r_1) \oplus \ldots \oplus h'(r_n)   \tag{since $\sfA$  is idempotent and by \eqref{eq:idemponetcy}}\\
= \; & h'(q) \oplus h'(r).
\end{align*}

Second, by (a) we have  $h'(q \times_{\mrd} r) = h'(q) \otimes h'(r)$.
Hence  $h'$  is a strong bimonoid homomorphism as claimed, and the result follows.
\end{proof}

Now we obtain the following further description of our strong bimonoids of polynomials.
Its proof is immediate by Theorems \ref{thm:S-is free}, \ref{thm:Pol-Polid-are-free}, Proposition \ref{prop:three-free-sb} and the uniqueness of free structures (cf. Lemma \ref{lm:bijective-generating-sets}).

\begin{corollary}\label{cor:three-sb-isomorphic}\rm$\,$
\begin{compactitem}
\item[(a)]   The strong bimonoids  $\sfFB(X)$  and   $\sfNsb[X]$ are isomorphic.
\item[(b)]   The strong bimonoid $\sfFB_{\mrd}(X)$ and the strong bimonoid $\sfNrd[X]$ of polynomials  are isomorphic.
\item[(c)]   The strong bimonoid $\sfFB_{\midrd}(X)$  and the strong bimonoid $\sfBidrd[X]$ of idempotent polynomials 
 are isomorphic.\hfill$\Box$
\end{compactitem}
\end{corollary}

Next we investigate the different congruence relations and their relationships for simple terms, for polynomial terms and for id-reduced polynomial terms.

\begin{theorem}\label{thm:three-implications}$\,$
\begin{compactitem}
\item[(a)] Let  $t, t' \in \rmST(X)$. Then  $[t]_{E} = [t']_{E}$  in $\sfFB(X)$  implies  $t =_{\rmAC} t'$  in $\sfST(X)$. 
\item[(b)] Let  $t, t' \in \rmPT(X)$. Then $[t]_{\mrd} = [t']_{\mrd}$  in $\sfFB_{\mrd}(X)$  implies
$t =_{\rmAC} t'$  in  $\sfST(X)$. 
\item[(c)]  Let  $t, t' \in \rmPTid(X)$. Then $[t]_{\midrd} = [t']_{\midrd}$  in $\sfFB_{\midrd}(X)$  implies
$t =_{\rmAC} t'$  in  $\sfST(X)$. 
\end{compactitem}
\end{theorem}
\begin{proof} First, observe that for  $x \in X$  we have  $[x]_{\rmAC} = \{x\}$, but  $[x]_E$  is infinite,
as it contains  $x$, $x \hat{+} \hat{0}$, $x \hat{\times} \hat{1}$, etc.

(a)  We define two mappings
$f: X/{\rmAC} \to X/\!E$  and  $g: X/\!E \to X/{\rmAC}$  by putting
$f([x]_{\rmAC}) = [x]_E$  and   $g([x]_E) = [x]_{\rmAC}$  for each  $x \in X$;
note that  $g$  is well-defined as  $[x]_E \cap X=\{ x \}$.

By Theorem \ref{thm:S-is free} and Proposition \ref{prop:three-free-sb}(a), $f$  and  $g$  extend to homomorphisms
$f' : \bbNsb[X] \to \rmFB(X)$  and  $g' : \rmFB(X) \to  \bbNsb[X]$, respectively. 

Moreover,  we can  show  that
\begin{equation}\label{eq:f-for-terms}
\text{$f'([t]_{\rmAC}) = [t]_{E}$  for each  $t \in \rmST(X)$}
\end{equation}
by induction on the size of  $t$.

Then  $g' \circ f': \bbNsb[X] \rightarrow \bbNsb[X]$  is a homomorphism which acts like the identity on  $X/{\rmAC}$.
Since  $X/{\rmAC}$  generates $\bbNsb[X]$, we obtain  that $g' \circ f'$ is the identity on $\bbNsb[X]$.
Then  $f'$  is injective, and together with \eqref{eq:f-for-terms} this implies the result.

(b)  We define the mappings  $f: X/{\rmAC} \to X/{\mrd}$  and  $g: X/{\mrd} \to X/{\rmAC}$
analogously to (a).
By Theorem \ref{thm:Pol-Polid-are-free}(a) and Proposition \ref{prop:three-free-sb}(b),  $f$  and  $g$  extend to homomorphisms
$f' : \bbNrd[X] \to \rmFB_{\mrd}(X)$  and  $g' : \rmFB_{\mrd}(X) \to  \bbNrd[X]$, respectively.

Now we claim that
\begin{equation}\label{eq:f-for-polynomial-terms}
\text{$f'([t]_{\rmAC}) = [t]_{\mrd}$  for each  $t \in \rmPT(X)$.}
\end{equation}
We proceed by case distinction according to Definition \ref{def:polynomial-terms} and induction on the size of $t$. We only consider case (c) of that definition, as the cases (a) and (b) are similar and easier.
Assume that  $t=s \hat{\times} t'$, where $s$  is a monomial term for which we have proved the claim
and  $t' \in \rmPT(X)$  for which the claim holds by induction hypothesis.
Then, using Definition \ref{def:operations}(a) and that  $f'$  is a homomorphism, we obtain
\begin{align*}
& f'([s \hat{\times} t']_{\rmAC}) = f'([s]_{\rmAC} \times [t']_{\rmAC}) = f'([s]_{\rmAC} \times_\mrd [t']_{\rmAC})
= \\ & f'([s]_{\rmAC}) \otimes_{\mrd} f'([t']_{\rmAC}) = [s]_{\mrd} \otimes_{\mrd} [t']_{\mrd} = [s \hat{\times} t']_{\mrd},
\end{align*}
as claimed. Then we can proceed as for the proof of (a) of our theorem.

(c)  Here we proceed analogously to (a) and (b), but using Theorem \ref{thm:Pol-Polid-are-free}(b) and Proposition \ref{prop:three-free-sb}(c), with the homomorphisms
$f' : \rmB_{\midrd}[X] \to \rmFB_{\midrd}(X)$  and  $g' : \rmFB_{\midrd}(X) \to  \rmB_{\midrd}[X]$.
Now we claim that
\begin{equation}\label{eq:f-for-id-polynomial-terms}
\text{$f'([t]_{\rmAC}) = [t]_{\midrd}$  for each  $t \in \rmPTid(X)$.}
\end{equation}
For the proof we follow Lemma \ref{lm:id-reduced-polynomial-terms} and Definition \ref{def:id-operations}.
\end{proof}

Here, for instance, Theorem \ref{thm:three-implications}(b) says the following.  If a polynomial term $t$  can be transformed into a polynomial term $t'$ via the identities given in $E_{\mrd}$, then we can transform  $t$  into  $t'$ by just using  the identities in $\rmAC$, i.e.,
the identities for associativity and commutativity of addition and associativity of multiplication.
In the term rewriting literature, this is regarded as an AC-reduction result,
cf. \cite{bacpla85,gnales86},  cf. also \cite[Sec. 11.1]{baanip98}.
This is usually achieved by an involved analysis of critical pairs; here we obtained it by algebraic means. 

An immediate consequence of Theorem \ref{thm:three-implications} is the following. If a term  $t \in \T(X)$
is "represented" by a polynomial term $t' \in \rmPT(X)$ meaning that $t =_{E_{\mrd}} t'$, then $t'$ is unique up to $\rmAC$-equivalence. 
Now we will see that such representations exist. 
This provides the analogy of the second statement in Theorem \ref{thm:free-semiring} for representations of terms here by simple terms, by polynomial terms and by id-reduced polynomial terms.

\begin{theorem}\label{thm:three-representations}$\,$
Given  $t \in \T(X)$, there are  $t_1 \in \rmST(X), t_2 \in \rmPT(X)$,  and  $t_3 \in \rmPTid(X)$  with  $t =_E t_1, t =_{E_{\mrd}} t_2$,  and  $t =_{E_{\midrd}} t_3$. 
\end{theorem}

\begin{proof} The simple term  $t_1$  is easy to obtain. In $t$, we replace all occurrences of subterms
$t' \hat{+} \hat{0}$ or  $\hat{0} \hat{+} t'$ by  $t'$, of $t' \hat{\times} \hat{0}$ or  $\hat{0} \hat{\times} t'$ by  $\hat{0}$, 
and of $t' \hat{\times} \hat{1}$ or  $\hat{1} \hat{\times} t'$ by  $t'$, in any order, continuing until all such subterms are eliminated. The resulting term $t_1$ is simple, has size at most  $\size(t)$ and satisfies $t =_E t_1$. 

For the second assertion concerning the existence of $t_2$, we follow the proof of Theorem \ref{thm:three-implications}(b). We consider the two mappings
$f: X/\rmAC \to X/\mrd$  and  $g: X/\mrd \to X/\rmAC$  given by 
$f([x]_\rmAC) = [x]_\mrd$  and   $g([x]_\mrd) = [x]_\rmAC$  for each  $x \in X$.
They extend to homomorphisms
$f' : \bbNrd[X] \to \rmFB_{\mrd}(X)$  and  $g' : \rmFB_{\mrd}(X) \to  \bbNrd[X]$, respectively, 
and  $f'$ satisfies equation (\ref{eq:f-for-polynomial-terms}).
Then $f' \circ g': \rmFB_\mrd(X) \rightarrow \rmFB_\mrd(X)$  is a homomorphism which acts like the identity on  $X/\mrd$. Since  $X/\mrd$  generates $\rmFB_\mrd(X)$, we obtain  that $f' \circ g'$ is the identity on $\rmFB_\mrd(X)$.
Given $t \in \T(X)$, choose  $t_2 \in \rmPT(X)$  with  $g'([t]_\mrd) = [t_2]_\rmAC$.
Then  $[t]_\mrd = (f' \circ g')([t]_\mrd) = f'([t_2]_{\rmAC}) = [t_2]_\mrd$  by equation (\ref{eq:f-for-polynomial-terms}). Thus  $t =_{E_\mrd} t_2$  as claimed.

For the third claim, we start with mappings $f: X/{\rmAC} \to X/\midrd$  and  $g: X/\midrd \to X/{\rmAC}$ as in the proof of Theorem \ref{thm:three-implications}(c) and follow the above argument to obtain homomorphisms 
$f' : \rmB_{\midrd}[X] \to \rmFB_\midrd(X)$  and  $g' : \rmFB_\midrd(X) \to  \rmB_{\midrd}[X]$  such that
$f' \circ g'$  is the identity on  $\rmFB_\midrd(X)$. Given $t \in \T(X)$, choose  $t_3 \in \rmPTid(X)$  with  $g'([t]_\midrd) = [t_3]_{\rmAC}$.
Then  $[t]_\midrd = (f'\circ g')([t]_\midrd) = f'([t_3]_{\rmAC}) = [t_3]_\midrd$  by equation (\ref{eq:f-for-id-polynomial-terms}). Thus  $t =_{E_\midrd} t_3$  as claimed.
\end{proof}


\section{Decision and construction procedures for simple and polynomial terms}
\label{sect:decision-procedures}

In this section, we turn to algorithms for deciding the equivalence of terms modulo our congruences, as indicated by Theorem \ref{thm:three-implications}, and for constructing representations of terms as stated in Theorem \ref{thm:three-representations}. 

First, we give an algorithm for deciding the equivalence of simple terms modulo $\rmAC$-equivalence. For the decision algorithm we will represent simple terms as labeled trees (cf. Section \ref{sect:free-strong-bimonoids}) and employ the well-known result that isomorphism of labeled trees is decidable in linear time (cf. \cite[Sect.~3.2]{ahohopull74}).

\begin{lemma} \label{lm:AC-equiv-decidable} \rm  For every $s,t \in  \rmST(X)$ it is decidable in linear time $O(n)$, 
where  $n = \max\{\size(s), \size(t)\}$, 
whether  $s \Leftrightarrow^*_{\rmAC} t$.
\end{lemma}

\begin{proof} By Proposition \ref{prop:AC-equiv-graphs}, two simple terms  $s$  and  $t$  are  $\rmAC$-equivalent if and only if there is an isomorphism between the labeled trees 
$\overline{s}$ and $\overline{t}$. 
For $s$ and $t$ we can construct the labeled trees 
$\overline{s}$ and $\overline{t}$ in $O(n)$ time. 
Then it can be decided if $\overline{s}$ and $\overline{t}$  are isomorphic in $O(n)$ time by \cite[Cor.~p.86]{ahohopull74}. 
\end{proof}

We will also need the following algorithm for constructing id-reduced terms. This may be known in the literature, but we could not find a reference. We define
\[\rmAC_\mrid=\rmAC\cup\{e_{11}\}.\]

\begin{lemma} \label{lm:AC-equiv-id-reduced} \rm  For each simple term  $s$  with  $n = \size(s)$  we can find in linear time  $O(n)$
an id-reduced term  $t$  with  $s =_{\rmAC_\mrid} t$. 
\end{lemma}

\begin{proof} Given $s$, we construct the labeled tree 
$\overline{s}$ in time $O(n)$. 
Then we traverse the labeled tree $\overline{s}$  as described in \cite[Example 3.2, ~pp. 85,86]{ahohopull74}  
and assign the  tuples and numbers to its nodes also as described, in time  $O(n)$. 

Then we traverse the labeled tree a second time, starting with vertices at level 1 and working up towards the root.  At each node $d$ labeled by  $\boxplus$, we consider the list of numbers $i_1, i_2, ..., i_k$ associated to the children $d_1,d_2,\ldots,d_k$ of $d$, and we check whether  $i_j = i_{j+1}$, for  $j = 1,\ldots,k-1$.
If $i_j = i_{j+1}$, i.e., the subtrees with root  $d_j$ and $d_{j+1}$ are isomorphic, then we delete the subtree having the vertex  $d_{j+1}$  as its root. It is important to note that if two subtrees with root $d$ and $d'$ at a level $i$ are isomorphic before the possible deletion of some of their subtrees, i.e., the same number is assigned to both, then they remain isomorphic after the deletion. So there is no need to change the number assigned to them. The second traversal can also be done in time  $O(n)$. 

Let $T$  be the labeled tree obtained by the above construction.
We construct a simple term  $t$  with  $\overline{t} = T$ by induction as follows. We show only the induction step. If the root of $T$ is $\boxplus$ or $\boxtimes$ with children $T_1,\ldots,T_k$, then by induction hypothesis simple terms $t_1,\ldots,t_k$ can be constructed such that  $\overline{t_i} = T_i$ for each $i\in[k]$.  If the root is $\boxplus$, then each $t_i$ is a product term or an element of $X\cup\{\hat{1}\}$. Then $t$ can be any sum term of which the sum-product decomposition is $t_1\hat{+}\ldots\hat{+}t_k$. If the root is $\boxtimes$, each $t_i$ is a sum term or an element of $X$. Then $t$ can be any product term of which the product-sum decomposition is $t_1\hat{\times}\ldots\hat{\times}t_k$.
This can also be done in time  $O(n)$, resulting in total time complexity $O(n)$ for the algorithm. 

 This algorithm has the effect that for each subterm  $s'$ of $s$  with a sum-product decomposition  $s' = s'_1 \hp \ldots \hp s'_m$, 
 say, we delete all summands  $s'_j$  occurring more than once. Therefore, $s =_{\rmAC_\mrid} t$, as can be shown again by induction on the size of  $s$.
\end{proof}

As a consequence, we obtain algorithms for deciding the equivalence of polynomial terms modulo our congruences   $=_{E_{\mrd}}$  and  $=_{E_{\midrd}}$.

\begin{corollary}\label{cor:polynom-term-equiv-is-decidable}\rm
It is decidable, given  $s, t \in \rmPT(X)$  and  $n = \max\{\size(s), \size(t)\}$, 
in linear time $O(n)$  whether  $s =_{E_{\mrd}} t$  and whether $s =_{E_{\midrd}} t$.
\end{corollary}

\begin{proof} For the first claim, we apply Theorem \ref{thm:three-implications}(b) and Lemma \ref{lm:AC-equiv-decidable}. For the second claim, using Lemma \ref{lm:AC-equiv-id-reduced} we construct $s', t' \in \rmPTid(X)$, each of size at most  $n$, with  
$s =_{E_{\midrd}} s'$  and  $t =_{E_{\midrd}} t'$  in  $O(n)$  steps.
Then   we have  
$s =_{E_{\midrd}} t$ if and only if $s' =_{E_\midrd} t'$ if and only if $s' =_\rmAC t'$, where the latter equivalence follows from Theorem \ref{thm:three-implications}(c). By Lemma \ref{lm:AC-equiv-decidable}, this latter $\rmAC$-equivalence can be decided  with a number of steps linear in the sizes of  $s', t'$.
\end{proof}

Now we turn to algorithms for constructing representations of terms as described in Theorem \ref{thm:three-representations}. We will determine the total number of operations needed. But, since we consider an application of right-distributivity as the possibly 'most complex' operation, we also describe the possible number of uses of right-distributivity, i.e.,
the application of the identity  
\[e_9: \big((z_1 \hat{+} z_2) \hat{\times} z_3 \ , \ (z_1 \hat{\times} z_3) \hat{+} (z_2 \hat{\times} z_3)\big).\]
In fact, we will apply the identity  $e_9$  only by replacing a subterm of the form  
$(t_1 \hat{+} t_2) \hat{\times} t_3$ by $(t_1 \hat{\times} t_3) \hat{+} (t_2 \hat{\times} t_3)$.

First we prove the following auxiliary result.

\begin{lemma}\label{lm:rd-reduction} \rm There is an effective procedure for constructing, given  $s \in \rmST(X)$ and $t \in \rmPT(X)$, a polynomial term  $t' \in \rmPT(X)$  with  $s \hti t =_{E_{\mrd}} t'$ and $\size(t') \leq 2^{\size(s)} \cdot \size(t)$  in  $O(2^{\size(s)})$ steps, using at most  $\size(s)$ applications of right-distributivity.
\end{lemma}

\begin{proof} We proceed by induction on  $\size(s)$. If  $s \in X$, the result is trivial, as $t' = s \hti t$  is already a polynomial term.

Let us first assume that  $s$  is a sum term, so  $s = s_1 \hp s_2$. 
Then  $s \hti t =_{E_\mrd} (s_1 \hti t) \hp (s_2 \hti t)$. 
By induction hypothesis, we can construct polynomial terms $t'_i \in \rmPT(X)$  
with  $s_i \hti t =_{E_{\mrd}} t'_i$ and $\size(t'_i) \leq 2^{\size(s_i)} \cdot \size(t)$  
in  $O(2^{\size(s_i)})$ steps, using at most  $\size(s_i)$ applications of right-distributivity, for  $i = 1,2$.
Put  $t' = t'_1 \hp t'_2$. Then  $t' \in \rmPT(X)$, 
and we have  $\size(t') = \size(t'_1) + \size(t'_2) + 1 \le (2^{\size(s_1)} + 2^{\size(s_2)}) \cdot \size(t) + 1 \le 2^{\size(s)} \cdot \size(t)$. 
Clearly, we needed at most  $O(2^{\size(s)})$ steps, and our construction involved at most
$1 + \size(s_1) + \size(s_2) = \size(s)$  applications of right-distributivity.

Second, assume that $s$  is a product term, say  $s = s_1 \hti s_2$.
By our induction hypothesis, we can construct a polynomial term $t'' \in \rmPT(X)$  
with  $s_2 \hti t =_{E_{\mrd}} t''$ and $\size(t'') \leq 2^{\size(s_2)} \cdot \size(t)$  
in  $O(2^{\size(s_2)})$ steps, using at most  $\size(s_2)$ applications of right-distributivity.
Then, again by our induction hypothesis, we can construct a polynomial term $t' \in \rmPT(X)$  
with  $s_1 \hti t'' =_{E_{\mrd}} t'$ and 
$\size(t') \leq 2^{\size(s_1)} \cdot \size(t'') 
\le 2^{\size(s_1)} \cdot 2^{\size(s_2)} \cdot \size(t) \le 2^{\size(s)} \cdot \size(t)$  
in  $O(2^{\size(s_1)})$ steps, using at most  $\size(s_1)$ applications of right-distributivity.
Then we have  $t' =_{E_{\mrd}} s_1 \hti t'' 
=_{E_{\mrd}} s_1 \hti (s_2 \hti t) =_\rmAC (s_1 \hti s_2) \hti t = s \hti t$, 
and we have obtained  $t'$ in at most  $O(2^{\size(s)})$ steps, using at most $\size(s)$ applications of right-distributivity.  
\end{proof}

Now we can show:

\begin{theorem}\label{thm:three-effective-representations}$\,$
There are effective procedures for constructing, given  $t \in \T(X)$, terms
\begin{compactitem}
\item[(a)] $t_1 \in \rmST(X)$ with $t =_E t_1$  and $\size(t_1) \leq \size(t)$  using at most  $O(\size(t))$  steps, 
\item[(b)] $t_2 \in \rmPT(X)$  with  $t =_{E_{\mrd}} t_2$ and $\size(t_2) \leq 2^{\size(t)}$  in  $O(2^{\size(t)})$ steps using at most  $\size(t)$ applications of right-distributivity, and
\item[(c)] $t_3 \in \rmPTid(X)$  with  $t =_{E_{\midrd}} t_3$ and  $\size(t_3) \leq 2^{\size(t)}$  
in $O(2^{\size(t)})$  steps using at most  $\size(t)$ applications of right-distributivity.
\end{compactitem}
\end{theorem}

\begin{proof} (a) For constructing $t_1$, we can proceed as described at the beginning of the proof of Theorem \ref{thm:three-representations}. This needs at most  $O(\size(t))$ steps.  

(b) Using (a), first we construct the simple term $t_1$ with $t =_E t_1$. 
Then we also have $t =_{E_\mrd} t_1$. Hence we may assume without loss of generality that 
$t \in \rmST(X)$.

Then we proceed by induction on the size of  $t$. If  $t \in X \cup \{\hat{0}, \hat{1} \}$, the result is trivial. Therefore we can assume that  $t$  is a sum term or a product term.

First, assume that  $t$  is a sum term, say,  $t = t_1 \hp t_2$. By induction hypothesis, we can construct polynomial terms $t'_i \in \rmPT(X)$  
with  $t_i =_{E_{\mrd}} t'_i$ and $\size(t'_i) \leq 2^{\size(t_i)}$  
in  $O(2^{\size(t_i)})$ steps, using at most  $\size(t_i)$ applications of right-distributivity, for  $i = 1,2$. Then with  $t' = t'_1 \hp t'_2 \in \rmPT(X)$ we obtain the result.

Second, assume that  $t$  is a product term, say,  $t = t_1 \hti t_2$.
By our induction hypothesis, we can construct a polynomial term $t'' \in \rmPT(X)$  
with  $t_2 =_{E_{\mrd}} t''$ and $\size(t'') \leq 2^{\size(t_2)}$  
in  $O(2^{\size(t_2)})$ steps, using at most  $\size(t_2)$ applications of right-distributivity.
Then, by Lemma \ref{lm:rd-reduction} (note that $t_1 \in \rmST(X)$), we can construct a polynomial term  $t' \in \rmPT(X)$ with  $t_1 \hti t'' =_{E_{\mrd}} t'$ and 
$\size(t') \leq 2^{\size(t_1)} \cdot \size(t'') 
\le 2^{\size(t_1)} \cdot 2^{\size(t_2)} \le 2^{\size(t)}$  in  $O(2^{\size(t_1)})$ steps, using at most  $\size(t_1)$ applications of right-distributivity.
In total, our construction of  $t'$  used at most $O(2^{\size(t)})$ steps with at most  $\size(t)$ applications of right-distributivity, and  $t =_{E_\mrd} t_1 \hti t'' =_{E_\mrd} t'$. 

(c) Similarly as in (b), we may assume without loss of generality that $t \in \rmST(X)$. First, as in (b), we construct  the polynomial term $t_2$ from $t$.
Then we apply Lemma \ref{lm:AC-equiv-id-reduced} to obtain in time  $O(\size(t_2))$ 
an id-reduced term  $t_3$  with  $t_2 =_{\rmAC_\mrid} t_3$  and  $\size(t_3) \le \size(t_2)$.
Then  $t_3$  is an id-reduced polynomial term. 
\end{proof}

Observe that the above proof immediately gives rise to an inductive algorithm for computing a polynomial term  $t' \in \rmPT(X)$ which is  $=_{E_\mrd}$-equivalent to a given term  $t \in \rmST(X)$,  as follows, by iterating (1) and (2):

(1) If  $t$  is a sum term with sum-product decomposition  $t = t_1 \hat{+} \ldots \hat{+} t_n$, apply the algorithm to each term $t_i$ separately and take the sum of the resulting polynomial terms.

(2)  If $t$  is a product term with product-sum decomposition  $t = t_1 \hti \ldots \hti t_n$, first find a polynomial term  $t'_n \in \rmPT_(X)$ with $t_n =_{E_\mrd} t'_n$. Then follow the proof of Lemma \ref{lm:rd-reduction} to find,
successively, polynomial terms  $t'_{n-1},\ldots,t'_1 \in \rmPT(X)$  with
$t'_i =_{E_\mrd} t_i \hti t'_{i+1}$, for  $i = n-1,\ldots,1$. Here, replace each possibly arising term of the form  $\hat{1} \hti s$  by  $s$. Then we have  $t =_{E_\mrd} t'_1$.

This algorithm will employ at most  $\size(t)$  applications of right-distributivity.

As a consequence, we obtain algorithms for deciding the equivalence of terms modulo our congruences, as indicated by Theorem \ref{thm:three-implications}. 
For deciding the $\rmAC$-equivalence of simple terms, we have already obtained in Lemma \ref{lm:AC-equiv-decidable} a linear time algorithm.

\begin{corollary}\label{cor:equiv-is-decidable}\rm
It is decidable, given  $s, t \in \T(X)$  and  $n = \max\{\size(s), \size(t)\}$, 
in linear time $O(n)$  whether $s =_E t$, and in exponential time $O(2^n)$  
whether $s =_{E_{\mrd}} t$  and whether $s =_{E_{\midrd}} t$.
\end{corollary}

\begin{proof}  For the first claim, given  $s, t \in \T(X)$, use Theorem \ref{thm:three-effective-representations} to construct
$s', t' \in \rmST(X)$ with $s =_E s'$  and  $t =_E t'$  in  $O(n)$  steps. By Theorem \ref{thm:three-implications}(a) it follows that  $s =_E t$ if and only if $s' =_E t'$ if and only if $s' =_\rmAC t'$. The latter $\rmAC$-equivalences can be decided in $O(n)$ steps by using Lemma \ref{lm:AC-equiv-decidable}. 

For the second claim, we proceed analogously. By Theorem \ref{thm:three-effective-representations},
we construct  $s', t' \in \rmPT(X)$, each of size at most  $2^n$, with  $s =_{E_{\mrd}} s'$ and $t =_{E_\mrd} t'$    in  $O(2^n)$  steps.
By Theorem \ref{thm:three-implications}(b) we have  
$s =_{E_{\mrd}} t$ if and only if $s' =_{E_\mrd} t'$ if and only if $s' =_\rmAC t'$. The latter $\rmAC$-equivalences
can be decided by Lemma \ref{lm:AC-equiv-decidable} with a number of steps linear in the sizes of  $s'$ and $t'$.

The proof of the third claim can be obtained from the proof of the second one by replacing $=_{E_\mrd}$ and Theorem \ref{thm:three-implications}(b) by $=_{E_\midrd}$ and Theorem \ref{thm:three-implications}(c), respectively.
\end{proof}

Unfortunately, as is well-known from the theory of terms and polynomials over  $\mathbb{N}$, the above decision algorithm needs exponentially many steps (in the size of the given terms). This also applies here, using only right-distributivity, The above algorithm rests on the construction of a polynomial equivalent to a given term.
Now, let  $t_1 = \hat{1} \hat{+} \hat{1}$ and inductively, 
$t_{n+1} = (\hat{1} \hat{+} \hat{1}) \hat{\times} t_n$ for each $n \geq 1$. With our algorithm, we can transform  $t_n$  into a sum term with  $2^n$ summands of  $\hat{1}$. Since  $t_n$  has size $3 + 4 \times (n-1)$, this algorithm employs  $O(2^n)$  many steps.

Next, we wish to further investigate effective procedures for constructing polynomial terms as in Theorem \ref{thm:three-effective-representations}.

Given $t \in \T(X)$, perform the following algorithm for constructing a polynomial term  
$t' \in \rmPT(X)$  with  $t =_{E_\mrd} t'$. 

(1) Eliminate additions of $\hat{0}$ and multiplications with $\hat{0}$ or $\hat{1}$, as in the proof of Theorem \ref{thm:three-representations}.

(2) Replace a subterm of the form  
$(t_1 \hat{\times} t_2) \hat{\times} t_3$ by $t_1 \hat{\times} (t_2 \hat{\times} t_3)$.

(3) Replace a subterm of the form
$(t_1 \hat{+} t_2) \hat{\times} t_3$ by $(t_1 \hat{\times} t_3) \hat{+} (t_2 \hat{\times} t_3)$.

(4) Iterate (1), (2) and (3) as far as possible. 

We will see that this process terminates (cf. Lemma \ref{lm:termination}). It transforms the term $t$ into a simple term $t'$ which has no subterms of the form  $(t_1 \hat{\times} t_2) \hat{\times} t_3$ or $(t_1 \hat{+} t_2) \hat{\times} t_3$.
By Lemma \ref{lm:rd-polynomial} shown below, then  $t' \in \rmPT(X)$. Since these reductions imply $=_{E_{\mrd}}$-congruency, we obtain $t =_{E_{\mrd}} t'$ as claimed.

\begin{lemma}\label{lm:rd-polynomial}\rm
Let  $t \in \rmST(X)$. If $t$ has neither subterms of the form  
$(t_1 \hat{\times} t_2) \hat{\times} t_3$ nor subterms of the form $(t_1 \hat{+} t_2) \hat{\times} t_3$, 
then $t \in \rmPT(X)$.
\end{lemma}

\begin{proof} We proceed by induction on the size of $t$. For this, let $t \in \rmST(X)$ which satisfies the condition of the lemma.

If $t\in \{\hat{0},\hat{1}\}$, then the statement is clear by Definition \ref{def:polynomial-terms}(1).

Let $t=u_1\hp u_2$. Then both  $u_1$ and  $u_2$
are in $\rmST(X)$ and they satisfy the condition of the lemma. Hence, by induction hypothesis,  $u_1,u_2\in  \rmPT(X)$ and thus, by Definition 
\ref{def:polynomial-terms}(2), $t \in \rmPT(X)$.

Lastly, let $t=u_1\hti u_2$. Since $t$ satisfies the condition of the lemma, $u_1$ is neither a sum term nor a product term, i.e., $u_1\in X$. Moreover,  by induction hypothesis, $u_2\in  \rmPT(X)$. Then
by Definition  \ref{def:polynomial-terms}(3), $t \in \rmPT(X)$.
\end{proof}

Next, we wish to show that in the above algorithm the replacements of (1), (2) and (3) can be done in any order and will terminate. Moreover, then we will obtain a \emph{unique} polynomial term, which can therefore be considered as the normal form of the term  $t$,
as will shown below in Theorem \ref{thm:polynomial-normal-form}.

For this, we recall some concepts from abstract reduction systems \cite[Sec.~2]{baanip98} and term rewriting  \cite[Sec.~4-6.]{baanip98}, see also \cite{klo92}. 

Let $A$ be a set and $\to$ a binary relation on $A$. An element $a\in A$ is a \emph{normal form} (with respect to $\to$) if there does not exist $b\in A$ such that $a\to b$.
For every $a,b\in A$, if $a\to^* b$ and $b$ is a normal form, then  \emph{$b$ is a normal form of $a$} (with respect to $\rightarrow$). We say that $\to$ is \emph{terminating } if there does not exist a family $(a_n\mid n\in\mathbb{N})$ of elements of 
$A$ such that $a_n \,\to\, a_{n+1}$ for each $n\in \mathbb{N}$. Moreover, $\to$ is \emph{confluent} if, for every $a,b_1,b_2\in A$ with $a\to^* b_1$ and $a\to_R^* b_2$, there exists $c\in A$ such that 
$b_1\to^* c$ and $b_2\to^* c$.

Now let $\Sigma$ be an arbitrary signature. 
We define the {\em set of positions} of terms by a mapping $\pos$ from $\T_\Sigma(Z)$ to the collection of finite subsets of $\mathbb{N}_+^*$ such that 
\begin{compactitem}
\item[(i)] for each $t \in (\Sigma^{(0)}\cup Z)$ let $\pos(t)=\{\varepsilon\}$ and 
\item[(ii)] for every $t=\sigma(t_1,\ldots,t_k)$ with $k \in \mathbb{N}_+$, $\sigma\in \Sigma^{(k)}$, and $t_1,\ldots,t_k\in \T_\Sigma(Z)$, let $\pos(t)=\{\varepsilon\}\cup\{iv \mid i \in [k], v \in \pos(t_i)\}$. 
\end{compactitem}

Then, let $t,u \in \T_\Sigma(Z)$ and $w \in \pos(t)$.
We define  the {\em subterm of $t$ at $w$}, denoted by $t|_w$ and the {\em replacement of the subterm of $t$ at $w$ by $u$}, denoted by $t[u]_w$, by structural induction as follows: 
\begin{compactitem}
\item[(i)] if $t \in (\Sigma^{(0)} \cup Z)$, then  $t|_\varepsilon = t$, and 
$t[u]_\varepsilon=u$, 
\item[(ii)] for every $t=\sigma(t_1,\ldots,t_k)$ with $k \in \mathbb{N}_+$, $\sigma \in \Sigma^{(k)}$, and $t_1,\ldots,t_k \in \T_\Sigma(Z)$,
  we let  $t|_\varepsilon = t$ and $t[u]_\varepsilon = u$, and for every $i \in [k]$ and $w' \in \pos(t_i)$, we define 
\[ t|_{iw'}=t_i|_{w'} \text{\ \ and \ \ } t[u]_{iw'}=\sigma(t_1,\ldots,t_{i-1},t_i[u]_{w'}, t_{i+1}, \ldots, t_k) .
\]
\end{compactitem}

By a \emph{substitution} we mean a mapping $\varphi: Z\to \T_\Sigma(Z)$. Such a mapping $\varphi$ extends uniquely to a $\Sigma$-algebra homomorphism from
$\mathsf{T}_\Sigma(Z)$ to itself. We denote this extension also by $\varphi$. The name substitution is due to the fact that for each $t \in \T_\Sigma(Z)$, the term $\varphi(t)$ is obtained by substituting $\varphi(z)$ for $z$ in $t$, for each $z\in Z$.

A  \emph{term rewriting system} over $\Sigma$ is a set $R$ of $\Sigma$-identities over $Z$ (cf. Section \ref{sec:preliminaries}) such that for each $(\ell,r)\in R$, all variables of $r$ occur in $\ell$.  We call the identities in $R$ \emph{rules} and we write a rule $(\ell,r)$ in the form $\ell \to r$. The reduction relation induced by $R$ on $\T_\Sigma(Z)$ (the adaptation of the corresponding concept in Section \ref{sec:preliminaries} for $E=R$ and $A=\T_\Sigma(Z)$) is the binary  relation $\Rightarrow_R$ on $\T_\Sigma(Z)$ defined as follows:  for every $t_1,t_2 \in \T_\Sigma(Z)$, we let $t_1 \Rightarrow_R t_2$ if there exist a position $w \in \pos(t_1)$, a rule $\ell \to r$ in $R$,  a substitution $\varphi: Z \to \T_\Sigma(Z)$ such that
$t_1|_w = \varphi(\ell)$ and $t_2 = t_1[\varphi(r)]_w$.

A term rewriting system $R$ is terminating  (respectively, confluent) if the relation $\Rightarrow_R$ is terminating  (respectively, confluent).

\begin{theorem}\label{thm:terminating-and-confluent}
\cite[Lm.~2.18]{baanip98} Let $R$ be a term rewriting system which is terminating and confluent.  Then each element of $\T_\Sigma(Z)$ has a unique normal form. \hfill$\Box$
\end{theorem}

Now we consider again how to construct, for a given term $t\in \T(X)$, a polynomial term  
$t' \in \rmPT(X)$  with  $t =_{E_\mrd} t'$. 

For this, we use a term rewriting system over the signature $\Sigma_{\mathrm{sb}}\cup X$, where we consider the elements of $X$ as nullary symbols.
We write $\T(X, Z)$ for $\T_{\Sigma_{\mathrm{sb}}\cup X}(Z)$, hence $\T(X)\subset \T(X,Z)$.

We consider the following term rewriting system
\begin{align*}
\cR=\{ & \rho_1\,:\,\hat{0} \hat{+} z \to z,\ \ \rho_2\,:\,z \hat{+} \hat{0} \to z,\  \\ 
& \rho_3\,:\,\hat{1} \hat{\times} z \to z,\ \ \rho_4\,:\,z \hat{\times} \hat{1} \to z \ \ \rho_5\,:\,\hat{0} \hat{\times} z \to \hat{0},\ \ \rho_6\,:\,z \hat{\times} \hat{0} \to \hat{0}, \\
& \rho_7\,: (z_1 \hat{\times}  z_2) \hat{\times} z_3 \to  z_1 \hat{\times}  (z_2 \hat{\times} z_3) \\
& \rho_8\,:\,(z_1 \hat{+} z_2)\hat{\times} z_3 \to (z_1 \hat{\times} z_3) \hat{+} (z_2\hat{\times} z_3)
 \}\enspace.
\end{align*}

Clearly, for every $s,t\in \T(X)$, the relation $s\Rightarrow^*_\cR t$ implies that $s=_{E_{\mrd}} t$. We will use this fact without further reference.
We will show that $\cR$ is terminating and confluent. 

We begin with the proof of termination. For the proof we employ the method 
based on reduction order and monotone polynomial interpretation, see e.g. Sections 5.2 and 5.3 of \cite{baanip98}. Since we do not need the full power of that method,  we do not recall all concepts and results in their general form.

Our aim is to define a mapping 
$|\;.\;|: \T(X, Z)\to \mathbb{N}\setminus\{0,1\}$ such that, 
for every $s,t \in \T(X, Z)$, if $s\Rightarrow_\cR t$, then $|s| > |t|$.
Then, for every $s,t \in \T(X, Z)$, if  $s \Rightarrow^+_\cR t$, then 
 $|s| > |t|$, hence there does not exist an infinite sequence $s_1 \Rightarrow_\cR s_2 \Rightarrow_\cR \ldots$ of reductions. Hence $\cR$ is terminating.

We define $|\;.\;|$ by induction as follows: for every $y\in X\cup Z\cup\{\hat{0},\hat{1}\}$, we let $|y|=2$, and for terms $t_1, t_2\in \T(X)$, we let $|t_1\hat{+}t_2|=|t_1|+|t_2|$ and 
$|t_1\hat{\times} t_2|=|t_1|^2|t_2|$. 

Next we show two properties of the mapping $|\;.\;|$. In the proof we will use the fact that $|t|>1$ for each $t\in \T(X)$ without any reference.

\begin{lemma}\rm \label{lm:prec-compatible-R-substitution}\hspace{1cm}
\begin{compactitem}
\item[(a)] For every rule
 $\ell\to r$ of $\cR$ and substitution $\varphi: Z \to \T(X, Z)$, we have $|\varphi(\ell)| > |\varphi(r)|$.
 \item[(b)] For every $s,t,u \in \T(X, Z)$, if $|s| > |t|$, then $|u\hat{+}s| > |u\hat{+}t|$, $|s\hat{+}u| > |t\hat{+}u|$, $|u\hat{\times}s| > |u\hat{\times}t|$, and $|s\hat{\times}u| >  |t\hat{\times}u|$.
\end{compactitem}
\end{lemma}
\begin{proof} 
The proof for rules  $\rho_1-\rho_6$ is omitted.
For rules $\rho_7$ and $\rho_8$,  let $\varphi(Z)\to \T(X, Z)$ be a substitution with $\varphi(z_i)=t_i$ for each $i\in\{1,2,3\}$. Then
\begin{align*}
|\varphi((z_1\hat{\times}z_2)\hat{\times} z_3)| =  |(t_1\hat{\times}t_2)\hat{\times} t_3|= |t_1|^4|t_2|^2|t_3| 
>   |t_1|^2|t_2|^2|t_3| = 
|t_1\hat{\times} (t_2\hat{\times}t_3)|=|\varphi(z_1\hat{\times} (z_2\hat{\times}z_3))|,
\end{align*}
and
\begin{align*}
& |\varphi((z_1\hat{+}z_2)\hat{\times} z_3)| =  |(t_1\hat{+}t_2)\hat{\times} t_3|= |t_1\hat{+}t_2|^2|t_3| =  (|t_1|+|t_2|)^2|t_3| 
>  \\ & |t_1|^2|t_3|+|t_2|^2|t_3| = 
|(t_1\hat{\times} t_3)\hat{+}(t_2\hat{\times} t_3)| = |\varphi((z_1\hat{\times} z_3)\hat{+}(z_2\hat{\times} z_3))|,
\end{align*}
which proves, respectively,  that (a) holds for rules $\rho_7$ and $\rho_8$. 

The proof of (b) is obvious by the definition of $|\;.\;|$.
\end{proof}

\begin{lemma}\label{lm:termination}\rm
The term rewriting system $\cR$ is terminating.
\end{lemma}

\begin{proof} By Lemma \ref{lm:prec-compatible-R-substitution}, it follows that for every $s,t \in \T(X, Z)$, if $s\Rightarrow_\cR t$, then $|s| > |t|$. Hence $\cR$ is terminating.
\end{proof}

Now we turn to the proof of confluency. For this, we recall the concept of a critical pair, cf. \cite[Def.~6.2.1]{baanip98}. Let again $\Sigma$ be an arbitrary signature. For the definition of the critical pair,  we need the following concepts.

Let $t_1, t_2 \in \T_\Sigma(Z)$. A substitution $\varphi:Z \to \T_\Sigma(Z)$ is a \emph{unifier of $t_1$ and $t_2$} if $\varphi(t_1) = \varphi(t_2)$. A unifier $\varphi$ of $t_1$ and $t_2$ is a \emph{most general unifier of $t_1$ and $t_2$} (for short: mgu) if, for each unifier $\psi$ of $t_1$ and $t_2$ there exists a substitution $\theta:Z \to \T_\Sigma(Z)$ such that $\psi = \theta \circ \varphi$. 

Let $\ell_1 \to r_1$ and $\ell_2 \to r_2$ be two rules of some term rewrite system $R$ whose variables have been renamed such that $\ell_1$ and $\ell_2$ have no common variables.  If there exists a $w \in \pos(\ell_1)$ such that $\ell_1(w) \not\in Z$, and there exists a most general unifier $\varphi: Z \to \T_\Sigma(Z)$ of $\ell_1|_w$ and $\ell_2$, then we say that the two rules \emph{overlap}. In this case these objects determine the \emph{critical pair} $\langle\varphi(r_1), \varphi(\ell_1)[\varphi(r_2)]_w\rangle$ of $R$.

\begin{proposition}\label{prop:no-critical-pair-implies-loc-confluent}\rm (cf. \cite[Cor.~6.2.5]{baanip98}) A terminating term  rewriting system $R$ is confluent iff for each  $\langle t_1,t_2\rangle$  of its critical pairs there exists a $t \in \T_\Sigma(Z)$ such that $t_1 \Rightarrow_R^* t$ and  $t_2 \Rightarrow_R^* t$. \hfill$\Box$
\end{proposition}

Now we return to our particular term rewriting system $\cR$ and show the following result.

\begin{lemma}\label{lm:local-confluence}\rm\rm$\,$
The term rewriting system $\cR$ is confluent.
\end{lemma}
\begin{proof} By Lemma \ref{lm:termination}, it suffices to show that all critical pairs of $\cR$ are joinable in the sense of Proposition \ref{prop:no-critical-pair-implies-loc-confluent}.

In the following table we show some critical pairs of $\cR$. In the first column we show the rules $\rho_i$ and $\rho_j$ which we consider for a critical pair, where $i,j\in \{1,\ldots,8\}$. By $\ell_i$ and $r_i$ we denote the left-hand side and the right-hand side of the rule $\rho_i$, respectively. 
We assume that the variables $z_1,z_2$, and $z_3$ in $\rho_j$ are renamed $y_1,y_2$, and $y_3$, respectively (cf. the definition of critical pair). In the second, third, and fourth column we show the position $w$ at which $\ell_i|_w$ and $\ell_j$ overlap, the most general unifier, and the critical pair, respectively.

\

\begin{tabular}{l|l|l|l}
rules $\rho_i$ and &  &  &  \\
 renamed $\rho_j$ & overlap & mgu $\varphi$ & critical pair $\langle\varphi(r_i), \varphi(\ell_i)[\varphi(r_j)]_w\rangle$ \\\hline
& & &\\
$\rho_7, \rho_7$ & $\ell_7|_1,\ell_7$ & $z_1\mapsto y_1\hat{\times} y_2$
& $\langle(y_1\hat{\times}y_2)\hat{\times}(y_3\hat{\times}z_3)$, $\big((y_1\hat{\times}y_2)\hat{\times}y_3\big)\hat{\times}z_3)\rangle$\\
& & $z_2 \mapsto y_3$& \\\hline
& & & \\
$\rho_7, \rho_8$ & $\ell_7|_1,\ell_8$ & $z_1\mapsto y_1\hat{+} y_2$ & $\langle(y_1\hat{+} y_2)\hat{\times}(y_3\hat{\times}z_3)$, $\big((y_1\hat{\times}y_3)\hat{+}(y_2\hat{\times}y_3)
\big)\hat{\times}z_3\rangle$\\
& & $z_2 \mapsto y_3$& 
\end{tabular}

\


It is easy to check that for each  $\langle t_1,t_2\rangle$  of the above critical pairs there exists a $t \in \T(X, Z)$ such that $t_1 \Rightarrow_\cR^* t$ and  $t_2 \Rightarrow_\cR^* t$.
For instance, let us consider the critical pair for $\rho_7, \rho_8$.  Then
\[(y_1\hat{+} y_2)\hat{\times}(y_3\hat{\times}z_3) \Rightarrow_\cR y_1 \hat{\times}(y_3\hat{\times}z_3) \hat{+} y_2 \hat{\times}(y_3\hat{\times}z_3)\Rightarrow_\cR^2 (y_1 \hat{\times}y_3)\hat{\times}z_3 \hat{+} (y_2 \hat{\times}y_3)\hat{\times}z_3
\]
by using rules $\rho_8$ and $\rho_7$ twice, and 
\[\big((y_1\hat{\times}y_3)\hat{+}(y_2\hat{\times}y_3)
\big)\hat{\times}z_3  \Rightarrow_\cR (y_1 \hat{\times}y_3)\hat{\times}z_3 \hat{+} (y_2 \hat{\times}y_3)\hat{\times}z_3\]
by using rule $\rho_8$.

The proof for the other critical pairs is left to the reader.
\end{proof}

As an immediate consequence, we obtain the following normal form result for representing arbitrary terms by polynomial terms.

\begin{theorem}\label{thm:polynomial-normal-form} For each term  $t \in \T(X)$, there is a unique polynomial term  $t' \in \rmPT(X)$  in normal form with respect to $\Rightarrow_\cR$ such that  $t \Rightarrow_\cR^* t'$, in particular  $t =_{E_\mrd} t'$.
\end{theorem}

\begin{proof} Let $t \in \T(X)$. By Lemmas 
\ref{lm:termination} and \ref{lm:local-confluence}, and Theorem \ref{thm:terminating-and-confluent}, there is a unique term $t'\in \T(X)$ in normal form with respect to $\Rightarrow_\cR$.
Due to the shape of the rules in $\cR$, $t' \in \rmST(X)$ and  $t'$ has neither subterms of the form  
$(t_1 \hat{\times} t_2) \hat{\times} t_3$ nor subterms of the form $(t_1 \hat{+} t_2) \hat{\times} t_3$.   Then $t' \in \rmPT(X)$ by Lemma  \ref{lm:rd-polynomial}. 
\end{proof}

We note that different reduction strategies by $\cR$
may involve different numbers of applications of the right-distributivity rule  $\rho_7$. As example, we use the term given after Corollary \ref{cor:equiv-is-decidable}. 

\begin{example}\label{ex:number-of-right-distributivity}\rm 
Let  $t_1 = \hat{1} \hat{+} \hat{1}$ and inductively, 
$t_n = (\hat{1} \hat{+} \hat{1}) \hat{\times} t_{n-1}$ for each $n \geq 2$.
Fix some  $n \geq 2$. Then $t_n$  is a product of  $n$  factors of $\hat{1} \hat{+} \hat{1}$.

In our first strategy, we apply right-distributivity always to the left-most factor.
We claim that then for  $t_n$  we will use  $2^{n-1} - 1$  applications of right-distributivity to obtain the equivalent polynomial term. We have:
\[t_n \Rightarrow_\cR (\hat{1} \hti t_{n-1}) \ \hp \ (\hat{1} \hti t_{n-1}) \Rightarrow_\cR^2 t_{n-1} \hp t_{n-1}\enspace.\]

By induction, for  $t_{n-1}$  our method employs  $2^{n-2} - 1$  applications of right-distributivity to obtain the equivalent polynomial term.
Substituting these for both  $t_{n-1}$, we obtain a polynomial term equivalent to  $t_n$ 
(a sum term with  $2^n$ summands of  $\hat{1}$).

In total then we have  $1 + 2(2^{n-2} - 1) = 2^{n-1} - 1$  applications of right-distributivity, as claimed. 
Since  $t_n$  has size $3 + 4 \times (n-1)$, 
this is exponential in the size of  $t_n$.

In our second strategy, we follow the algorithm described after the proof of Theorem \ref{thm:three-effective-representations}. First, note that the final factor  
$(\hat{1} \hat{+} \hat{1})$  is already a polynomial term. 
Now we can apply right-distributivity to the next factor on the left and rule  $\rho_3$, and we obtain
\begin{align*}
t_n & \Rightarrow_\cR^* t_{n-2} \hti \big((\hat{1} \hat{+} \hat{1})\ \hti \ (\hat{1} \hat{+} \hat{1})\big) 
\Rightarrow_\cR^3  t_{n-2} \hti \big((\hat{1} \hat{+} \hat{1}) \hp (\hat{1} \hat{+} \hat{1})\big)
\end{align*}
In contrast to the previous sum term, here we have obtained again a product term, now with  $n-1$  factors.
By an induction hypothesis, this employs  $n-2$  further applications of right-distributivity
(just apply right-distributivity successively from the right to the left).

In total, we need  $n-1$  applications of right-distributivity to find the (same) equivalent polynomial term.
This number is bounded by the size of  $t_n$.\hfill$\Box$
\end{example}

Next we wish to derive a result like Theorem \ref{thm:polynomial-normal-form} for the idempotent case, i.e., for a construction of an id-reduced polynomial term as in Theorem \ref{thm:three-representations}. For this, in order to simplify, e.g., a term like $s\hp (t \hp s)$ to $s\hp t$, our term rewriting rules will need to include rules for associativity and commutativity. However, as is well-known, an inclusion of both implications taken from the identities  $e_1, e_2$ leads to a term rewriting systems which is (obviously) not terminating. Therefore, as is standard in the literature, we will employ $R/E$ rewriting, where $R$ is a term rewriting system and $E$ is a set of identities, see e.g. \cite{joukir84} and \cite{bacpla85}. The relation $\Rightarrow_{R/E}$, called \emph{$R$-rewriting modulo $E$}, is defined by  $\Rightarrow_{R/E} \; = \;  \Leftrightarrow^*_E\circ \Rightarrow_R \circ \Leftrightarrow^*_E$. We say that \emph{$R$ is $E$-terminating} if the relation $\Rightarrow_{R/E}$ is terminating.

Let 
\[\rho_9\,:\,z \hat{+} z \to z
\text{\ \ and\ \ } \cR_{\mrid} = \cR \cup \{\rho_9\}.\]
Clearly, the term rewriting systems  $\cR$  and  $\cR_{\mrid}$  are not  $\rmAC$-terminating, as can be seen by considering $\rho_7 \in \cR$ and $e_4 \in \rmAC$. Therefore, subsequently we consider the set $\rmAC_+ = \{ e_1, e_2 \}$. Also, let  $\rmAC_{+,\mrid} = \{ e_1, e_2, e_{11} \}$.

\begin{lemma}\label{lm:idempotent-termination}\rm\rm$\,$
The term rewriting system $\cR_{\mrid}$ is $\rmAC_+$-terminating on $\T(X)$.
\end{lemma}

\begin{proof} Let $s,t \in \T(X)$
such that $s \Rightarrow_{\cR_{\mrid}/\rmAC_+} t$. There exist
 $s',t' \in \T(X)$  with  
$s =_{\rmAC_+} s' \Rightarrow_{\cR_{\mrid}} t' =_{\rmAC_+} t$. 
Then  $|s| = |s'| > |t'| = |t|$. Hence, $\cR_{\mrid}/\rmAC_+$ is terminating.
\end{proof}

Next, we derive a confluence type result for the relation $\Rightarrow^*_{\cR_{\mrid}/\rmAC_+}$.

\begin{lemma}\label{lm:idempotent-confluence}\rm Let $s,t,u \in \rmST(X)$  be such that  $u \Rightarrow^*_{\cR_{\mrid}/\rmAC_+} s$
and $u \Rightarrow^*_{\cR_{\mrid}/\rmAC_+} t$.
Then there are $s'', t'' \in \rmPTid(X)$ in $\cR$-normal form such that
$s \Rightarrow^*_{\cR_{\mrid}/\rmAC_+} s''$ and
$t \Rightarrow^*_{\cR_{\mrid}/\rmAC_+} t''$, and $s'' =_{\rmAC_+} t''$.
\end{lemma}

\begin{proof} By Theorem \ref{thm:polynomial-normal-form}, we find polynomial terms $s', t' \in \rmPT(X)$  in  $\cR$-normal form such that $s \Rightarrow_\cR^* s'$  and  $t \Rightarrow_\cR^* t'$. Now, by using $\rmAC_+$ we can re-order the summands of sum-product decompositions of  $s'$ and $t'$ and of their subterms so that identical summands are next to each other. 
Then we can apply rule  $\rho_9$  to delete multiple summands.
Here, it may happen that we reduce a sum of the form  $(\hat{1} \hp \ldots \hp \hat{1})$  to  $\hat{1}$.

In case the sum  $(\hat{1} \hp \ldots \hp \hat{1})$  occurred as a factor in a product-sum decomposition of  $s'$  resp.  $t'$ or their subterms, it must have occurred as the last factor of this product, 
since otherwise we could have applied right-distributivity (rule $\rho_8)$, 
but  $s'$  and  $t'$  are in  $\cR$-normal form. 
In this case we delete the factor  $\hat{1}$ in these products by applying rule  $\rho_4$.
Hence we obtain 
$s'', t'' \in \rmPTid(X)$ such that
$s' \Rightarrow^*_{\cR_{\mrid}/\rmAC_+} s''$ and
$t' \Rightarrow^*_{\cR_{\mrid}/\rmAC_+} t''$.
These applications of rule  $\rho_9$ and possibly rule $\rho_4$ do not change the structure of the parenthesizing, so we cannot apply rules  $\rho_7$ or $\rho_8$. Hence  $s''$  and  $t''$  are in $\cR$-normal form.

Then, we have
$s'' =_{\rmAC_\mrid} s' =_{E_\mrd} s =_{E_\midrd} u =_{E_\midrd} t =_{E_\mrd} t' =_{\rmAC_\mrid} t''$, so $s'' =_{E_\midrd} t''$.
By Theorem \ref{thm:three-implications}(c), we obtain  $s'' =_\rmAC t''$.
Since  $s''$ and $t''$  are in $\cR$-normal form, we obtain $s'' =_{\rmAC_+} t''$.
\end{proof}

As an immediate consequence, we obtain the following normal form result for representing arbitrary terms by id-reduced polynomial terms.

\begin{theorem}\label{thm:idempotent-polynomial-normal-form} For each term  $t \in \T(X)$, there is an $\rmAC_+$-unique id-reduced polynomial term  $t' \in \rmPTid(X)$  in  $\cR_{\mrid}/\rmAC_+$-normal form such that  $t \Rightarrow^*_{\cR_{\mrid}/\rmAC_+} t'$, in particular  $t =_{E_\midrd} t'$.
\end{theorem}

\begin{proof}
Immediate by Lemmas \ref{lm:idempotent-termination} and \ref{lm:idempotent-confluence}.
\end{proof}



\section{An application in weighted automata theory}\label{sect:wlc-strong-bimonoids}

As noted in Section \ref{sect:introduction}, weighted (tree) automata  assign to each input
(i.e., a word or a term over a ranked alphabet) a value from some given weight structure. A weighted (tree) automaton is called \emph{finite-valued}, if the set of all values assigned is finite. From the beginning of the theory of weighted automata,
it has been an essential question to find conditions for determining whether weighted automata over particular weight structures are finite-valued. It is easy to see that if the weight structure is a strong bimonoid  $\sfB$  which is locally finite (i.e., each finitely generated strong subbimonoid is finite), then each weighted (tree) automaton over  $\sfB$  is finite-valued (cf., e.g., \cite{fulvog22}). 
For more detailed investigations, cf. e.g. \cite{drostuvog10}, \cite{drofulkosvog21}, and \cite{fulvog22}. Among others, the notion of "weakly locally finite" strong bimonoid was introduced (precise definition is given below).  
It was shown that each weighted  automaton over a weakly locally finite strong bimonoid is finite-valued 
\cite[Lm.~18]{drostuvog10}. 
However, the situation changes for weighted tree automata. In fact, by Theorem \cite[Thm.~5.1]{drofultepvog24}, for each ranked alphabet which contains at least one binary symbol and for each finitely generated strong bimonoid, we can construct a weighted tree automaton over that ranked alphabet which takes on all elements of the given strong bimonoid as values. Hence, if the strong bimonoid is finitely generated and infinite, then this weighted tree automaton is infinite-valued even if the strong bimonoid is weakly locally finite. Therefore, it is important to know if there are strong bimonoids which are weakly locally finite but not locally finite. The question was answered positively in \cite[Thm.~3.5]{drofultepvog24}, where a right-distributive and weakly locally finite but not locally finite strong bimonoid was constructed.

The main goal of this section is to sharpen this result as follows:
we will construct an idempotent right-distributive strong bimonoid which is weakly locally finite but not locally finite. 

We start with giving the relevant definitions. Let $\Sigma$ be a signature and $\sfA=(A,\theta)$ be a $\Sigma$-algebra. For $A' \subseteq A$, we denote by $\langle A' \rangle_{\theta(\Sigma)}$ the smallest subset of $A$ which contains $A'$ and is closed under the operations in $\theta(\Sigma)$. 
The $\Sigma$-algebra $\sfA=(A,\theta)$ is \emph{locally finite} if, for each finite subset $A' \subseteq A$, the set $\langle A' \rangle_{\theta(\Sigma)}$ is finite.

Let $\B=(B,\oplus,\otimes,\0,\1)$ be a strong bimonoid and $A\subseteq B$. The \emph{weak closure of $A$ (with respect to $\B$)}, denoted by $\wcl(A)$,  is the smallest subset $C\subseteq B$ such that $A\cup\{\0,\1\}\subseteq C$ and, for every  $b,b'\in C$ and $a \in A$, we have $b\oplus b' \in C$ and $b\otimes a \in C$. 

We call the strong bimonoid $\B$
\begin{compactitem}
  \item \emph{additively locally finite} if  $(B,\oplus,\0)$ is locally finite,
  \item \emph{multiplicatively locally finite} if  $(B,\otimes,\1)$ is locally finite,
  \item \emph{bi-locally finite} if it is additively and multiplicatively locally finite, and
  \item \emph{weakly locally finite} if, for each finite subset $A \subseteq B$, the weak closure of $A$ is finite.
  \end{compactitem}
  
 The following  implications between the above properties of strong bimonoids immediately follow from the corresponding definitions:
 \[\text{locally finite}\ \Rightarrow \ \text{weakly locally finite}\ \Rightarrow \ \text{bi-locally finite}.\]
  
 We will also use the following result.
 
\begin{lemma}\label{obs:biloc-fin+right-disrt-loc-fin} \rm  \cite[Rem.~17]{drostuvog10}  Let $\B$ be a right distributive strong bimonoid.  
 Then $\B$ is bi-locally finite if and only if  $\B$ is weakly locally finite.
\end{lemma}

We just note in passing that without the assumption of right-distributivity, 
there are examples of strong bimonoids which are bi-locally finite but not weakly locally finite, 
see \cite[Ex.~2.1(2)]{drovog12} and \cite[Ex.~25]{drostuvog10} (cf. \cite[Ex.~2.6.10(2),(9)]{fulvog22}).

In the rest  of this section we will prove the following result (for the proof cf. Theorem~\ref{thm:M-is-what-we-want}).

\begin{theorem}\label{thm:existence-theorem}  There exists an idempotent right-distributive strong bimonoid  $\M$
which is weakly locally finite but not locally finite.
\end{theorem}

The underlying idea is indicated by the following lemma.

\begin{lemma}\label{lm:weakly locally finite strong bimonoid} \rm Let  $\sfB = (B,+,\times,0,1)$  be an idempotent right-distributive strong bimonoid, and let  $d \in B$. Let  $\sim$  be a congruence relation on  $\sfB$ such that for all $a,b,c \in \sfB \setminus \{ 0, 1\}$, we have
$a \times (b \times c) \sim d \times (d \times d)$. 
Then  $\sfB/{\sim}$ is an idempotent right-distributive weakly locally finite strong bimonoid.    
\end{lemma}

\begin{proof} We write  $\sfB/{\sim} = (B/{\sim},\oplus,\otimes,\0,\1)$.
Clearly,  $\sfB/{\sim}$  is an idempotent, hence additively locally finite, right-distributive strong bimonoid.  We show that it is also multiplicatively locally finite as follows.
If  $F$  is a finite subset of $\sfB/{\sim}$,
then the multiplicative submonoid  $\langle F\rangle_{\{\otimes\}}$  of $\sfB/{\sim}$  generated by  $F$ contains $F \cup \{ \1\}$,
all binary products of elements of  $F$  and, possibly, $[d]_\sim\otimes ([d]_\sim \otimes [d]_\sim)$.
Thus $\langle F\rangle_{\{\otimes\}}$ is finite.

Hence, $\sfB/{\sim}$  is bi-locally finite. Since  $\sfB/{\sim}$  is right-distributive, by Lemma \ref{obs:biloc-fin+right-disrt-loc-fin}, it  is weakly locally finite.    
\end{proof}

The problem is to construct a strong bimonoid  $\sfB$ and a congruence  $\sim$ satisfying the conditions of Lemma \ref{lm:weakly locally finite strong bimonoid}, but so that the strong bimonoid  $\sfB/{\sim}$  is not locally finite.

For this, we will exploit the idempotent right-distributive strong bimonoid  $\sfBidrd[X]$. 
Clearly,  $\sfBidrd[X]$  is not multiplicatively locally finite, since for each  $x \in X$,
the multiplicative submonoid generated by  $\{[x]_\rmAC\}$  is infinite.
Hence our goal is to choose a suitable congruence  $\sim$ on  $\sfBidrd[X]$ so that the strong bimonoid  $\sfM(X) = \sfBidrd[X]/{\sim}$  satisfies the conditions of Lemma \ref{lm:weakly locally finite strong bimonoid}, without being locally finite.

We will need the notion of subpolynomial. 
Intuitively, an id-reduced polynomial $p$  is a subpolynomial of an id-reduced polynomial $q$, if  $q$  can be obtained from  $p$
and arbitrary id-reduced polynomials resp. monomials by the constructions described in Observation \ref{obs:id-reduced-polynomials}. The exact definition is the following.
\begin{definition}\label{def:subpolynomial}\rm Let  $p, q \in \bbBidrd[X]$.
We say that  $p$  is a \emph{subpolynomial} of  $q$, 
if whenever  $U$  is a subset of  $\sfBidrd[X]$  satisfying that    
\begin{compactitem}
\item[--] $p \in U$,
\item[--] if  $r \in U$, $r' \in \bbBidrd[X]$ and  $m$  is a monomial, then  $r +_{\mrid} r' \in U$  and 
$m \times r \in U$, and
\item[--] if  $m \in U$  is a monomial and  $r \in \bbBidrd[X]$, then  $m \times r \in U$,
\end{compactitem}
then  $q \in U$. \hfill$\Box$
\end{definition}

Equivalently, we have  $p = [s]_\rmAC$ and  $q = [t]_\rmAC$
for some  $s,t \in \rmPTid(X)$  such that  $s$  is a subterm of  $t$.

We define the collection of  large polynomials in  $\bbBidrd[X]$ as follows. Intuitively, all products $ p \times_{\mrd} (q \times_{\mrd} r)$ of id-reduced polynomials
$p, q, r \in \bbBidrd[X]\setminus \{ 0, 1\}$ are large, and any id-reduced polynomial obtained from a large polynomial by adding or multiplying it with any further 
id-reduced polynomials should also be large. As formal definition we take the following.

\begin{definition}\label{def:large-polynomial-ALT} \rm A polynomial $q \in \bbBidrd[X] \setminus \{ 0, 1 \}$  is \emph{large},
if  $q$  has a subpolynomial  $p$  of the form   
\[p = [x]_\rmAC \times (([y_1]_\rmAC \times p') + \ldots + ([y_n]_\rmAC \times p')),\] 
where  $n \geq 1$, $x,y_1 \in X$, $y_i \in X \cup \{ \hat{1} \}$  for each  $i \in [n]$, 
the elements  $y_1,\ldots, y_n$ are pairwise different, and  
$p'\in \bbBidrd[X] \setminus \{0, 1\}$.\hfill$\Box$
\end{definition}

Subsequently, we will use the following implication several times:
If  $q \in \bbBidrd[X]$  is large,  $r \in \bbBidrd[X]$  and  $m$  is a monomial, 
then  $q +_{\mrid} r = r +_{\mrid} q$  and  $m \times q$  are also large. 
This is immediate, since each subpolynomial of  $q$  
is also a subpolynomial of  $q +_{\mrid} r$  and of  $m \times q$.

\begin{lemma}\label{lm:p-q-r-large}\rm Let  $p, q, r \in \bbBidrd[X] \setminus \{ 0, 1\}$.
Then  $p \times_{\mrd} (q \times_{\mrd} r)$  is large. 
\end{lemma}

\begin{proof} We put  $u = p \times_{\mrd} (q \times_{\mrd} r)$. We proceed by induction on  $\size(p) + \size(q)$.
First, assume that  $p = [x]_\rmAC$  with  $x \in X$. 
Then  $u = p \times (q \times_{\mrd} r)$.

If also  $q \in X/{\rmAC}$, say  $q = [y]_\rmAC$  with  $y \in X$, 
then  $u = [x]_\rmAC \times ([y]_\rmAC \times r)$
has the required form showing that $u$ is large.

Now, let  $q$  be a sum polynomial. Write
$q = q_1 + \ldots + q_n$  where $n \geq 2$ and each  $q_i$  is a product polynomial
or an element of $X/{\rmAC} \cup \{ 1 \}$, and the elements  $q_1, \ldots, q_n$  are pairwise different. 
Recall that by Corollary \ref{cor:cancellation-in-Pol}(b), the elements
$q_1 \times_{\mrd} r, \ldots, q_n \times_{\mrd} r$  are pairwise different.
Then
$u = p \times (q_1 \times_{\mrd} r + \ldots + q_n \times_{\mrd} r)$. 

First assume that  $q_i \in X/{\rmAC} \cup \{ 1 \}$  for each  $i \in [n]$. 
Since  $q$  is id-reduced, we have  $q_i = 1$  for at most one  $i \in [n]$. 
Since  $n \geq 2$, we may assume that  $q_1\in X/{\rmAC}$.
Then  $u = p \times (q_1 \times r + \ldots + q_n \times r)$, 
showing that  $u$  has the form described in Definition \ref{def:large-polynomial-ALT}.

Second, assume that there exists $i \in [n]$ such that $q_i$  is a product polynomial.
Then $q_i = m_i \times q'_i$ for a monomial  $m_i$  and a polynomial  $q'_i \neq 1$.
Since $\size(m_i) + \size(q'_i) < \size(q) < \size(p) + \size(q)$,
by our induction hypothesis the product
$q_i \times_{\mrd} r = m_i \times_{\mrd} (q'_i \times_{\mrd} r)$
is large. Then, as noted above,  
$q_1 \times_{\mrd} r + \ldots + q_n \times_{\mrd} r$  is also large, hence also  $u$. 

Next, let  $q$  be a product polynomial. Write
$q = m \times q'$  with a monomial  $m$  and a polynomial $q'$. 
Then  $u = p \times (m \times (q' \times_{\mrd} r))$. 
Since $\size(m) + \size(q') < \size(q) < \size(p) + \size(q)$,
by our induction hypothesis the product  
$m \times (q' \times_{\mrd} r) = m \times_{\mrd} (q' \times_{\mrd} r) $  is large.
Again, as noted above, then also  $u$  is large. 

Secondly, assume that  $p$  is a sum polynomial. Write
$p = p_1 + \ldots + p_n$  where  $n \geq 2$, and each  $p_i$  is a product polynomial
or an element of $X/{\rmAC} \cup \{ 1 \}$, and the elements $p_1, \ldots, p_n$  are pairwise different. 
Since  $n \geq 2$, there is  $j \in [n]$  with  $p_j \neq 1$. We have
$u = p_1 \times_{\mrd} (q \times_{\mrd} r) + \ldots + p_n \times_{\mrd} (q \times_{\mrd} r)$.
Since $\size(p_j) + \size(q) < \size(p) + \size(q)$, 
by our induction hypothesis the product
$p_j \times_{\mrd} (q \times_{\mrd} r)$  is large, consequently also  $u$. 

Finally, let  $p$  be a product polynomial. Write
$p = m \times p'$  with a monomial  $m$  and a polynomial $p'$. Then
$u = m \times (p' \times_{\mrd} (q \times_{\mrd} r))$.
Since $\size(p') + \size(q) < \size(p) + \size(q)$, 
by the induction hypothesis the product  $p' \times_{\mrd} (q \times_{\mrd} r)$
is large, consequently again also  $u$. 
\end{proof}

Next we show that the collection of large polynomials is closed under the addition of 
and multiplication with arbitrary id-reduced polynomials.

\begin{lemma}\label{lm:p-generatum-large}\rm Let  $p, q \in \bbBidrd[X] \setminus \{ 0, 1 \}$.
If  $p$  is large, then  $p +_{\mrid} q$, $p \times_{\mrd} q$  and  $q \times_{\mrd} p$  are also large. 
\end{lemma}

\begin{proof} Let  $p$  be large. The result for  $p +_{\mrid} q$  was already noted before.

Now we consider  $p \times_{\mrd} q$. We proceed by case distinction.

First, let  $p$  be a monomial. Then   $p \times_{\mrd} q = p \times q$.
Each subpolynomial of  $p$  is also a subpolynomial of $p \times q$.
Hence  $p \times q$  is large. 

Secondly, let  $p$  be a product polynomial. Then $p = m \times p'$  for a monomial  $m$
and a polynomial  $p' \neq 1$. Then by Lemma \ref{lm:p-q-r-large},
$p \times_{\mrd} q = m \times_{\mrd} (p' \times_{\mrd} q)$  is large.

Thirdly, assume that  $p$  is a sum polynomial. Write
$p = p_1 + \ldots + p_n$  where  $n \geq 2$ and each  $p_i$  is a product polynomial
or an element of $X/{\rmAC} \cup \{ 1 \}$, and the elements $p_1,\ldots, p_n$  are pairwise different.
Clearly, we cannot have  $p_i \in X/{\rmAC} \cup \{ 1 \}$  for each  $i \in [n]$,
since then  $p$  would not be large.
Hence there is  $i \in [n]$ such that  $p_i$  is a product polynomial.
Then, as shown above,  $p_i \times_{\mrd} q$  is large by Lemma \ref{lm:p-q-r-large}.
As already seen, then the sum polynomial
$p \times_{\mrd} q = p_1 \times_{\mrd} q + \ldots + p_n \times_{\mrd} q$
is also large.

Finally, we consider  $q \times_{\mrd} p$. We proceed again by case distinction.

If  $q$  is a monomial, we have  $q \times_{\mrd} p = q \times p$
and our assumption on  $p$ implies the result for  $q \times p$.

Also, if  $q$  is a product polynomial, by Lemma \ref{lm:p-q-r-large}, $q \times_{\mrd} p$  is large.

It remains to consider the case that  $q$  is a sum polynomial. Write
$q = q_1 + \ldots + q_n$ where  $n \geq 2$ and each  $q_i$  is a product polynomial
or an element of $X/{\rmAC} \cup \{ 1 \}$, and the elements  $q_1,\ldots, q_n$  are pairwise different. Then
$q \times_{\mrd} p = q_1 \times_{\mrd} p + \ldots + q_n \times_{\mrd} p$.
Now if  $q_1$  is a product polynomial, then by Lemma \ref{lm:p-q-r-large},
$q_1 \times_{\mrd} p$  is large. From this we obtain, as before, that the sum given by
$q \times_{\mrd} p$  is large.
But if  $q_1 \in X/{\rmAC} \cup \{ 1 \}$, then   $q_1 \times_{\mrd} p = q_1 \times p$,
and the assumption on  $p$ implies that  $q_1 \times p$  is large,
consequently again also  $q \times_{\mrd} p$.
\end{proof}

We note in passing that, as a consequence, a polynomial  $\pi \in \sfBidrd[X]$ is large
if and only if  $\pi$  contains a product of the form  
$p \times_{\mrd} (q \times_{\mrd} r)$, for some  $p, q, r \in \bbBidrd[X] \setminus \{ 0, 1\}$, as a subpolynomial.
Here the "if"  part is immediate by Lemma \ref{lm:p-q-r-large}. 
The converse is also immediate noting
that the subpolynomial described in Definition \ref{def:large-polynomial-ALT} can be written as
$[x]_\rmAC \times_\mrd (([y_1]_\rmAC  + \ldots + [y_n]_\rmAC) \times_\mrd p'))$. 
This was our goal mentioned before Definition \ref{def:large-polynomial-ALT}.
However, this characterization will not be used subsequently.

Now we define a binary relation $\sim_L$  on  $\bbBidrd[X]$  as follows: for every $q, r \in \bbBidrd[X]$, we let

$q \sim_L r$  if and only if  $q = r$  or  both  $q$  and  $r$ are large.

Clearly, by Lemma \ref{lm:p-generatum-large}, if  $p, q, r \in \bbBidrd[X]$  and  $q \sim_L r$, then also
$p +_{\mrid} q \sim_L p +_{\mrid} r$, $p \times_{\mrd} q \sim_L p \times_{\mrd} r$, and  $q \times_{\mrd} p \sim_L r \times_{\mrd} p$.
Hence  $\sim_L$  is a congruence on   $\sfBidrd[X]$.

The quotient algebra of $\sfBidrd[X]$ with respect to $\sim_L$ is
\[ \sfBidrd[X]/{\sim_L}=(\bbBidrd[X]/{\sim_L}, +_{\mrid}/{\sim_L}, \times_{\mrd}/{\sim_L}, [0]_{\sim_L}, [1]_{\sim_L}) \enspace.\]
The algebra $\sfBidrd[X]/{\sim_L}$ is an idempotent,  right-distributive strong bimonoid because it is a factor algebra of the strong bimonoid $\sfBidrd[X]$.

In the following we abbreviate $\sfBidrd[X]/{\sim_L}$ by $\sfM(X)$, and abbreviate also the components of $\sfBidrd[X]/{\sim_L}$ by writing
$\sfM(X) =(\rmM(X),\oplus,\otimes,\0,\1)$. Moreover, for  each $q\in \bbBidrd[X]$, we abbreviate $[q]_{\sim_L}$ by $[q]_L$.

With the following result, we also obtain Theorem \ref{thm:existence-theorem}.

\begin{theorem}\label{thm:M-is-what-we-want} The strong bimonoid $\sfM(X) =(\rmM(X),\oplus,\otimes,\0,\1)$ is idempotent,  right-distributive, weakly locally finite and not locally finite.
\end{theorem}

\begin{proof} We show that  $\sfM(X)$  is weakly locally finite. 
Choose any  $\pi \in \bbBidrd[X] \setminus \{ 0, 1\}$, and let  $p, q, r \in \bbBidrd[X] \setminus \{ 0, 1\}$. By Lemma \ref{lm:p-q-r-large}, the products  $p \times_{\mrd} (q \times_{\mrd} r)$  and
$\pi \times_{\mrd} (\pi \times_{\mrd} \pi)$  are both large, hence
$p \times_{\mrd} (q \times_{\mrd} r) \sim_L \pi \times_{\mrd} (\pi \times_{\mrd} \pi)$.
Now Lemma \ref{lm:weakly locally finite strong bimonoid}  shows that  
$\sfM(X) = \sfBidrd[X]/{\sim_L}$  is weakly locally finite, as claimed.

It remains to show that  $\sfM(X)$ is not locally finite. Choose an $x \in X$.

We define, for each  $n \in \mathbb{N}$,  the polynomials  $p_n \in \rmPT(X)$  inductively by letting
\begin{align*}
\text{$p_0 = [x]_\rmAC$  and  
$p_{n+1} = [x]_\rmAC \times (1 + p_n) = [x]_{\rmAC} \times_{\mrd} (1 + p_n)$.}
\end{align*}

So, e.g.,  $p_1 = [x]_\rmAC \times (1 + [x]_\rmAC)$  and 
$p_2 = [x]_\rmAC \times (1 + p_1) =  [x]_\rmAC \times (1 + ([x]_\rmAC \times (1 + [x]_\rmAC)))$.
Then $p_n \in \bbBidrd[X]$ for each $n\in \mathbb{N}$.

Moreover, for each $n \in \mathbb{N}$,  $p_n$  is a product polynomial and,
in particular, we have  $1 + p_n = 1 +_{\mrid} p_n$. 

Now we show that, for each  $n \in \mathbb{N}$,  $p_n$  is not large. 
Clearly, $p_0$  and  $p_1$  are not large. 
Now let  $n \geq 1$ such that  $p_n$  is not large, but suppose that  $p_{n+1}$  is large.
Then  $p_{n+1}$  has a subpolynomial  $p$  of the form described in Definition \ref{def:large-polynomial-ALT}.
Since  $p_{n+1} = 1 + p_n$  is a sum polynomial but  $p$ is a product polynomial,
we have  $p_{n+1} \neq p$. Hence there are polynomials  $q_1, q_2 \in \bbBidrd[X]$
such that  $p$  is a subpolynomial of  $q_1$  and  $p_{n+1} =  q_1 +_{\mrid} q_2$.
By Lemma \ref{lm:uniqueness-lemma-polynomials}(b), we obtain  $q_1 = p_n$  and  $q_2 = 1$.
Hence, $p$  is a subpolynomial of  $p_n$, a contradiction.

Alternatively, we may argue for  $p_{n+1}$  and its subpolynomial $p$  as follows.
Choose simple terms  $s, t \in \rmST(X)$  such that  
$p_{n+1} = [t]_\rmAC$, $p = [s]_\rmAC$  and  $s$  is a subterm of  $t$.
Then, due to the form of $p$ and  the product term  $s$, the labeled graph  $\overline{t}$
would contain a vertex labeled with  $\boxplus$ whose children are not  $\overline{1}$.
But considering the construction of $p_{n+1}$, it is easy to see that in  $\overline{t}$,
each vertex labeled with  $\boxplus$ has precisely two children, one of them being $\overline{1}$.

Consequently, for each $n \in \mathbb{N}$, $p_n$  is not large. 

Next, let  $m \in \mathbb{N}$. We claim that  $p_m \neq p_n$ for all  $n \in \mathbb{N}$  with  $m < n$.
We proceed by induction on  $m$. Trivially,  $p_0 \neq p_n$  for each  $n \in \mathbb{N}_+$.
Now let $m, n \in \mathbb{N}$  with  $m+1 < n$. Suppose we had  $p_{m+1} = p_n$.
Then $n \geq 2$  and $1 + p_m = p_{m+1} = p_n = 1 + p_{n-1}$.
Since  $p_m$  and  $p_{n-1}$  are product polynomials,
by Lemma \ref{lm:uniqueness-lemma-polynomials}(b), we obtain  $p_m = p_{n-1}$.
But since  $m < n-1$, our induction hypothesis implies  $p_m \neq p_{n-1}$, a contradiction.

Now for each  $n \in \mathbb{N}$, we let $a_n = [p_n]_{L} \in \rmM(X)$.
Then
$a_{n+1} = [[x]_{\rmAC}]_L \otimes (\1 \oplus  a_n)$  for each  $n \in \mathbb{N}$,
so $a_n \in \langle \{ \1, a_0 \} \rangle_{\{\oplus, \otimes\}}$.

Now if  $m, n \in \mathbb{N}$  with  $m < n$, both  $p_m$  and  $p_n$ are not large and  $p_m \neq p_n$.
Hence, $p_m \not\sim_L p_n$, showing  $a_m \neq a_n$. 

Thus,  $\langle \{ \1, a_0 \} \rangle_{\{\oplus, \otimes\}}$  is infinite,
showing that  $\sfM(X)$  is not locally finite.
\end{proof}




\bibliographystyle{alpha}
\phantomsection
\addcontentsline{toc}{chapter}{Bibliography} 
\bibliography{signature2.bbl}

\end{document}